\documentclass{ws-m3as}

\pdfpageattr{/CropBox [90 80 532 720]} 

\usepackage{pgfplots}
\usepackage{tikz}
\usepackage{amsmath}
\usepackage{siunitx}
\usepackage{booktabs}
\usepackage{subcaption}
\usepackage{multicol}
\usepackage{longtable}
\usepackage{aligned-overset}
\usepackage{pdflscape}

\definecolor{airforceblue}{rgb}{0.36, 0.54, 0.66}			    
\definecolor{aliceblue}{rgb}{0.94, 0.97, 1.0}				    
\definecolor{alizarin}{rgb}{0.82, 0.1, 0.26}				    
\definecolor{almond}{rgb}{0.94, 0.87, 0.8}					    
\definecolor{amaranth}{rgb}{0.9, 0.17, 0.31}				    
\definecolor{amber}{rgb}{1.0, 0.75, 0.0}						
\definecolor{amber(sae/ece)}{rgb}{1.0, 0.49, 0.0}	            
\definecolor{americanrose}{rgb}{1.0, 0.01, 0.24}	            
\definecolor{amethyst}{rgb}{0.6, 0.4, 0.8}	                    
\definecolor{anti-flashwhite}{rgb}{0.95, 0.95, 0.96}	        
\definecolor{antiquebrass}{rgb}{0.8, 0.58, 0.46}	            
\definecolor{antiquefuchsia}{rgb}{0.57, 0.36, 0.51}	            
\definecolor{antiquewhite}{rgb}{0.98, 0.92, 0.84}	            
\definecolor{ao}{rgb}{0.0, 0.0, 1.0}	                        
\definecolor{ao(english)}{rgb}{0.0, 0.5, 0.0}	                
\definecolor{applegreen}{rgb}{0.55, 0.71, 0.0}	                
\definecolor{apricot}{rgb}{0.98, 0.81, 0.69}	                
\definecolor{aqua}{rgb}{0.0, 1.0, 1.0}	                        
\definecolor{aquamarine}{rgb}{0.5, 1.0, 0.83}	                
\definecolor{armygreen}{rgb}{0.29, 0.33, 0.13}	                
\definecolor{arsenic}{rgb}{0.23, 0.27, 0.29}	                
\definecolor{arylideyellow}{rgb}{0.91, 0.84, 0.42}	            
\definecolor{ashgrey}{rgb}{0.7, 0.75, 0.71}	                    
\definecolor{asparagus}{rgb}{0.53, 0.66, 0.42}	                
\definecolor{atomictangerine}{rgb}{1.0, 0.6, 0.4}	            
\definecolor{auburn}{rgb}{0.43, 0.21, 0.1}	                    
\definecolor{aureolin}{rgb}{0.99, 0.93, 0.0}	                
\definecolor{aurometalsaurus}{rgb}{0.43, 0.5, 0.5}	            
\definecolor{awesome}{rgb}{1.0, 0.13, 0.32}	                    
\definecolor{azure(colorwheel)}{rgb}{0.0, 0.5, 1.0}	            
\definecolor{azure(web)(azuremist)}{rgb}{0.94, 1.0, 1.0}	    
\definecolor{babyblue}{rgb}{0.54, 0.81, 0.94}	                
\definecolor{babyblueeyes}{rgb}{0.63, 0.79, 0.95}	            
\definecolor{babypink}{rgb}{0.96, 0.76, 0.76}	                
\definecolor{ballblue}{rgb}{0.13, 0.67, 0.8}	                
\definecolor{bananamania}{rgb}{0.98, 0.91, 0.71}	            
\definecolor{bananayellow}{rgb}{1.0, 0.88, 0.21}	            
\definecolor{battleshipgrey}{rgb}{0.52, 0.52, 0.51}	            
\definecolor{bazaar}{rgb}{0.6, 0.47, 0.48}	                    
\definecolor{beaublue}{rgb}{0.74, 0.83, 0.9}	                
\definecolor{beaver}{rgb}{0.62, 0.51, 0.44}	                    
\definecolor{beige}{rgb}{0.96, 0.96, 0.86}	                    
\definecolor{bisque}{rgb}{1.0, 0.89, 0.77}	                    
\definecolor{bistre}{rgb}{0.24, 0.17, 0.12}	                    
\definecolor{bittersweet}{rgb}{1.0, 0.44, 0.37}	                
\definecolor{black}{rgb}{0.0, 0.0, 0.0}	                        
\definecolor{blanchedalmond}{rgb}{1.0, 0.92, 0.8}	            
\definecolor{bleudefrance}{rgb}{0.19, 0.55, 0.91}	            
\definecolor{blizzardblue}{rgb}{0.67, 0.9, 0.93}	            
\definecolor{blond}{rgb}{0.98, 0.94, 0.75}	                    
\definecolor{blue}{rgb}{0.0, 0.0, 1.0}	                        
\definecolor{blue(munsell)}{rgb}{0.0, 0.5, 0.69}	            
\definecolor{blue(ncs)}{rgb}{0.0, 0.53, 0.74}	                
\definecolor{blue(pigment)}{rgb}{0.2, 0.2, 0.6}	                
\definecolor{blue(ryb)}{rgb}{0.01, 0.28, 1.0}	                
\definecolor{bluebell}{rgb}{0.64, 0.64, 0.82}	                
\definecolor{bluegray}{rgb}{0.4, 0.6, 0.8}	                    
\definecolor{blue-green}{rgb}{0.0, 0.87, 0.87}	                
\definecolor{blue-violet}{rgb}{0.54, 0.17, 0.89}	            
\definecolor{blush}{rgb}{0.87, 0.36, 0.51}	                    
\definecolor{bole}{rgb}{0.47, 0.27, 0.23}	                    
\definecolor{bondiblue}{rgb}{0.0, 0.58, 0.71}	                
\definecolor{bostonuniversityred}{rgb}{0.8, 0.0, 0.0}	        
\definecolor{brandeisblue}{rgb}{0.0, 0.44, 1.0}	                
\definecolor{brass}{rgb}{0.71, 0.65, 0.26}	                    
\definecolor{brickred}{rgb}{0.8, 0.25, 0.33}	                
\definecolor{brightcerulean}{rgb}{0.11, 0.67, 0.84}	            
\definecolor{brightgreen}{rgb}{0.4, 1.0, 0.0}	                
\definecolor{brightlavender}{rgb}{0.75, 0.58, 0.89}	            
\definecolor{brightmaroon}{rgb}{0.76, 0.13, 0.28}	            
\definecolor{brightpink}{rgb}{1.0, 0.0, 0.5}	                
\definecolor{brightturquoise}{rgb}{0.03, 0.91, 0.87}	        
\definecolor{brightube}{rgb}{0.82, 0.62, 0.91}	                
\definecolor{brilliantlavender}{rgb}{0.96, 0.73, 1.0}	        
\definecolor{brilliantrose}{rgb}{1.0, 0.33, 0.64}	            
\definecolor{brinkpink}{rgb}{0.98, 0.38, 0.5}	                
\definecolor{britishracinggreen}{rgb}{0.0, 0.26, 0.15}	        
\definecolor{bronze}{rgb}{0.8, 0.5, 0.2}	                    
\definecolor{brown(traditional)}{rgb}{0.59, 0.29, 0.0}	        
\definecolor{brown(web)}{rgb}{0.65, 0.16, 0.16}	                
\definecolor{bubblegum}{rgb}{0.99, 0.76, 0.8}	                
\definecolor{bubbles}{rgb}{0.91, 1.0, 1.0}	                    
\definecolor{buff}{rgb}{0.94, 0.86, 0.51}	                    
\definecolor{bulgarianrose}{rgb}{0.28, 0.02, 0.03}	            
\definecolor{burgundy}{rgb}{0.5, 0.0, 0.13}	                    
\definecolor{burlywood}{rgb}{0.87, 0.72, 0.53}	                
\definecolor{burntorange}{rgb}{0.8, 0.33, 0.0}	                
\definecolor{burntsienna}{rgb}{0.91, 0.45, 0.32}	            
\definecolor{burntumber}{rgb}{0.54, 0.2, 0.14}	                
\definecolor{byzantine}{rgb}{0.74, 0.2, 0.64}	                
\definecolor{byzantium}{rgb}{0.44, 0.16, 0.39}	                
\definecolor{cadet}{rgb}{0.33, 0.41, 0.47}	                    
\definecolor{cadetblue}{rgb}{0.37, 0.62, 0.63}	                
\definecolor{cadetgrey}{rgb}{0.57, 0.64, 0.69}	                
\definecolor{cadmiumgreen}{rgb}{0.0, 0.42, 0.24}	            
\definecolor{cadmiumorange}{rgb}{0.93, 0.53, 0.18}	            
\definecolor{cadmiumred}{rgb}{0.89, 0.0, 0.13}	                
\definecolor{cadmiumyellow}{rgb}{1.0, 0.96, 0.0}	            
\definecolor{calpolypomonagreen}{rgb}{0.12, 0.3, 0.17}	        
\definecolor{cambridgeblue}{rgb}{0.64, 0.76, 0.68}	            
\definecolor{camel}{rgb}{0.76, 0.6, 0.42}	                    
\definecolor{camouflagegreen}{rgb}{0.47, 0.53, 0.42}	        
\definecolor{canaryyellow}{rgb}{1.0, 0.94, 0.0}	                
\definecolor{candyapplered}{rgb}{1.0, 0.03, 0.0}	            
\definecolor{candypink}{rgb}{0.89, 0.44, 0.48}	                
\definecolor{capri}{rgb}{0.0, 0.75, 1.0}	                    
\definecolor{caputmortuum}{rgb}{0.35, 0.15, 0.13}	            
\definecolor{cardinal}{rgb}{0.77, 0.12, 0.23}	                
\definecolor{caribbeangreen}{rgb}{0.0, 0.8, 0.6}	            
\definecolor{carmine}{rgb}{0.59, 0.0, 0.09}	                    
\definecolor{carminepink}{rgb}{0.92, 0.3, 0.26}	                
\definecolor{carminered}{rgb}{1.0, 0.0, 0.22}	                
\definecolor{carnationpink}{rgb}{1.0, 0.65, 0.79}	            
\definecolor{carnelian}{rgb}{0.7, 0.11, 0.11}	                
\definecolor{carolinablue}{rgb}{0.6, 0.73, 0.89}	            
\definecolor{carrotorange}{rgb}{0.93, 0.57, 0.13}	            
\definecolor{ceil}{rgb}{0.57, 0.63, 0.81}	                    
\definecolor{celadon}{rgb}{0.67, 0.88, 0.69}	                
\definecolor{celestialblue}{rgb}{0.29, 0.59, 0.82}	            
\definecolor{cerise}{rgb}{0.87, 0.19, 0.39}	                    
\definecolor{cerisepink}{rgb}{0.93, 0.23, 0.51}	                
\definecolor{cerulean}{rgb}{0.0, 0.48, 0.65}	                
\definecolor{ceruleanblue}{rgb}{0.16, 0.32, 0.75}	            
\definecolor{chamoisee}{rgb}{0.63, 0.47, 0.35}	                
\definecolor{champagne}{rgb}{0.97, 0.91, 0.81}	                
\definecolor{charcoal}{rgb}{0.21, 0.27, 0.31}	                
\definecolor{chartreuse(traditional)}{rgb}{0.87, 1.0, 0.0}	    
\definecolor{chartreuse(web)}{rgb}{0.5, 1.0, 0.0}	            
\definecolor{cherryblossompink}{rgb}{1.0, 0.72, 0.77}	        
\definecolor{chestnut}{rgb}{0.8, 0.36, 0.36}	                
\definecolor{chocolate(traditional)}{rgb}{0.48, 0.25, 0.0}	    
\definecolor{chocolate(web)}{rgb}{0.82, 0.41, 0.12}	            
\definecolor{chromeyellow}{rgb}{1.0, 0.65, 0.0}	                
\definecolor{cinereous}{rgb}{0.6, 0.51, 0.48}	                
\definecolor{cinnabar}{rgb}{0.89, 0.26, 0.2}	                
\definecolor{cinnamon}{rgb}{0.82, 0.41, 0.12}	                
\definecolor{citrine}{rgb}{0.89, 0.82, 0.04}	                
\definecolor{classicrose}{rgb}{0.98, 0.8, 0.91}	                
\definecolor{cobalt}{rgb}{0.0, 0.28, 0.67}	                    
\definecolor{cocoabrown}{rgb}{0.82, 0.41, 0.12}	                
\definecolor{columbiablue}{rgb}{0.61, 0.87, 1.0}	            
\definecolor{coolblack}{rgb}{0.0, 0.18, 0.39}	                
\definecolor{coolgrey}{rgb}{0.55, 0.57, 0.67}	                
\definecolor{copper}{rgb}{0.72, 0.45, 0.2}	                    
\definecolor{copperrose}{rgb}{0.6, 0.4, 0.4}	                
\definecolor{coquelicot}{rgb}{1.0, 0.22, 0.0}	                
\definecolor{coral}{rgb}{1.0, 0.5, 0.31}	                    
\definecolor{coralpink}{rgb}{0.97, 0.51, 0.47}	                
\definecolor{coralred}{rgb}{1.0, 0.25, 0.25}	                
\definecolor{cordovan}{rgb}{0.54, 0.25, 0.27}	                
\definecolor{corn}{rgb}{0.98, 0.93, 0.36}	                    
\definecolor{cornellred}{rgb}{0.7, 0.11, 0.11}	                
\definecolor{cornflowerblue}{rgb}{0.39, 0.58, 0.93}	            
\definecolor{cornsilk}{rgb}{1.0, 0.97, 0.86}	                
\definecolor{cosmiclatte}{rgb}{1.0, 0.97, 0.91}	                
\definecolor{cottoncandy}{rgb}{1.0, 0.74, 0.85}	                
\definecolor{cream}{rgb}{1.0, 0.99, 0.82}	                    
\definecolor{crimson}{rgb}{0.86, 0.08, 0.24}	                
\definecolor{crimsonglory}{rgb}{0.75, 0.0, 0.2}	                
\definecolor{cyan}{rgb}{0.0, 1.0, 1.0}	                        
\definecolor{cyan(process)}{rgb}{0.0, 0.72, 0.92}	            
\definecolor{daffodil}{rgb}{1.0, 1.0, 0.19}	                    
\definecolor{dandelion}{rgb}{0.94, 0.88, 0.19}	                
\definecolor{darkblue}{rgb}{0.0, 0.0, 0.55}	                    
\definecolor{darkbrown}{rgb}{0.4, 0.26, 0.13}	                
\definecolor{darkbyzantium}{rgb}{0.36, 0.22, 0.33}	            
\definecolor{darkcandyapplered}{rgb}{0.64, 0.0, 0.0}	        
\definecolor{darkcerulean}{rgb}{0.03, 0.27, 0.49}	            
\definecolor{darkchampagne}{rgb}{0.76, 0.7, 0.5}	            
\definecolor{darkchestnut}{rgb}{0.6, 0.41, 0.38}	            
\definecolor{darkcoral}{rgb}{0.8, 0.36, 0.27}	                
\definecolor{darkcyan}{rgb}{0.0, 0.55, 0.55}	                
\definecolor{darkelectricblue}{rgb}{0.33, 0.41, 0.47}	        
\definecolor{darkgoldenrod}{rgb}{0.72, 0.53, 0.04}	            
\definecolor{darkgray}{rgb}{0.66, 0.66, 0.66}	                
\definecolor{darkgreen}{rgb}{0.0, 0.2, 0.13}	                
\definecolor{darkjunglegreen}{rgb}{0.1, 0.14, 0.13}	            
\definecolor{darkkhaki}{rgb}{0.74, 0.72, 0.42}	                
\definecolor{darklava}{rgb}{0.28, 0.24, 0.2}	                
\definecolor{darklavender}{rgb}{0.45, 0.31, 0.59}	            
\definecolor{darkmagenta}{rgb}{0.55, 0.0, 0.55}	                
\definecolor{darkmidnightblue}{rgb}{0.0, 0.2, 0.4}	            
\definecolor{darkolivegreen}{rgb}{0.33, 0.42, 0.18}	            
\definecolor{darkorange}{rgb}{1.0, 0.55, 0.0}	                
\definecolor{darkorchid}{rgb}{0.6, 0.2, 0.8}	                
\definecolor{darkpastelblue}{rgb}{0.47, 0.62, 0.8}	            
\definecolor{darkpastelgreen}{rgb}{0.01, 0.75, 0.24}	        
\definecolor{darkpastelpurple}{rgb}{0.59, 0.44, 0.84}	        
\definecolor{darkpastelred}{rgb}{0.76, 0.23, 0.13}	            
\definecolor{darkpink}{rgb}{0.91, 0.33, 0.5}	                
\definecolor{darkpowderblue}{rgb}{0.0, 0.2, 0.6}	            
\definecolor{darkraspberry}{rgb}{0.53, 0.15, 0.34}	            
\definecolor{darkred}{rgb}{0.55, 0.0, 0.0}	                    
\definecolor{darksalmon}{rgb}{0.91, 0.59, 0.48}	                
\definecolor{darkscarlet}{rgb}{0.34, 0.01, 0.1}	                
\definecolor{darkseagreen}{rgb}{0.56, 0.74, 0.56}	            
\definecolor{darksienna}{rgb}{0.24, 0.08, 0.08}	                
\definecolor{darkslateblue}{rgb}{0.28, 0.24, 0.55}	            
\definecolor{darkslategray}{rgb}{0.18, 0.31, 0.31}	            
\definecolor{darkspringgreen}{rgb}{0.09, 0.45, 0.27}	        
\definecolor{darktan}{rgb}{0.57, 0.51, 0.32}	                
\definecolor{darktangerine}{rgb}{1.0, 0.66, 0.07}	            
\definecolor{darktaupe}{rgb}{0.28, 0.24, 0.2}	                
\definecolor{darkterracotta}{rgb}{0.8, 0.31, 0.36}	            
\definecolor{darkturquoise}{rgb}{0.0, 0.81, 0.82}	            
\definecolor{darkviolet}{rgb}{0.58, 0.0, 0.83}	                
\definecolor{dartmouthgreen}{rgb}{0.05, 0.5, 0.06}	            
\definecolor{davy\'sgrey}{rgb}{0.33, 0.33, 0.33}	            
\definecolor{debianred}{rgb}{0.84, 0.04, 0.33}	                
\definecolor{deepcarmine}{rgb}{0.66, 0.13, 0.24}	            
\definecolor{deepcarminepink}{rgb}{0.94, 0.19, 0.22}	        
\definecolor{deepcarrotorange}{rgb}{0.91, 0.41, 0.17}	        
\definecolor{deepcerise}{rgb}{0.85, 0.2, 0.53}	                
\definecolor{deepchampagne}{rgb}{0.98, 0.84, 0.65}	            
\definecolor{deepchestnut}{rgb}{0.73, 0.31, 0.28}	            
\definecolor{deepfuchsia}{rgb}{0.76, 0.33, 0.76}	            
\definecolor{deepjunglegreen}{rgb}{0.0, 0.29, 0.29}	            
\definecolor{deeplilac}{rgb}{0.6, 0.33, 0.73}	                
\definecolor{deepmagenta}{rgb}{0.8, 0.0, 0.8}	                
\definecolor{deeppeach}{rgb}{1.0, 0.8, 0.64}	                
\definecolor{deeppink}{rgb}{1.0, 0.08, 0.58}	                
\definecolor{deepsaffron}{rgb}{1.0, 0.6, 0.2}	                
\definecolor{deepskyblue}{rgb}{0.0, 0.75, 1.0}	                
\definecolor{denim}{rgb}{0.08, 0.38, 0.74}	                    
\definecolor{desert}{rgb}{0.76, 0.6, 0.42}	                    
\definecolor{desertsand}{rgb}{0.93, 0.79, 0.69}	                
\definecolor{dimgray}{rgb}{0.41, 0.41, 0.41}	                
\definecolor{dodgerblue}{rgb}{0.12, 0.56, 1.0}	                
\definecolor{dogwoodrose}{rgb}{0.84, 0.09, 0.41}	            
\definecolor{dollarbill}{rgb}{0.52, 0.73, 0.4}	                
\definecolor{drab}{rgb}{0.59, 0.44, 0.09}	                    
\definecolor{dukeblue}{rgb}{0.0, 0.0, 0.61}	                    
\definecolor{earthyellow}{rgb}{0.88, 0.66, 0.37}	            
\definecolor{ecru}{rgb}{0.76, 0.7, 0.5}	                        
\definecolor{eggplant}{rgb}{0.38, 0.25, 0.32}	                
\definecolor{eggshell}{rgb}{0.94, 0.92, 0.84}	                
\definecolor{egyptianblue}{rgb}{0.06, 0.2, 0.65}	            
\definecolor{electricblue}{rgb}{0.49, 0.98, 1.0}	            
\definecolor{electriccrimson}{rgb}{1.0, 0.0, 0.25}	            
\definecolor{electriccyan}{rgb}{0.0, 1.0, 1.0}	                
\definecolor{electricgreen}{rgb}{0.0, 1.0, 0.0}	                
\definecolor{electricindigo}{rgb}{0.44, 0.0, 1.0}	            
\definecolor{electriclavender}{rgb}{0.96, 0.73, 1.0}	        
\definecolor{electriclime}{rgb}{0.8, 1.0, 0.0}	                
\definecolor{electricpurple}{rgb}{0.75, 0.0, 1.0}	            
\definecolor{electricultramarine}{rgb}{0.25, 0.0, 1.0}	        
\definecolor{electricviolet}{rgb}{0.56, 0.0, 1.0}	            
\definecolor{electricyellow}{rgb}{1.0, 1.0, 0.0}	            
\definecolor{emerald}{rgb}{0.31, 0.78, 0.47}	                
\definecolor{etonblue}{rgb}{0.59, 0.78, 0.64}	                
\definecolor{fallow}{rgb}{0.76, 0.6, 0.42}	                    
\definecolor{falured}{rgb}{0.5, 0.09, 0.09}	                    
\definecolor{fandango}{rgb}{0.71, 0.2, 0.54}	                
\definecolor{fashionfuchsia}{rgb}{0.96, 0.0, 0.63}	            
\definecolor{fawn}{rgb}{0.9, 0.67, 0.44}	                    
\definecolor{feldgrau}{rgb}{0.3, 0.36, 0.33}	                
\definecolor{ferngreen}{rgb}{0.31, 0.47, 0.26}	                
\definecolor{ferrarired}{rgb}{1.0, 0.11, 0.0}	                
\definecolor{fielddrab}{rgb}{0.42, 0.33, 0.12}	                
\definecolor{firebrick}{rgb}{0.7, 0.13, 0.13}	                
\definecolor{fireenginered}{rgb}{0.81, 0.09, 0.13}	            
\definecolor{flame}{rgb}{0.89, 0.35, 0.13}	                    
\definecolor{flamingopink}{rgb}{0.99, 0.56, 0.67}	            
\definecolor{flavescent}{rgb}{0.97, 0.91, 0.56}	                
\definecolor{flax}{rgb}{0.93, 0.86, 0.51}	                    
\definecolor{floralwhite}{rgb}{1.0, 0.98, 0.94}	                
\definecolor{fluorescentorange}{rgb}{1.0, 0.75, 0.0}	        
\definecolor{fluorescentpink}{rgb}{1.0, 0.08, 0.58}	            
\definecolor{fluorescentyellow}{rgb}{0.8, 1.0, 0.0}	            
\definecolor{folly}{rgb}{1.0, 0.0, 0.31}	                    
\definecolor{forestgreen(traditional)}{rgb}{0.0, 0.27, 0.13}	
\definecolor{forestgreen(web)}{rgb}{0.13, 0.55, 0.13}	        
\definecolor{frenchbeige}{rgb}{0.65, 0.48, 0.36}	            
\definecolor{frenchblue}{rgb}{0.0, 0.45, 0.73}	                
\definecolor{frenchlilac}{rgb}{0.53, 0.38, 0.56}	            
\definecolor{frenchrose}{rgb}{0.96, 0.29, 0.54}	                
\definecolor{fuchsia}{rgb}{1.0, 0.0, 1.0}	                    
\definecolor{fuchsiapink}{rgb}{1.0, 0.47, 1.0}	                
\definecolor{fulvous}{rgb}{0.86, 0.52, 0.0}	                    
\definecolor{fuzzywuzzy}{rgb}{0.8, 0.4, 0.4}	                
\definecolor{gainsboro}{rgb}{0.86, 0.86, 0.86}	                
\definecolor{gamboge}{rgb}{0.89, 0.61, 0.06}	                
\definecolor{ghostwhite}{rgb}{0.97, 0.97, 1.0}	                
\definecolor{ginger}{rgb}{0.69, 0.4, 0.0}	                    
\definecolor{glaucous}{rgb}{0.38, 0.51, 0.71}	                
\definecolor{gold(metallic)}{rgb}{0.83, 0.69, 0.22}	            
\definecolor{gold(web)(golden)}{rgb}{1.0, 0.84, 0.0}	        
\definecolor{goldenbrown}{rgb}{0.6, 0.4, 0.08}	                
\definecolor{goldenpoppy}{rgb}{0.99, 0.76, 0.0}	                
\definecolor{goldenyellow}{rgb}{1.0, 0.87, 0.0}	                
\definecolor{goldenrod}{rgb}{0.85, 0.65, 0.13}	                
\definecolor{grannysmithapple}{rgb}{0.66, 0.89, 0.63}	        
\definecolor{gray}{rgb}{0.5, 0.5, 0.5}	                        
\definecolor{gray(html/cssgray)}{rgb}{0.5, 0.5, 0.5}	        
\definecolor{gray(x11gray)}{rgb}{0.75, 0.75, 0.75}	            
\definecolor{gray-asparagus}{rgb}{0.27, 0.35, 0.27}	            
\definecolor{green(colorwheel)(x11green)}{rgb}{0.0, 1.0, 0.0}	
\definecolor{green(html/cssgreen)}{rgb}{0.0, 0.5, 0.0}	        
\definecolor{green(munsell)}{rgb}{0.0, 0.66, 0.47}	            
\definecolor{green(ncs)}{rgb}{0.0, 0.62, 0.42}	                
\definecolor{green(pigment)}{rgb}{0.0, 0.65, 0.31}	            
\definecolor{green(ryb)}{rgb}{0.4, 0.69, 0.2}	                
\definecolor{green-yellow}{rgb}{0.68, 1.0, 0.18}	            
\definecolor{grullo}{rgb}{0.66, 0.6, 0.53}	                    
\definecolor{guppiegreen}{rgb}{0.0, 1.0, 0.5}	                
\definecolor{halayaube}{rgb}{0.4, 0.22, 0.33}	                
\definecolor{hanblue}{rgb}{0.27, 0.42, 0.81}	                
\definecolor{hanpurple}{rgb}{0.32, 0.09, 0.98}	                
\definecolor{hansayellow}{rgb}{0.91, 0.84, 0.42}	            
\definecolor{harlequin}{rgb}{0.25, 1.0, 0.0}	                
\definecolor{harvardcrimson}{rgb}{0.79, 0.0, 0.09}	            
\definecolor{harvestgold}{rgb}{0.85, 0.57, 0.0}	                
\definecolor{heartgold}{rgb}{0.5, 0.5, 0.0}	                    
\definecolor{heliotrope}{rgb}{0.87, 0.45, 1.0}	                
\definecolor{hollywoodcerise}{rgb}{0.96, 0.0, 0.63}	            
\definecolor{honeydew}{rgb}{0.94, 1.0, 0.94}	                
\definecolor{hooker\'sgreen}{rgb}{0.0, 0.44, 0.0}	            
\definecolor{hotmagenta}{rgb}{1.0, 0.11, 0.81}	                
\definecolor{hotpink}{rgb}{1.0, 0.41, 0.71}	                    
\definecolor{huntergreen}{rgb}{0.21, 0.37, 0.23}	            
\definecolor{iceberg}{rgb}{0.44, 0.65, 0.82}	                
\definecolor{icterine}{rgb}{0.99, 0.97, 0.37}	                
\definecolor{inchworm}{rgb}{0.7, 0.93, 0.36}	                
\definecolor{indiagreen}{rgb}{0.07, 0.53, 0.03}	                
\definecolor{indianred}{rgb}{0.8, 0.36, 0.36}	                
\definecolor{indianyellow}{rgb}{0.89, 0.66, 0.34}	            
\definecolor{indigo(dye)}{rgb}{0.0, 0.25, 0.42}	                
\definecolor{indigo(web)}{rgb}{0.29, 0.0, 0.51}	                
\definecolor{internationalkleinblue}{rgb}{0.0, 0.18, 0.65}	    
\definecolor{internationalorange}{rgb}{1.0, 0.31, 0.0}	        
\definecolor{iris}{rgb}{0.35, 0.31, 0.81}	                    
\definecolor{isabelline}{rgb}{0.96, 0.94, 0.93}	                
\definecolor{islamicgreen}{rgb}{0.0, 0.56, 0.0}	                
\definecolor{ivory}{rgb}{1.0, 1.0, 0.94}	                    
\definecolor{jade}{rgb}{0.0, 0.66, 0.42}	                    
\definecolor{jasper}{rgb}{0.84, 0.23, 0.24}	                    
\definecolor{jazzberryjam}{rgb}{0.65, 0.04, 0.37}	            
\definecolor{jonquil}{rgb}{0.98, 0.85, 0.37}	                
\definecolor{junebud}{rgb}{0.74, 0.85, 0.34}	                
\definecolor{junglegreen}{rgb}{0.16, 0.67, 0.53}	            
\definecolor{kellygreen}{rgb}{0.3, 0.73, 0.09}	                
\definecolor{khaki(html/css)(khaki)}{rgb}{0.76, 0.69, 0.57}	    
\definecolor{khaki(x11)(lightkhaki)}{rgb}{0.94, 0.9, 0.55}	    
\definecolor{lasallegreen}{rgb}{0.03, 0.47, 0.19}	            
\definecolor{languidlavender}{rgb}{0.84, 0.79, 0.87}	        
\definecolor{lapislazuli}{rgb}{0.15, 0.38, 0.61}	            
\definecolor{laserlemon}{rgb}{1.0, 1.0, 0.13}	                
\definecolor{lava}{rgb}{0.81, 0.06, 0.13}	                    
\definecolor{lavender(floral)}{rgb}{0.71, 0.49, 0.86}	        
\definecolor{lavender(web)}{rgb}{0.9, 0.9, 0.98}	            
\definecolor{lavenderblue}{rgb}{0.8, 0.8, 1.0}	                
\definecolor{lavenderblush}{rgb}{1.0, 0.94, 0.96}	            
\definecolor{lavendergray}{rgb}{0.77, 0.76, 0.82}	            
\definecolor{lavenderindigo}{rgb}{0.58, 0.34, 0.92}	            
\definecolor{lavendermagenta}{rgb}{0.93, 0.51, 0.93}	        
\definecolor{lavendermist}{rgb}{0.9, 0.9, 0.98}	                
\definecolor{lavenderpink}{rgb}{0.98, 0.68, 0.82}	            
\definecolor{lavenderpurple}{rgb}{0.59, 0.48, 0.71}	            
\definecolor{lavenderrose}{rgb}{0.98, 0.63, 0.89}	            
\definecolor{lawngreen}{rgb}{0.49, 0.99, 0.0}	                
\definecolor{lemon}{rgb}{1.0, 0.97, 0.0}	                    
\definecolor{lemonchiffon}{rgb}{1.0, 0.98, 0.8}	                
\definecolor{lightapricot}{rgb}{0.99, 0.84, 0.69}	            
\definecolor{lightblue}{rgb}{0.68, 0.85, 0.9}	                
\definecolor{lightbrown}{rgb}{0.71, 0.4, 0.11}	                
\definecolor{lightcarminepink}{rgb}{0.9, 0.4, 0.38}	            
\definecolor{lightcoral}{rgb}{0.94, 0.5, 0.5}	                
\definecolor{lightcornflowerblue}{rgb}{0.6, 0.81, 0.93}	        
\definecolor{lightcyan}{rgb}{0.88, 1.0, 1.0}	                
\definecolor{lightfuchsiapink}{rgb}{0.98, 0.52, 0.9}	        
\definecolor{lightgoldenrodyellow}{rgb}{0.98, 0.98, 0.82}	    
\definecolor{lightgray}{rgb}{0.83, 0.83, 0.83}	                
\definecolor{lightgreen}{rgb}{0.56, 0.93, 0.56}	                
\definecolor{lightkhaki}{rgb}{0.94, 0.9, 0.55}	                
\definecolor{lightmauve}{rgb}{0.86, 0.82, 1.0}	                
\definecolor{lightpastelpurple}{rgb}{0.69, 0.61, 0.85}	        
\definecolor{lightpink}{rgb}{1.0, 0.71, 0.76}	                
\definecolor{lightsalmon}{rgb}{1.0, 0.63, 0.48}	                
\definecolor{lightsalmonpink}{rgb}{1.0, 0.6, 0.6}	            
\definecolor{lightseagreen}{rgb}{0.13, 0.7, 0.67}	            
\definecolor{lightskyblue}{rgb}{0.53, 0.81, 0.98}	            
\definecolor{lightslategray}{rgb}{0.47, 0.53, 0.6}	            
\definecolor{lighttaupe}{rgb}{0.7, 0.55, 0.43}	                
\definecolor{lightthulianpink}{rgb}{0.9, 0.56, 0.67}	        
\definecolor{lightyellow}{rgb}{1.0, 1.0, 0.88}	                
\definecolor{lilac}{rgb}{0.78, 0.64, 0.78}	                    
\definecolor{lime(colorwheel)}{rgb}{0.75, 1.0, 0.0}	            
\definecolor{lime(web)(x11green)}{rgb}{0.0, 1.0, 0.0}	        
\definecolor{limegreen}{rgb}{0.2, 0.8, 0.2}	                    
\definecolor{lincolngreen}{rgb}{0.11, 0.35, 0.02}	            
\definecolor{linen}{rgb}{0.98, 0.94, 0.9}	                    
\definecolor{liver}{rgb}{0.33, 0.29, 0.31}	                    
\definecolor{lust}{rgb}{0.9, 0.13, 0.13}	                    
\definecolor{macaroniandcheese}{rgb}{1.0, 0.74, 0.53}	        
\definecolor{magenta}{rgb}{1.0, 0.0, 1.0}	                    
\definecolor{magenta(dye)}{rgb}{0.79, 0.08, 0.48}	            
\definecolor{magenta(process)}{rgb}{1.0, 0.0, 0.56}	            
\definecolor{magicmint}{rgb}{0.67, 0.94, 0.82}	                
\definecolor{magnolia}{rgb}{0.97, 0.96, 1.0}	                
\definecolor{mahogany}{rgb}{0.75, 0.25, 0.0}	                
\definecolor{maize}{rgb}{0.98, 0.93, 0.37}	                    
\definecolor{majorelleblue}{rgb}{0.38, 0.31, 0.86}	            
\definecolor{malachite}{rgb}{0.04, 0.85, 0.32}	                
\definecolor{manatee}{rgb}{0.59, 0.6, 0.67}	                    
\definecolor{mangotango}{rgb}{1.0, 0.51, 0.26}	                
\definecolor{maroon(html/css)}{rgb}{0.5, 0.0, 0.0}	            
\definecolor{maroon(x11)}{rgb}{0.69, 0.19, 0.38}	            
\definecolor{mauve}{rgb}{0.88, 0.69, 1.0}	                    
\definecolor{mauvetaupe}{rgb}{0.57, 0.37, 0.43}	                
\definecolor{mauvelous}{rgb}{0.94, 0.6, 0.67}	                
\definecolor{mayablue}{rgb}{0.45, 0.76, 0.98}	                
\definecolor{meatbrown}{rgb}{0.9, 0.72, 0.23}	                
\definecolor{mediumaquamarine}{rgb}{0.4, 0.8, 0.67}	            
\definecolor{mediumblue}{rgb}{0.0, 0.0, 0.8}	                
\definecolor{mediumcandyapplered}{rgb}{0.89, 0.02, 0.17}	    
\definecolor{mediumcarmine}{rgb}{0.69, 0.25, 0.21}	            
\definecolor{mediumchampagne}{rgb}{0.95, 0.9, 0.67}	            
\definecolor{mediumelectricblue}{rgb}{0.01, 0.31, 0.59}	        
\definecolor{mediumjunglegreen}{rgb}{0.11, 0.21, 0.18}	        
\definecolor{mediumlavendermagenta}{rgb}{0.8, 0.6, 0.8}	        
\definecolor{mediumorchid}{rgb}{0.73, 0.33, 0.83}	            
\definecolor{mediumpersianblue}{rgb}{0.0, 0.4, 0.65}	        
\definecolor{mediumpurple}{rgb}{0.58, 0.44, 0.86}	            
\definecolor{mediumred-violet}{rgb}{0.73, 0.2, 0.52}	        
\definecolor{mediumseagreen}{rgb}{0.24, 0.7, 0.44}	            
\definecolor{mediumslateblue}{rgb}{0.48, 0.41, 0.93}	        
\definecolor{mediumspringbud}{rgb}{0.79, 0.86, 0.54}	        
\definecolor{mediumspringgreen}{rgb}{0.0, 0.98, 0.6}	        
\definecolor{mediumtaupe}{rgb}{0.4, 0.3, 0.28}	                
\definecolor{mediumtealblue}{rgb}{0.0, 0.33, 0.71}	            
\definecolor{mediumturquoise}{rgb}{0.28, 0.82, 0.8}	            
\definecolor{mediumviolet-red}{rgb}{0.78, 0.08, 0.52}	        
\definecolor{melon}{rgb}{0.99, 0.74, 0.71}	                    
\definecolor{midnightblue}{rgb}{0.1, 0.1, 0.44}	                
\definecolor{midnightgreen(eaglegreen)}{rgb}{0.0, 0.29, 0.33}	
\definecolor{mikadoyellow}{rgb}{1.0, 0.77, 0.05}	            
\definecolor{mint}{rgb}{0.24, 0.71, 0.54}	                    
\definecolor{mintcream}{rgb}{0.96, 1.0, 0.98}	                
\definecolor{mintgreen}{rgb}{0.6, 1.0, 0.6}	                    
\definecolor{mistyrose}{rgb}{1.0, 0.89, 0.88}	                
\definecolor{moccasin}{rgb}{0.98, 0.92, 0.84}	                
\definecolor{modebeige}{rgb}{0.59, 0.44, 0.09}	                
\definecolor{moonstoneblue}{rgb}{0.45, 0.66, 0.76}	            
\definecolor{mordantred19}{rgb}{0.68, 0.05, 0.0}	            
\definecolor{mossgreen}{rgb}{0.68, 0.87, 0.68}	                
\definecolor{mountainmeadow}{rgb}{0.19, 0.73, 0.56}	            
\definecolor{mountbattenpink}{rgb}{0.6, 0.48, 0.55}	            
\definecolor{mulberry}{rgb}{0.77, 0.29, 0.55}	                
\definecolor{mustard}{rgb}{1.0, 0.86, 0.35}	                    
\definecolor{myrtle}{rgb}{0.13, 0.26, 0.12}	                    
\definecolor{msugreen}{rgb}{0.09, 0.27, 0.23}	                
\definecolor{nadeshikopink}{rgb}{0.96, 0.68, 0.78}	            
\definecolor{napiergreen}{rgb}{0.16, 0.5, 0.0}	                
\definecolor{naplesyellow}{rgb}{0.98, 0.85, 0.37}	            
\definecolor{navajowhite}{rgb}{1.0, 0.87, 0.68}	                
\definecolor{navyblue}{rgb}{0.0, 0.0, 0.5}	                    
\definecolor{neoncarrot}{rgb}{1.0, 0.64, 0.26}	                
\definecolor{neonfuchsia}{rgb}{1.0, 0.25, 0.39}	                
\definecolor{neongreen}{rgb}{0.22, 0.88, 0.08}	                
\definecolor{non-photoblue}{rgb}{0.64, 0.87, 0.93}	            
\definecolor{oceanboatblue}{rgb}{0.0, 0.47, 0.75}	            
\definecolor{ochre}{rgb}{0.8, 0.47, 0.13}	                    
\definecolor{officegreen}{rgb}{0.0, 0.5, 0.0}	                
\definecolor{oldgold}{rgb}{0.81, 0.71, 0.23}	                
\definecolor{oldlace}{rgb}{0.99, 0.96, 0.9}	                    
\definecolor{oldlavender}{rgb}{0.47, 0.41, 0.47}	            
\definecolor{oldmauve}{rgb}{0.4, 0.19, 0.28}	                
\definecolor{oldrose}{rgb}{0.75, 0.5, 0.51}	                    
\definecolor{olive}{rgb}{0.5, 0.5, 0.0}	                        
\definecolor{olivedrab(web)(olivedrab3)}{rgb}{0.42, 0.56, 0.14}	
\definecolor{olivedrab7}{rgb}{0.24, 0.2, 0.12}	                
\definecolor{olivine}{rgb}{0.6, 0.73, 0.45}	                    
\definecolor{onyx}{rgb}{0.06, 0.06, 0.06}	                    
\definecolor{operamauve}{rgb}{0.72, 0.52, 0.65}	                
\definecolor{orange(colorwheel)}{rgb}{1.0, 0.5, 0.0}	        
\definecolor{orange(ryb)}{rgb}{0.98, 0.6, 0.01}	                
\definecolor{orange(webcolor)}{rgb}{1.0, 0.65, 0.0}	            
\definecolor{orangepeel}{rgb}{1.0, 0.62, 0.0}	                
\definecolor{orange-red}{rgb}{1.0, 0.27, 0.0}	                
\definecolor{orchid}{rgb}{0.85, 0.44, 0.84}	                    
\definecolor{otterbrown}{rgb}{0.4, 0.26, 0.13}	                
\definecolor{outerspace}{rgb}{0.25, 0.29, 0.3}	                
\definecolor{outrageousorange}{rgb}{1.0, 0.43, 0.29}	        
\definecolor{oxfordblue}{rgb}{0.0, 0.13, 0.28}	                
\definecolor{oucrimsonred}{rgb}{0.6, 0.0, 0.0}	                
\definecolor{pakistangreen}{rgb}{0.0, 0.4, 0.0}	                
\definecolor{palatinateblue}{rgb}{0.15, 0.23, 0.89}	            
\definecolor{palatinatepurple}{rgb}{0.41, 0.16, 0.38}	        
\definecolor{paleaqua}{rgb}{0.74, 0.83, 0.9}	                
\definecolor{paleblue}{rgb}{0.69, 0.93, 0.93}	                
\definecolor{palebrown}{rgb}{0.6, 0.46, 0.33}	                
\definecolor{palecarmine}{rgb}{0.69, 0.25, 0.21}	            
\definecolor{palecerulean}{rgb}{0.61, 0.77, 0.89}	            
\definecolor{palechestnut}{rgb}{0.87, 0.68, 0.69}	            
\definecolor{palecopper}{rgb}{0.85, 0.54, 0.4}	                
\definecolor{palecornflowerblue}{rgb}{0.67, 0.8, 0.94}	        
\definecolor{palegold}{rgb}{0.9, 0.75, 0.54}	                
\definecolor{palegoldenrod}{rgb}{0.93, 0.91, 0.67}	            
\definecolor{palegreen}{rgb}{0.6, 0.98, 0.6}	                
\definecolor{palemagenta}{rgb}{0.98, 0.52, 0.9}	                
\definecolor{palepink}{rgb}{0.98, 0.85, 0.87}	                
\definecolor{paleplum}{rgb}{0.8, 0.6, 0.8}	                    
\definecolor{palered-violet}{rgb}{0.86, 0.44, 0.58}	            
\definecolor{palerobineggblue}{rgb}{0.59, 0.87, 0.82}	        
\definecolor{palesilver}{rgb}{0.79, 0.75, 0.73}	                
\definecolor{palespringbud}{rgb}{0.93, 0.92, 0.74}	            
\definecolor{paletaupe}{rgb}{0.74, 0.6, 0.49}	                
\definecolor{paleviolet-red}{rgb}{0.86, 0.44, 0.58}	            
\definecolor{pansypurple}{rgb}{0.47, 0.09, 0.29}	            
\definecolor{papayawhip}{rgb}{1.0, 0.94, 0.84}	                
\definecolor{parisgreen}{rgb}{0.31, 0.78, 0.47}	                
\definecolor{pastelblue}{rgb}{0.68, 0.78, 0.81}	                
\definecolor{pastelbrown}{rgb}{0.51, 0.41, 0.33}	            
\definecolor{pastelgray}{rgb}{0.81, 0.81, 0.77}	                
\definecolor{pastelgreen}{rgb}{0.47, 0.87, 0.47}	            
\definecolor{pastelmagenta}{rgb}{0.96, 0.6, 0.76}	            
\definecolor{pastelorange}{rgb}{1.0, 0.7, 0.28}	                
\definecolor{pastelpink}{rgb}{1.0, 0.82, 0.86}	                
\definecolor{pastelpurple}{rgb}{0.7, 0.62, 0.71}	            
\definecolor{pastelred}{rgb}{1.0, 0.41, 0.38}	                
\definecolor{pastelviolet}{rgb}{0.8, 0.6, 0.79}	                
\definecolor{pastelyellow}{rgb}{0.99, 0.99, 0.59}	            
\definecolor{patriarch}{rgb}{0.5, 0.0, 0.5}	                    
\definecolor{payne\'sgrey}{rgb}{0.25, 0.25, 0.28}	            
\definecolor{peach}{rgb}{1.0, 0.9, 0.71}	                    
\definecolor{peach-orange}{rgb}{1.0, 0.8, 0.6}	                
\definecolor{peachpuff}{rgb}{1.0, 0.85, 0.73}	                
\definecolor{peach-yellow}{rgb}{0.98, 0.87, 0.68}	            
\definecolor{pear}{rgb}{0.82, 0.89, 0.19}	                    
\definecolor{pearl}{rgb}{0.94, 0.92, 0.84}	                    
\definecolor{peridot}{rgb}{0.9, 0.89, 0.0}	                    
\definecolor{periwinkle}{rgb}{0.8, 0.8, 1.0}	                
\definecolor{persianblue}{rgb}{0.11, 0.22, 0.73}	            
\definecolor{persiangreen}{rgb}{0.0, 0.65, 0.58}	            
\definecolor{persianindigo}{rgb}{0.2, 0.07, 0.48}	            
\definecolor{persianorange}{rgb}{0.85, 0.56, 0.35}	            
\definecolor{peru}{rgb}{0.8, 0.52, 0.25}	                    
\definecolor{persianpink}{rgb}{0.97, 0.5, 0.75}	                
\definecolor{persianplum}{rgb}{0.44, 0.11, 0.11}	            
\definecolor{persianred}{rgb}{0.8, 0.2, 0.2}	                
\definecolor{persianrose}{rgb}{1.0, 0.16, 0.64}	                
\definecolor{persimmon}{rgb}{0.93, 0.35, 0.0}	                
\definecolor{phlox}{rgb}{0.87, 0.0, 1.0}	                    
\definecolor{phthaloblue}{rgb}{0.0, 0.06, 0.54}	                
\definecolor{phthalogreen}{rgb}{0.07, 0.21, 0.14}	            
\definecolor{piggypink}{rgb}{0.99, 0.87, 0.9}	                
\definecolor{pinegreen}{rgb}{0.0, 0.47, 0.44}	                
\definecolor{pink}{rgb}{1.0, 0.75, 0.8}	                        
\definecolor{pink-orange}{rgb}{1.0, 0.6, 0.4}	                
\definecolor{pinkpearl}{rgb}{0.91, 0.67, 0.81}	                
\definecolor{pinksherbet}{rgb}{0.97, 0.56, 0.65}	            
\definecolor{pistachio}{rgb}{0.58, 0.77, 0.45}	                
\definecolor{platinum}{rgb}{0.9, 0.89, 0.89}	                
\definecolor{plum(traditional)}{rgb}{0.56, 0.27, 0.52}	        
\definecolor{plum(web)}{rgb}{0.8, 0.6, 0.8}	                    
\definecolor{portlandorange}{rgb}{1.0, 0.35, 0.21}	            
\definecolor{powderblue(web)}{rgb}{0.69, 0.88, 0.9}	            
\definecolor{princetonorange}{rgb}{1.0, 0.56, 0.0}	            
\definecolor{prune}{rgb}{0.44, 0.11, 0.11}	                    
\definecolor{prussianblue}{rgb}{0.0, 0.19, 0.33}	            
\definecolor{psychedelicpurple}{rgb}{0.87, 0.0, 1.0}	        
\definecolor{puce}{rgb}{0.8, 0.53, 0.6}	                        
\definecolor{pumpkin}{rgb}{1.0, 0.46, 0.09}	                    
\definecolor{purple(html/css)}{rgb}{0.5, 0.0, 0.5}	            
\definecolor{purple(munsell)}{rgb}{0.62, 0.0, 0.77}	            
\definecolor{purple(x11)}{rgb}{0.63, 0.36, 0.94}	            
\definecolor{purpleheart}{rgb}{0.41, 0.21, 0.61}	            
\definecolor{purplemountainmajesty}{rgb}{0.59, 0.47, 0.71}	    
\definecolor{purplepizzazz}{rgb}{1.0, 0.31, 0.85}	            
\definecolor{purpletaupe}{rgb}{0.31, 0.25, 0.3}	                
\definecolor{radicalred}{rgb}{1.0, 0.21, 0.37}	                
\definecolor{raspberry}{rgb}{0.89, 0.04, 0.36}	                
\definecolor{raspberryglace}{rgb}{0.57, 0.37, 0.43}	            
\definecolor{raspberrypink}{rgb}{0.89, 0.31, 0.61}	            
\definecolor{raspberryrose}{rgb}{0.7, 0.27, 0.42}	            
\definecolor{rawumber}{rgb}{0.51, 0.4, 0.27}	                
\definecolor{razzledazzlerose}{rgb}{1.0, 0.2, 0.8}	            
\definecolor{razzmatazz}{rgb}{0.89, 0.15, 0.42}	                
\definecolor{red}{rgb}{1.0, 0.0, 0.0}	                        
\definecolor{red(munsell)}{rgb}{0.95, 0.0, 0.24}	            
\definecolor{red(ncs)}{rgb}{0.77, 0.01, 0.2}	                
\definecolor{red(pigment)}{rgb}{0.93, 0.11, 0.14}	            
\definecolor{red(ryb)}{rgb}{1.0, 0.15, 0.07}	                
\definecolor{red-brown}{rgb}{0.65, 0.16, 0.16}	                
\definecolor{red-violet}{rgb}{0.78, 0.08, 0.52}	                
\definecolor{redwood}{rgb}{0.67, 0.31, 0.32}	                
\definecolor{regalia}{rgb}{0.32, 0.18, 0.5}	                    
\definecolor{richblack}{rgb}{0.0, 0.25, 0.25}	                
\definecolor{richbrilliantlavender}{rgb}{0.95, 0.65, 1.0}	    
\definecolor{richcarmine}{rgb}{0.84, 0.0, 0.25}	                
\definecolor{richelectricblue}{rgb}{0.03, 0.57, 0.82}	        
\definecolor{richlavender}{rgb}{0.67, 0.38, 0.8}	            
\definecolor{richlilac}{rgb}{0.71, 0.4, 0.82}	                
\definecolor{richmaroon}{rgb}{0.69, 0.19, 0.38}	                
\definecolor{riflegreen}{rgb}{0.25, 0.28, 0.2}	                
\definecolor{robineggblue}{rgb}{0.0, 0.8, 0.8}	                
\definecolor{rose}{rgb}{1.0, 0.0, 0.5}	                        
\definecolor{rosebonbon}{rgb}{0.98, 0.26, 0.62}	                
\definecolor{roseebony}{rgb}{0.4, 0.3, 0.28}	                
\definecolor{rosegold}{rgb}{0.72, 0.43, 0.47}	                
\definecolor{rosemadder}{rgb}{0.89, 0.15, 0.21}	                
\definecolor{rosepink}{rgb}{1.0, 0.4, 0.8}	                    
\definecolor{rosequartz}{rgb}{0.67, 0.6, 0.66}	                
\definecolor{rosetaupe}{rgb}{0.56, 0.36, 0.36}	                
\definecolor{rosevale}{rgb}{0.67, 0.31, 0.32}	                
\definecolor{rosewood}{rgb}{0.4, 0.0, 0.04}	                    
\definecolor{rossocorsa}{rgb}{0.83, 0.0, 0.0}	                
\definecolor{rosybrown}{rgb}{0.74, 0.56, 0.56}	                
\definecolor{royalazure}{rgb}{0.0, 0.22, 0.66}	                
\definecolor{royalblue(traditional)}{rgb}{0.0, 0.14, 0.4}	    
\definecolor{royalblue(web)}{rgb}{0.25, 0.41, 0.88}	            
\definecolor{royalfuchsia}{rgb}{0.79, 0.17, 0.57}	            
\definecolor{royalpurple}{rgb}{0.47, 0.32, 0.66}	            
\definecolor{ruby}{rgb}{0.88, 0.07, 0.37}	                    
\definecolor{ruddy}{rgb}{1.0, 0.0, 0.16}	                    
\definecolor{ruddybrown}{rgb}{0.73, 0.4, 0.16}	                
\definecolor{ruddypink}{rgb}{0.88, 0.56, 0.59}	                
\definecolor{rufous}{rgb}{0.66, 0.11, 0.03}	                    
\definecolor{russet}{rgb}{0.5, 0.27, 0.11}	                    
\definecolor{rust}{rgb}{0.72, 0.25, 0.05}	                    
\definecolor{sacramentostategreen}{rgb}{0.0, 0.34, 0.25}	    
\definecolor{saddlebrown}{rgb}{0.55, 0.27, 0.07}	            
\definecolor{safetyorange(blazeorange)}{rgb}{1.0, 0.4, 0.0}	    
\definecolor{saffron}{rgb}{0.96, 0.77, 0.19}	                
\definecolor{st.patrick\'sblue}{rgb}{0.14, 0.16, 0.48}	        
\definecolor{salmon}{rgb}{1.0, 0.55, 0.41}	                    
\definecolor{salmonpink}{rgb}{1.0, 0.57, 0.64}	                
\definecolor{sand}{rgb}{0.76, 0.7, 0.5}	                        
\definecolor{sanddune}{rgb}{0.59, 0.44, 0.09}	                
\definecolor{sandstorm}{rgb}{0.93, 0.84, 0.25}	                
\definecolor{sandybrown}{rgb}{0.96, 0.64, 0.38}	                
\definecolor{sandytaupe}{rgb}{0.59, 0.44, 0.09}	                
\definecolor{sangria}{rgb}{0.57, 0.0, 0.04}	                    
\definecolor{sapgreen}{rgb}{0.31, 0.49, 0.16}	                
\definecolor{sapphire}{rgb}{0.03, 0.15, 0.4}	                
\definecolor{satinsheengold}{rgb}{0.8, 0.63, 0.21}	            
\definecolor{scarlet}{rgb}{1.0, 0.13, 0.0}	                    
\definecolor{schoolbusyellow}{rgb}{1.0, 0.85, 0.0}	            
\definecolor{screamin\'green}{rgb}{0.46, 1.0, 0.44}	            
\definecolor{seagreen}{rgb}{0.18, 0.55, 0.34}	                
\definecolor{sealbrown}{rgb}{0.2, 0.08, 0.08}	                
\definecolor{seashell}{rgb}{1.0, 0.96, 0.93}	                
\definecolor{selectiveyellow}{rgb}{1.0, 0.73, 0.0}	            
\definecolor{sepia}{rgb}{0.44, 0.26, 0.08}	                    
\definecolor{shadow}{rgb}{0.54, 0.47, 0.36}	                    
\definecolor{shamrockgreen}{rgb}{0.0, 0.62, 0.38}	            
\definecolor{shockingpink}{rgb}{0.99, 0.06, 0.75}	            
\definecolor{sienna}{rgb}{0.53, 0.18, 0.09}	                    
\definecolor{silver}{rgb}{0.75, 0.75, 0.75}	                    
\definecolor{sinopia}{rgb}{0.8, 0.25, 0.04}	                    
\definecolor{skobeloff}{rgb}{0.0, 0.48, 0.45}	                
\definecolor{skyblue}{rgb}{0.53, 0.81, 0.92}	                
\definecolor{skymagenta}{rgb}{0.81, 0.44, 0.69}	                
\definecolor{slateblue}{rgb}{0.42, 0.35, 0.8}	                
\definecolor{slategray}{rgb}{0.44, 0.5, 0.56}	                
\definecolor{smalt(darkpowderblue)}{rgb}{0.0, 0.2, 0.6}	        
\definecolor{smokeytopaz}{rgb}{0.58, 0.25, 0.03}	            
\definecolor{smokyblack}{rgb}{0.06, 0.05, 0.03}	                
\definecolor{snow}{rgb}{1.0, 0.98, 0.98}	                    
\definecolor{spirodiscoball}{rgb}{0.06, 0.75, 0.99}	            
\definecolor{splashedwhite}{rgb}{1.0, 0.99, 1.0}	            
\definecolor{springbud}{rgb}{0.65, 0.99, 0.0}	                
\definecolor{springgreen}{rgb}{0.0, 1.0, 0.5}	                
\definecolor{steelblue}{rgb}{0.27, 0.51, 0.71}	                
\definecolor{stildegrainyellow}{rgb}{0.98, 0.85, 0.37}	        
\definecolor{straw}{rgb}{0.89, 0.85, 0.44}	                    
\definecolor{sunglow}{rgb}{1.0, 0.8, 0.2}	                    
\definecolor{sunset}{rgb}{0.98, 0.84, 0.65}	                    
\definecolor{tan}{rgb}{0.82, 0.71, 0.55}	                    
\definecolor{tangelo}{rgb}{0.98, 0.3, 0.0}	                    
\definecolor{tangerine}{rgb}{0.95, 0.52, 0.0}	                
\definecolor{tangerineyellow}{rgb}{1.0, 0.8, 0.0}	            
\definecolor{taupe}{rgb}{0.28, 0.24, 0.2}	                    
\definecolor{taupegray}{rgb}{0.55, 0.52, 0.54}	                
\definecolor{teagreen}{rgb}{0.82, 0.94, 0.75}	                
\definecolor{tearose(orange)}{rgb}{0.97, 0.51, 0.47}	        
\definecolor{tearose(rose)}{rgb}{0.96, 0.76, 0.76}	            
\definecolor{teal}{rgb}{0.0, 0.5, 0.5}	                        
\definecolor{tealblue}{rgb}{0.21, 0.46, 0.53}	                
\definecolor{tealgreen}{rgb}{0.0, 0.51, 0.5}	                
\definecolor{tenne-tawny}{rgb}{0.8, 0.34, 0.0}	                
\definecolor{terracotta}{rgb}{0.89, 0.45, 0.36}	                
\definecolor{thistle}{rgb}{0.85, 0.75, 0.85}	                
\definecolor{thulianpink}{rgb}{0.87, 0.44, 0.63}	            
\definecolor{ticklemepink}{rgb}{0.99, 0.54, 0.67}	            
\definecolor{tiffanyblue}{rgb}{0.04, 0.73, 0.71}	            
\definecolor{tiger\'seye}{rgb}{0.88, 0.55, 0.24}	            
\definecolor{timberwolf}{rgb}{0.86, 0.84, 0.82}	                
\definecolor{titaniumyellow}{rgb}{0.93, 0.9, 0.0}	            
\definecolor{tomato}{rgb}{1.0, 0.39, 0.28}	                    
\definecolor{toolbox}{rgb}{0.45, 0.42, 0.75}	                
\definecolor{tractorred}{rgb}{0.99, 0.05, 0.21}	                
\definecolor{trolleygrey}{rgb}{0.5, 0.5, 0.5}	                
\definecolor{tropicalrainforest}{rgb}{0.0, 0.46, 0.37}	        
\definecolor{trueblue}{rgb}{0.0, 0.45, 0.81}	                
\definecolor{tuftsblue}{rgb}{0.28, 0.57, 0.81}	                
\definecolor{tumbleweed}{rgb}{0.87, 0.67, 0.53}	                
\definecolor{turkishrose}{rgb}{0.71, 0.45, 0.51}	            
\definecolor{turquoise}{rgb}{0.19, 0.84, 0.78}	                
\definecolor{turquoiseblue}{rgb}{0.0, 1.0, 0.94}	            
\definecolor{turquoisegreen}{rgb}{0.63, 0.84, 0.71}	            
\definecolor{tuscanred}{rgb}{0.51, 0.21, 0.21}	                
\definecolor{twilightlavender}{rgb}{0.54, 0.29, 0.42}	        
\definecolor{tyrianpurple}{rgb}{0.4, 0.01, 0.24}	            
\definecolor{uablue}{rgb}{0.0, 0.2, 0.67}	                    
\definecolor{uared}{rgb}{0.85, 0.0, 0.3}	                    
\definecolor{ube}{rgb}{0.53, 0.47, 0.76}	                    
\definecolor{uclablue}{rgb}{0.33, 0.41, 0.58}	                
\definecolor{uclagold}{rgb}{1.0, 0.7, 0.0}	                    
\definecolor{ufogreen}{rgb}{0.24, 0.82, 0.44}	                
\definecolor{ultramarine}{rgb}{0.07, 0.04, 0.56}	            
\definecolor{ultramarineblue}{rgb}{0.25, 0.4, 0.96}	            
\definecolor{ultrapink}{rgb}{1.0, 0.44, 1.0}	                
\definecolor{umber}{rgb}{0.39, 0.32, 0.28}	                    
\definecolor{unitednationsblue}{rgb}{0.36, 0.57, 0.9}	        
\definecolor{unmellowyellow}{rgb}{1.0, 1.0, 0.4}	            
\definecolor{upforestgreen}{rgb}{0.0, 0.27, 0.13}	            
\definecolor{upmaroon}{rgb}{0.48, 0.07, 0.07}	                
\definecolor{upsdellred}{rgb}{0.68, 0.09, 0.13}	                
\definecolor{urobilin}{rgb}{0.88, 0.68, 0.13}	                
\definecolor{usccardinal}{rgb}{0.6, 0.0, 0.0}	                
\definecolor{uscgold}{rgb}{1.0, 0.8, 0.0}	                    
\definecolor{utahcrimson}{rgb}{0.83, 0.0, 0.25}	                
\definecolor{vanilla}{rgb}{0.95, 0.9, 0.67}	                    
\definecolor{vegasgold}{rgb}{0.77, 0.7, 0.35}	                
\definecolor{venetianred}{rgb}{0.78, 0.03, 0.08}	            
\definecolor{verdigris}{rgb}{0.26, 0.7, 0.68}	                
\definecolor{vermilion}{rgb}{0.89, 0.26, 0.2}	                
\definecolor{veronica}{rgb}{0.63, 0.36, 0.94}	                
\definecolor{violet}{rgb}{0.56, 0.0, 1.0}	                    
\definecolor{violet(colorwheel)}{rgb}{0.5, 0.0, 1.0}	        
\definecolor{violet(ryb)}{rgb}{0.53, 0.0, 0.69}	                
\definecolor{violet(web)}{rgb}{0.93, 0.51, 0.93}	            
\definecolor{viridian}{rgb}{0.25, 0.51, 0.43}	                
\definecolor{vividauburn}{rgb}{0.58, 0.15, 0.14}	            
\definecolor{vividburgundy}{rgb}{0.62, 0.11, 0.21}	            
\definecolor{vividcerise}{rgb}{0.85, 0.11, 0.51}	            
\definecolor{vividtangerine}{rgb}{1.0, 0.63, 0.54}	            
\definecolor{vividviolet}{rgb}{0.62, 0.0, 1.0}	                
\definecolor{warmblack}{rgb}{0.0, 0.26, 0.26}	                
\definecolor{wenge}{rgb}{0.39, 0.33, 0.32}	                    
\definecolor{wheat}{rgb}{0.96, 0.87, 0.7}	                    
\definecolor{white}{rgb}{1.0, 1.0, 1.0}	                        
\definecolor{whitesmoke}{rgb}{0.96, 0.96, 0.96}	                
\definecolor{wildblueyonder}{rgb}{0.64, 0.68, 0.82}	            
\definecolor{wildstrawberry}{rgb}{1.0, 0.26, 0.64}	            
\definecolor{wildwatermelon}{rgb}{0.99, 0.42, 0.52}	            
\definecolor{wisteria}{rgb}{0.79, 0.63, 0.86}	                
\definecolor{xanadu}{rgb}{0.45, 0.53, 0.47}	                    
\definecolor{yaleblue}{rgb}{0.06, 0.3, 0.57}	                
\definecolor{yellow}{rgb}{1.0, 1.0, 0.0}	                    
\definecolor{yellow(munsell)}{rgb}{0.94, 0.8, 0.0}	            
\definecolor{yellow(ncs)}{rgb}{1.0, 0.83, 0.0}	                
\definecolor{yellow(process)}{rgb}{1.0, 0.94, 0.0}	            
\definecolor{yellow(ryb)}{rgb}{1.0, 1.0, 0.2}	                
\definecolor{yellow-green}{rgb}{0.6, 0.8, 0.2}	                
\definecolor{zaffre}{rgb}{0.0, 0.08, 0.66}	                    
\definecolor{zinnwalditebrown}{rgb}{0.17, 0.09, 0.03}	        

\definecolor{solarized-yellow}{HTML}	{B58900}
\definecolor{solarized-orange}{HTML}	{CB4B16}
\definecolor{solarized-red}{HTML}		{DC322F}
\definecolor{solarized-magenta}{HTML}	{D33682}
\definecolor{solarized-violet}{HTML}	{6C71C4}
\definecolor{solarized-blue}{HTML}		{268BD2}
\definecolor{solarized-cyan}{HTML}		{2AA198}
\definecolor{solarized-green}{HTML}		{859900}

\definecolor{solarized-base03}{HTML}	{002B36}
\definecolor{solarized-base02}{HTML}	{073642}
\definecolor{solarized-base01}{HTML}	{586E75}
\definecolor{solarized-base00}{HTML}	{657B83}
\definecolor{solarized-base0}{HTML}		{839496}
\definecolor{solarized-base1}{HTML}		{93A1A1}
\definecolor{solarized-base2}{HTML}		{EEE8D5}
\definecolor{solarized-base3}{HTML}		{FDF6E3}

\usepackage{colortbl}
\usepackage{array}
\newcolumntype{A}{>{\columncolor{silver!10}}c}
\newcolumntype{B}{>{\columncolor{black!10}}c}
\newcolumntype{C}{>{\columncolor{alizarin!10}}c}
\newcolumntype{D}{>{\columncolor{vividviolet!10}}c}
\newcolumntype{E}{>{\columncolor{aquamarine!10}}c}
\pgfplotsset{compat=newest}
\usepackage{pgfplots}
\usepackage{tikz}
\usetikzlibrary{patterns, arrows.meta}  %
\usetikzlibrary{shapes,decorations.pathmorphing}

\usepackage{amssymb,graphicx,multirow,esint}

\usepackage{acronym}
\usepackage{hyperref}

\newtheorem{thm}{Theorem}[section]
\newtheorem{defn}{Definition}[section]
\newtheorem{lem}[thm]{Lemma} 
\newtheorem{rem}[thm]{Remark}

\newtheorem{cor}[thm]{Corollary}
\usepackage{comment}

\numberwithin{figure}{section}
\numberwithin{table}{section}
\numberwithin{equation}{section}

\newcommand{\bx}{\boldsymbol{x}}

\newcommand{\dx}{\,\mathrm{d}\bx}
\newcommand{\ds}{\,\mathrm{d}s}
\newcommand{\dt}{\,\mathrm{d}t}
\newcommand{\dz}{\,\mathrm{d}z}
\newcommand\norm[1]{\lVert#1\rVert}
\newcommand{\bn}{\boldsymbol{n}}
\newcommand{\trinl}{\ensuremath{|\!|\!|}}
\newcommand{\trinr}{\ensuremath{|\!|\!|}}
\definecolor{lightgreen}{rgb}{0.22,0.50,0.25}
\definecolor{lightblue}{rgb}{0.22,0.45,0.70}
\definecolor{darkred}{rgb}{0.82,0.15,0.12}
\newcommand{\cred}[1]{}
\newcommand{\ccyan}[1]{}
\newcommand{\cblue}[1]{}

\allowdisplaybreaks

\acrodef{pde}[PDE]{partial differential equation} 
\acrodefplural{pde}[PDEs]{partial differential equations} 
\acrodef{dof}[DOF]{degree of freedom} 
\acrodefplural{dof}[DOFs]{degrees of freedom} 
\acrodef{fe}[FE]{finite element} 
\acrodefplural{fe}[FEs]{finite elements} 
\acrodef{dg}[dG]{discontinuous Galerkin}

\acrodef{ted}[TED]{thermoelastic diffusion}
\acrodef{tpe}[TPE]{thermo-poroelasticity}

\newcommand{\rev}{}

\begin{document}
\markboth{Nataraj, Ruiz-Baier, Yousuf}{Thermoelastic diffusion and thermo-poroelasticity of thin plates}

\catchline{}{}{}{}{}

\title{\large Unified numerical analysis for thermoelastic diffusion \\ and thermo-poroelasticity of thin plates}

\author{\normalsize Neela Nataraj}
\address{Department of Mathematics, Indian Institute of Technology Bombay, Mumbai, Maharashtra
400076, India. Email: \email{neela@math.iitb.ac.in}} 
\author{\normalsize Ricardo Ruiz-Baier}
\address{Corresponding author. School of Mathematics, Monash University, 
9 Rainforest Walk, 3800 Melbourne, Australia; and \\
Universidad Adventista de Chile, Casilla 7-D, Chill\'an, Chile; and \\
Institute for Computer Science and Mathematical Modelling, Sechenov First Moscow State Medical University, Moscow, Russia.  Email: \email{ricardo.ruizbaier@monash.edu}}
\author{\normalsize Aamir Yousuf}
\address{IITB–Monash Research Academy, Indian Institute of Technology Bombay, Mumbai, Maharashtra 400076, India. Email: \email{aamir72@iitb.ac.in}}
\maketitle

\begin{history}
\revised{\today}
\end{history}

\begin{abstract}
We investigate a coupled hyperbolic-parabolic system modeling thermoelastic diffusion (resp. thermo-poro\-elast\-icity) in plates, consisting of a fourth-order hyperbolic partial differential equation for plate deflection and two second-order parabolic partial differential equations for the first moments of temperature and chemical potential (resp. pore pressure). The unique solvability of the system is established \rev{by means of a} Galerkin approach, and the additional regularity of the solution is obtained under appropriately strengthened data. For numerical approximation, we employ the Newmark method for time discretization of the hyperbolic term and a continuous interior penalty scheme for the spatial discretization of displacement. For the parabolic equations that represent the first moments of temperature and chemical potential (resp. pore pressure), we use the Crank--Nicolson method for time discretization and conforming finite elements for spatial discretization. The convergence of the fully discrete scheme with quasi-optimal rates in space and time is established. The numerical experiments demonstrate the effectiveness of the 2D Kirchhoff--Love plate model in capturing thermoelastic diffusion  and thermo-poroelastic behavior in specific materials. We illustrate that, as plate thickness decreases, the two-dimensional simulations closely approximate the results of the three-dimensional problem. Finally, the numerical experiments also validate the theoretical rates of convergence.
\end{abstract}

\ccode{\rev{Mathematics Subject Classifications (2020): 65M30, 65M12.}}

\section{Introduction}\label{section-intro}
\noindent\textbf{Scope and presentation of the problem.} 
This study presents a unified  analysis of thin plate structures that describe \ac{ted} and \ac{tpe}. The coupled system comprises of a fourth-order hyperbolic \ac{pde} governing the plate deflection  with two second-order parabolic \acp{pde} describing the first moments of temperature and chemical potential (resp. pressure) in the case of \ac{ted} (resp. \ac{tpe}).
 A combination of the $C^0$ interior penalty ($C^0$IP) scheme and conforming \acp{fe} is used for spatial discretization. The temporal discretization utilizes the Newmark and Crank--Nicolson schemes for approximating the second and first-order terms, respectively. We establish optimal order theoretical rates of convergence and the numerical experiments validate them.

Thermodiffusion in an elastic solid results from the coupling of strain, temperature, and mass diffusion fields. {In the context of \ac{tpe}, this coupling replaces chemical potential by pore pressure, accounting for the interactions between mechanical deformation, thermal effects, and fluid flow within porous media. The \ac{ted} phenomena play a critical role in various engineering applications, for example, in satellite,  aircraft operations, and in the  manufacturing of integrated circuits, integrated resistors, semiconductor substrates, and transistors.  Additionally, \ac{ted} is a key part in the heat and mass transfer processes involved in enhancing oil extraction conditions from deposits. Understanding diffusion properties in thin thermoelastic plates is critical in the study of \rev{advanced} materials predicting stress distribution, material fatigue, and potential failure such as warping or cracking during operation. Also, determination of the flexural motion of fluid-saturated poroelastic plates is an important problem in structural and geotechnical engineering, bioengineering, and geodynamics.

\begin{figure}[ht]
    \centering
    \begin{minipage}{0.45\textwidth}
{\small \begin{tabular}{ll}
\hline
\textit{Coeff.}& \textit{Description} \\
\hline
\quad $\lambda$ & Lam\'e's first constant \\
\quad$\mu$ & Shear modulus  \\
\quad$\varrho$  & Measure of the diffusive effect \\
\quad$\alpha_t$ & Thermal expansion \\
\quad$\alpha_c$ & Diffusion expansion \\
\quad$\varpi$   & Measure of thermodiffusion effect \\
\quad$c_E$ & Specific heat at constant strain \\
\quad$\rho$& Mass density per unit volume \\
\quad$k_1$ &  Thermal conductivity\\
\quad$k_2$ & Diffusion conductivity\\
\quad $\beta^*$ & Biot--Willis constant\\
\quad $\gamma^*$& Thermal dilation coefficient\\
\quad $\varrho^*$& Biot modulus\\
\quad $k_2^*$& Permeability\\
\end{tabular}}
\captionof{table}{{Physical constants.}}
        \label{fig:subfig1}
    \end{minipage}\hfill
    \begin{minipage}{0.45\textwidth}
    \centering
\begin{tikzpicture}
\pgfdeclarepatternformonly
        {dense dots} 
        {\pgfpoint{-1pt}{-1pt}} 
        {\pgfpoint{2pt}{2pt}}   
        {\pgfpoint{0.8pt}{0.8pt}}   
        {
            \pgfpathcircle{\pgfpoint{0pt}{0pt}}{0.5pt} 
            \pgfusepath{fill}
        }

    \begin{axis}[
        hide axis, 
        view={80}{10}, 
        colormap/hot,
        clip=false, 
        axis x line=middle, 
        axis y line=middle, 
       axis z line=middle, 
        xmin=-2, xmax=16, 
        ymin=-2, ymax=10, 
        zmin=-0.2, zmax=0.2 
    ]
        \definecolor{lightblue}{rgb}{0.678, 0.847, 0.902}

        \draw[dashed,thick] (axis cs:0,0,0) -- (axis cs:14,0,0) -- (axis cs:14,8,0) -- (axis cs:0,8,0) -- cycle;
        \draw[thin] (axis cs:0,0,0.02) -- (axis cs:14,0,0.02) -- (axis cs:14,8,0.02) -- (axis cs:0,8,0.02) -- cycle;
        \draw[thin] (axis cs:0,0,-0.02) -- (axis cs:14,0,-0.02) -- (axis cs:14,8,-0.02) -- (axis cs:0,8,-0.02) -- cycle;

         \fill[pattern=dense dots, pattern color=lightblue] 
        (axis cs:0,0,0.02) -- 
        (axis cs:0,0,-0.02) -- 
        (axis cs:14,0,-0.02) -- 
        (axis cs:14,0,0.02) -- 
        cycle;
        \fill[pattern=dense dots, pattern color=lightblue] 
        (axis cs:0,0,0.02) -- 
        (axis cs:0,0,-0.02) -- 
        (axis cs:0,8,-0.02) -- 
        (axis cs:0,8,0.02) -- 
        cycle;
        \fill[pattern=dense dots, pattern color=lightblue] 
        (axis cs:0,8,0.02) -- 
        (axis cs:0,8,-0.02) -- 
        (axis cs:14,8,-0.02) -- 
        (axis cs:14,8,0.02) -- 
        cycle;
        \fill[pattern=dense dots, pattern color=lightblue] 
        (axis cs:13.8,0.1,0.0199) -- 
        (axis cs:13.8,0.1,-0.0199) -- 
        (axis cs:13.8,7.9,-0.0199) -- 
        (axis cs:13.8,7.9,0.0199) -- 
        cycle;
        \fill[pattern=dense dots, pattern color=lightblue] 
        (axis cs:0,0,0) -- (axis cs:14,0,0) -- (axis cs:14,8,0) -- (axis cs:0,8,0) -- cycle;

        \draw[thin] (axis cs:0,0,0) -- (axis cs:0,0,0.02);
        \draw[thin] (axis cs:14,0,0) -- (axis cs:14,0,0.02);
        \draw[thin] (axis cs:14,8,0) -- (axis cs:14,8,0.02);
        \draw[thin] (axis cs:0,8,0) -- (axis cs:0,8,0.02);
        \draw[thin] (axis cs:0,0,0) -- (axis cs:0,0,-0.02);
        \draw[thin] (axis cs:14,0,0) -- (axis cs:14,0,-0.02);
        \draw[thin] (axis cs:14,8,0) -- (axis cs:14,8,-0.02);
        \draw[thin] (axis cs:0,8,0) -- (axis cs:0,8,-0.02);
 \definecolor{darkblue}{rgb}{0.678, 3.847, 7.902}
\draw[thick,darkblue, decoration={snake, segment length=5mm, amplitude=1mm}, decorate, -{Stealth}] (axis cs:11,1,-0.01) -- (axis cs:11,7.2,-0.01);
                \node at (axis cs:0,7,-0.02) [align=center] {{\textcolor{darkblue}{Mass}}};

                  \node at (axis cs:4,5,-0.13) [align=center] {{\textcolor{red}{Heat Source}}};
\fill[orange] (2,4.7,-0.107) circle (1.4); 
\fill[yellow] (2,4.7,-0.107) circle (0.8);
\fill[red] (2,4.7,-0.107) circle (0.5);
\fill[blue] (2,4.7,-0.107) circle (0.2);
        \draw[-{Stealth}, thick] (axis cs:10,3,0.1) -- (axis cs:10,3,0.01);
         \draw[-{Stealth}, thick] (axis cs:10,4,0.1) -- (axis cs:10,4,0.01);
          \draw[-{Stealth}, thick] (axis cs:10,2,0.1) -- (axis cs:10,2,0.01);
        \node at (axis cs:10,3,0.115) [align=center] {\text{Load}};
      \draw[-{Stealth},-{Stealth}, thin] (axis cs:14,8.5,0.018) -- (axis cs:14,8.5,-0.022);
      \draw[-{Stealth},-{Stealth}, thin] (axis cs:14,8.5,-0.022) -- (axis cs:14,8.5,0.018);
      \node at (axis cs:13,9,-0.003) [align=left] {\text{\scriptsize{d}}};

    \end{axis}
\end{tikzpicture}
\vspace{-0.3cm}
   \caption{3D plate in reference configuration.}
 \vspace{0.3cm}
\begin{tikzpicture}
\begin{axis}[
 hide axis, 
        view={80}{10}, 
        colormap/hot,
        clip=false, 
        axis x line=middle, 
        axis y line=middle, 
       axis z line=middle, 
        xmin=-2, xmax=16, 
        ymin=-2, ymax=10, 
        zmin=-0.2, zmax=0.2 
       ]
  \addplot3[
            surf,
            shader=flat,
            opacity=0.7,
            samples=50,
            samples y=25,
            domain=0:14,
            y domain=0:8
        ] 
        { -0.07* sin(deg(pi * x / 14)) * sin(deg(pi * y / 8)) }; 
        \end{axis}
  \end{tikzpicture}

       \caption{ {Mid-surface  at current configuration.}}
        \label{fig:subfig2}
    \end{minipage}
    \label{fig:mainfigure}
\end{figure}
Let $\widehat{\Omega} \subset \mathbb{R}^3$ denote a thin, isotropic, flat plate with a uniform thickness $  d $.  Additionally, we define the time interval as $[0, T]$. We denote the mid-surface of the plate as $\Omega \subset \mathbb{R}^2$, which is assumed to lie in alignment with the $xy$-axis, forming a bounded domain with a Lipschitz continuous boundary $\Gamma$. The elastodynamics of the mid-surface of the plate is characterized by the deflection $$u(\bx,t)=\frac{1}{ d }\int_{ - d /2}^{  d /2} \hat{u}_3 \dz,$$ which represents the transverse displacement $\hat{u}_3(x,y,z,t)$ averaged through the thickness and is a scalar function of $\bx=(x, y)$ and $t$ only. The first moments of temperature $\hat{\theta}(x,y,z,t)$ and chemical potential (resp. pore pressure) $\hat{p}(x,y,z,t)$  (resp. $\hat{p}^*(x,y,z,t)$) are denoted by $$ \theta(\bx,t)= \int_{ - d /2}^{  d /2} z\hat{\theta} \dz,\; \text{and} \;\; p(\bx,t)=\int_{-  d /2}^{  d /2} z\hat{p} \dz \; (\text{ resp. }p(\bx,t)=\int_{-  d /2}^{  d /2} z\hat{p}^* \dz). $$  
The authors in Ref. \refcite{Aouadi} formulated a model from the 3D \eqref{3D1}-\eqref{3D3} for \ac{ted} in thin plates, under the assumption that body forces, external loads, and sources of heat and diffusion are absent. This model is based on the 2D Kirchhoff--Love hypotheses for thin plates, with classical Fourier's law for heat conduction and Fick's law for diffusion. An enhanced {\it novel} model considered in this article for \ac{ted} and \ac{tpe} that include external loads, heat source, and mass diffusion 
and is presented as follows: the {coupled} model aims to determine  {mid-surface deflection $u$, first moments of temperature $\theta$ and chemical potential {(resp. pore pressure) }  $p$} 
such that
\begin{subequations}
\label{eq:coupled}
\begin{align}
    u_{tt}-a_0 \Delta u_{tt}+d_0\Delta^2 u+\alpha \Delta \theta+\beta \Delta p&=f(\bx,t) \qquad  \text{ in }\Omega \times (0,T],\label{p2;model11}\\
    a_1\theta_t-\gamma p_t  {+b_1\theta} -c_1 \Delta \theta -\alpha \Delta u_t&=\phi(\bx,t) \qquad     \text{ in }\Omega \times (0,T],\label{model22}\\
    a_2 p_t-\gamma \theta_t-\kappa \Delta p-\beta \Delta  u_t&=g(\bx,t) \qquad \hspace{0.05cm} \text{ in }\Omega \times (0,T],\label{p2;model33}\\
    u=\partial_{\bn} u=0, \;\theta=0,\; p&=0 \qquad  \qquad\hspace{0.05cm}    \text{ on }\Gamma\times {[0,T]}, \label{bound-cond}\\
     u|_{t=0}=u^0, \; u_{t}|_{t=0}=u^{*0}, \; \theta |_{t=0}=\theta ^0, \; p|_{t=0}&=p^0 \qquad \qquad   \text{in }\Omega, \label{init-conds}
\end{align}\end{subequations}
where $\bn$ is the outward-pointing unit normal,
$\partial_{\bn} u  = \nabla u \cdot \bn $ is the outer normal derivative of $u$ on $\partial \Omega$,   $u_{t} ,\theta_t, p_t$ {(resp. $u_{tt}$)}  denote the first { (resp. second)}-order  derivatives with respect to {time.} {Here, the chemical potential (resp. pore pressure) across the plate is assumed to be linear}.  The coefficients in the { system \eqref{eq:coupled}} depend on the constants listed in the {Table \ref{fig:subfig1}}, with further details deferred to {Subsection~\ref{subsec-numerics-Verif-kirchoff}}.

As mentioned in {Subsection~\ref{subsec-numerics-Verif-kirchoff}}, it is possible to use models of thermoelastic  plates with voids or vacuous pores\cite{birsan2003bending,ghiba2013temporal,hussein2020mathematical,liu2017well}. 
In this fully coupled system the Kirchhoff--Love equations for the deflection interact with the dynamics of the total
amount of fluid, and the thermal energy conservation exhibits a dependence on the plate poromechanics through
the thermal stress and thermal dilation contributions. {The 3D system of equations \eqref{3d-poro1}–\eqref{3d-poro3} for \ac{tpe} closely resembles the 3D \ac{ted} system \eqref{3D1}–\eqref{3D3}, with the primary differences being the physical constants involved and the sign of the coupling constant  between the second and third equations. Therefore, by following the dimensional reduction approach and using Darcy's law for fluid flow (in contrast to Fick's law for diffusion) as done in eqns. (9)-(46) of Ref. \refcite{Aouadi}, one can derive the 2D \ac{tpe} model from \eqref{3d-poro1}–\eqref{3d-poro3}, which leads to the system \eqref{p2;model11}–\eqref{p2;model33}.}

In this paper, we assume that all coefficients, except for \(\gamma\), are positive. These assumptions are realistic because, for the \ac{ted} model (see an explicit representation in Subsection~\ref{subsec-numerics-Verif-kirchoff}) and the \ac{tpe} thin plate model, the coefficients remain positive provided the basic 3D constants listed in   {Table~\ref{fig:subfig1}} are positive. 
Regarding the parameter \(\gamma\), we allow \(\gamma \in \mathbb{R}\). For the \ac{ted} model the condition  \(a_1a_2 - \gamma^2 > 0\) is inherently satisfied due to the material's constitutive properties, as detailed in Table~\ref{table-coeefients}.  This condition is typically assumed for the \ac{tpe} model to ensure well-posedness and physical realism\cite{MR4146798,son-young}. This shows that, $|\gamma|/a_1 < a_2/|\gamma|${ and } hence there {exists} some $\gamma_0>0$ such that $|\gamma|/a_1<\gamma_0 < a_2/|\gamma|$, and consequently, 
 \begin{equation}
   a_1-|\gamma|/\gamma_0>0 \qquad \text{ and }\qquad  a_2-|\gamma|\gamma_0>0 . \label{gamma_0}
\end{equation}

\medskip 
\noindent\textbf{Literature overview.} 
The foundational \ac{ted} theory was initially proposed \rev{in} Ref. \refcite{Nowacki}. Rigorous derivations have been undertaken to establish the linear Kirchhoff--Love thermoelastic plate model, as shown in Ref. \refcite{Lagnese}, where the plate is assumed to be homogeneous, as well as elastically and thermally isotropic. Poroelastic models based on the Kirchhoff--Love plate theory and Biot’s theory of poroelasticity are discussed in Ref. \refcite{Taber} and Ref. \refcite{Cederbaum}. In these works, the pressure variation in the longitudinal section is neglected in the former, while a linear pressure distribution across the plate is considered in the latter. 
Refs. \refcite{Aouadi,deng2024stability} discuss hyperbolic problems; the well-posedness of the problem is analyzed using the semigroup theory approach, after transforming the system into an evolution equation by introducing velocity as a new variable with vertical displacement. 
The model discussed in this paper builds upon the derivations presented in Ref. \refcite{Aouadi}, which incorporate diffusion effects in homogeneous and isotropic thermoelastic thin plates. Our analysis uses a Galerkin method and compactness arguments for showing existence and uniqueness of weak and strong solutions\cite{brun2019well}.

Regarding numerical methods for  fully coupled multiphysics system, we mention that Ref. \refcite{Zhou} employs a mixed \rev{finite}  element method, the $H^1$-Galerkin method, and the interior penalty \ac{dg} method (IP-DG) for spatial discretization of the Kirchhoff--Love thermoelastic system, combined with the backward Euler method for temporal discretization. In Ref. \refcite{Iliev}, a quasi-static poroelastic model is considered, where the pressure moment is discretized using a standard FE approximation, while the biharmonic problem is addressed using a 
$C^0$IP method 
and a two-level scheme with weights for temporal discretization. 
In the three-dimensional setting, the Biot equations for poromechanics can be coupled with the thermal energy equation leading to a hyperbolic-parabolic system in fully dynamic  or elliptic-parabolic system in the quasi-static case. Galerkin methods for this problem are investigated in Ref. \refcite{zhang}, while mixed \ac{fe} and \ac{dg} discretizations are explored in, for example\cite{antonietti2023discontinuous,brun2019well,MR4146798}.
 Moreover,  fully discrete approximations using the conforming 
\rev{$\mathcal{P}_1$} \ac{fe} method and the implicit Euler scheme are studied for one-dimensional \ac{ted} problems in porous media\cite{bazarra2023numerical,jose}. In Ref. \refcite{Madureira_num}, semi- and fully discrete schemes for solving a one-dimensional \ac{ted} problem with a moving boundary and quadratic convergence in both time and space are established by employing conforming \acp{fe} for spatial discretization and Newmark’s time discretization. More recently, in  Ref. \refcite{khot_mc24}, the authors address the steady Biot--Kichhoff--Love problem with centered difference and backward Euler semi-discretization in time, and conforming and non-conforming virtual element methods for spatial discretization. They establish a priori error estimates in the best-approximation form, derive residual-based reliable and efficient a posteriori error estimates in appropriate norms, and demonstrate that these error bounds are robust with respect to the key model parameters.

\medskip 
\noindent\textbf{Main contributions.} 
In this paper, we analyze the unique solvability and numerical approximation for an asymptotic model for  \ac{ted}  and \ac{tpe} plate models consisting of a coupled PDE system of one hyperbolic fourth-order PDE for the plate's vertical  deflection, and two second-order parabolic PDEs for the thickness-averaged (first moment)  temperature distribution, and chemical potential/pore pressure. The unique solvability of the continuous formulation is based on the classical Galerkin approach (see, for example Refs. \refcite{MR3987945,MR3471659,oudaani2022existence,madureira2019global}). 
For the spatial discretization, {$C^0$IP  method}  and conforming 
$\mathcal{P}_1$ elements for {temperature and chemical potential (or pore pressure)} are employed. In terms of temporal discretization, we adopt Newmark's scheme for the first hyperbolic equation and apply the Crank--Nicolson method for the remaining parabolic equations, ensuring quadratic convergence in time.
  Following our recent work\cite{nry_preprint}, we utilize a modified Ritz projection for the analysis, based on the companion operator\cite{CC}.  In conjunction with this, we employ the standard 
$H^1$-conforming Ritz projections for temperature and chemical potential/pore pressure to obtain the error estimates.

The key contributions of this work are outlined below:
\begin{itemize}
\item The present analysis is robust with respect to the parameter \(\gamma\). Allowing \(\gamma\) to take values in \(\mathbb{R}\) enables a unified analytical framework that accommodates both the thermoelastic diffusion
and thermo-poroelastic thin plate models.

    \item The well-posedness of the fully coupled hyperbolic-parabolic thermoelastic diffusion
and thermo-poroelastic systems is demonstrated {in Subsection~\ref{subsec-existence} under reasonable} data regularity conditions. It should be noted that uniqueness for hyperbolic/parabolic coupled problems under the assumptions of Theorem~\ref{existence;thm} is not straightforward. It requires the use of  mollified test functions,  as explained later in the proof.
    \item  {A consistent and stable fully discrete scheme is developed in Section~\ref{sec:semi}. Due to the coupling of second-order terms, 
    special care must be taken in the choice of compatible \ac{fe} spaces that \rev{play} a crucial role in the choice of the test functions in the proofs of stability and error estimates. Moreover, similar care is required when approximating coupling terms involving time derivatives of different orders.}
    \item A novel concept of approximating the solution at the initial time step (see \eqref{3.10}), while incorporating the approximation properties established in Lemma~\ref{initial-error-lemm}, is introduced to the literature, facilitating the development of a fully discrete scheme for general hyperbolic-parabolic coupled systems without any assumptions regarding the solution and its approximation at this time step (see Ref. \refcite{Madureira_num}).
    \item A priori error estimates are derived in the best approximation form in both $L^2$, $H^1$ and energy norm for displacement {in Section~\ref{p2;fully-error-sec}.} \rev{Optimal}  error rates are also established in  $L^2$ and  $H^1$ norm for {temperature and chemical potential/pore pressure}. Also, the combination of Newmark--Crank--Nicolson time discretization schemes to approximate the second and first-order time derivatives, respectively, appearing in \eqref{eq:coupled}  \rev{yields} quadratic convergence rates.
    
    \item The superconvergence of the projected error in the energy norm is established (see Remark~\ref{rem-super}), in turn leading to lower  $H^s$- order estimates with $s=0,1$ (resp. $s=0$)  for $u$ (resp. $\theta$ and $p$) as established in Corollary~\ref{corr} (resp. Theorem~\ref{error-estimates thm}). While such superconvergence is expected in uncoupled problems, it is not straightforward in the current coupled problem since the polynomial degrees of the \ac{fe} spaces $V_h$ and $W_h$ used to approximate  first {\eqref{p2;model11} and last two equations \eqref{model22}-\eqref{p2;model33}} are different. Consequently, the Ritz projection defined in \eqref{p2;ritzprojection2} lacks orthogonality when the test function is chosen from a \ac{fe} space $V_h$.

\item  Subsection~\ref{subsec-numerics-Verif-kirchoff} demonstrates that the Kirchhoff--Love plate model is effective in capturing \ac{ted}  and \ac{tpe} behavior in specific materials (such as copper and flat layers of Berea sandstone, respectively). The findings indicate that as the plate thickness decreases, the two-dimensional simulations closely approximate the results from three-dimensional modeling, with a substantial reduction in computational time. This \rev{emphasizes} the efficiency and accuracy of 2D modeling for thin-plate structures.

\item Numerical results are provided in {Subsections~\ref{subsec-num-smooth}-\ref{subsec-num-lshape} to validate theoretical estimates and illustrate the effective performance of the proposed scheme with different values of $\gamma$}.
\end{itemize}

\medskip 
\noindent\textbf{Plan of the paper.} 
This paper is organized as follows. The remainder of this section introduces the common notation used throughout the manuscript. Section~\ref{sec:eq} provides definitions for solutions in the  weak sense, establishes the well-posedness of the system, and discusses regularity for weak solutions. {Section~\ref{sec:semi} details the spatial and temporal discretizations.  The fully discrete scheme, its unique solvability, and stability are presented in Section~\ref{sec:fully}. The  error analysis is discussed in  Section~\ref{p2;fully-error-sec}. In Section~\ref{sec:numer}, we present a few representative numerical examples that confirm the rates of convergence specified by the theoretical analysis. Subsection~\ref{subsec-numerics-Verif-kirchoff} discusses the detailed model description for the thermoelastic diffusion
and thermo-poroelastic systems.}

\medskip 
\noindent\textbf{Preliminaries.} 
For an open set $O \subset \mathbb{R}^2$,  we denote the Sobolev space $W^{m,2}(O)$ by $H^m(O)$  and equip it with the norm $\norm{w}_{H^m(O)}= (\underset{{|i| \le m}}{\sum} \norm{D^iw}_{L^2(O)}^2)^{1/2}$ and semi-norm $| w |^2_{H^m(O)}=(\underset{{|i| = m}}{\sum}\|D^iw \|_{L^2(O)}^2)^{1/2}$. For simplicity, we denote $L^2$ inner product by $(\bullet, \bullet)$ and norm by $\| \bullet \|$. Throughout this paper, $\mathcal{T}$ denotes a shape-regular triangulation of $\Omega$, $H^m({\cal T})$ denotes the Hilbert space $\underset {{K \in {\cal T}}}{\prod}H^m (K)$, and $\mathcal{P}_r({\cal T})$, the space of globally $L^2$ functions which are polynomials of degree at most $r$ in each $K$. The notation $\nabla$ (resp. $\nabla^2$)  denotes the   gradient (resp. Hessian). The piecewise energy norm is denoted by $\trinl \bullet \trinr_{\text{pw}}:=|\bullet|_{H^{2}({\cal T})}$ and $D^2_\text{pw}$ (resp. $\Delta_{\rm pw}$) stands for the piecewise Hessian (Laplacian).

Let $X$ be a normed space with norm $\|\bullet\|_X$ and $g:(0,T) \rightarrow X$ be a measurable function. Then for $1 \le p \le \infty$, we recall that  
\begin{align*}
\norm{g}_{L^p(0,T;X)}&=\norm{g}^p_{L^p(X)}:=\int_0^T \norm{g(t)}_X^p \dt, \ 1\le p< \infty, \quad \text{ and }\\
\ \norm{g}_{L^\infty(0,T;X)}&:=\underset{0 \le t \le T}{\text{ess sup}} \norm{g(t)}_X.
\end{align*}
Let  $L^p(0,T;X):=\left\{g:(0,T) \rightarrow X: \norm{g}_{L^p(X)}<\infty\right\}$. The space {$W^{1,p}(0,T;X)$  consists of all functions  $u \in L^p(0,T;X)$  such that  $u_t$  exists in the weak sense and belongs to  $L^p(0,T;X).$} For all non-negative integers $k$,   $C^k([0,T];X)$ denotes all $C^k$ functions $s :[0,T] \rightarrow X$ with $\norm {s}_{C^k([0,T];X)}=\underset{0 \le i \le k}{\sum} \underset{0 \le t\le T}{\max} {\norm{\frac{\partial^i s}{\partial t^i}}} < \infty$. 

For  real numbers $a > 0$, $b > 0$, and $\epsilon > 0$, we will make repeated use of the {Young's}  inequality $ab \leq \frac{\epsilon}{2}a^2 + \frac{1}{2\epsilon}b^2$. Finally, as usual, the notation $a \lesssim b$ represents $a \le Cb$, where the generic constant $C$ is independent of both mesh-size and time discretization parameter.

\begin{lem}[Gronwall's Lemma\cite{MR1638130}]\label{P1 gronwall}
Let $g$, $h$, and $r$ be non-negative integrable functions on $[0,T]$ and let $g$ satisfy 
 $\displaystyle g(t) \le h(t)+ \int_0^t r(s) g(s) \ds$  for all $t \in (0,T)$. Then 
\begin{equation*}
  g(t) \le h(t)+ \int_0^t h(s)r(s) e^{\int_s ^{t}{r(\tau )\;\text{d}\tau }}\ds  \quad \text{for all $t \in (0,T)$}.
   \end{equation*}
 \end{lem}

\section{Well-posedness and regularity results}\label{sec:eq}
In this section, we establish the well-posedness of the problem through the finite Galerkin approach, which follows these steps: (i) construct a sequence of approximate solutions to the continuous problem, (ii) derive a priori bounds on these approximations based on the initial data, (iii) use a compactness argument to show the existence of a limit for a subsequence in the weak topology, and (iv) prove that this limit is the weak solution. After this, we also prove the additional regularity of the continuous solution given the extra regularity conditions on given data, which is required in  later sections for  error analysis. 

\subsection{Existence and uniqueness of weak solution}\label{subsec-existence}
\begin{defn}[Weak solution]\label{def:weak}
   The triplet $(u,\theta,p)$ is a \textit{weak solution} to the problem \eqref{eq:coupled} if \eqref{init-conds} holds and 
   \begin{subequations}\label{new-reg-weak}
  \begin{align}
 & u \in C([0,T];H^1(\Omega))\cap L^{\infty}(0,T;H^2_0(\Omega)),
\;u_t  \in C([0,T];L^2(\Omega))\cap L^\infty(0,T;H^1_0(\Omega)),\label{w-reg1}\\
 &\qquad\qquad\qquad\qquad\qquad\theta , p \in C([0,T];L^2(\Omega))\cap L^2(0,T;H^1_0(\Omega)),\label{w-reg2}
  \end{align}\end{subequations}
   satisfy the relations
   \begin{subequations}
   \label{weak_a}
   \begin{align}
       &\int_0^T \big[- (u_t,v^t_t)-a_0(\nabla u_t, \nabla v^t_t)+d_0(\nabla^2 u,\nabla^2 v^t)-\alpha (\nabla \theta, \nabla v^t)-\beta (\nabla p, \nabla v^t)\big]\dt\nonumber\\
       &\qquad =\int_0^T (f,v^t) \dt+ (u^{*0},v^t(0))+a_0(\nabla u^{*0}, \nabla v^t(0)), \label{u}\\
    &\int_0^T \big[-a_1(\theta, \psi^t_t)+\gamma (p, \psi^t_t)+b_1(\theta,\psi^t) +c_1 (\nabla \theta,\nabla \psi^t) +\alpha (\nabla u_t, \nabla \psi^t)\big]\dt\nonumber\\
    &\qquad=\int_0^T(\phi,\psi^t)\dt +a_1(\theta^0, \psi^t(0))-\gamma (p^0, \psi^t(0)), \label{theta} \\
     &\int_0^T \big[-a_2 (p,q^t_t)+\gamma (\theta,q^t_t)+\kappa (\nabla p,\nabla q^t)+\beta (\nabla u_t,\nabla q^t)\big]\dt\nonumber\\
     &\qquad=\int_0^T(g,q^t) \dt+a_2 (p^0,q^t(0))-\gamma (\theta^0,q^t(0)) ,\label{p}
   \end{align}\end{subequations}
  for any $v^t \in C^2([0,T];H^2_0(\Omega))$ and both $\psi^t,q^t \in  C^1([0,T];H^1_0(\Omega)).$
\end{defn}

\noindent For any $u \in H^2_0(\Omega)$, $\theta  \in H^1_0(\Omega)$, and $p  \in H^1_0(\Omega)$, motivated by \eqref{weak_a} and an appropriate choice of the test functions, we define the \text{system energy} at any time $ 0 \le t \le T$ ~by
\begin{align}
 E(u,\theta,p;t) &:=\frac{1}{2}\big(\norm{u_{t}}^2+a_0\norm{\nabla u_{t}}^2+ d_0\norm{\nabla^2 u}^2+(a_1-|\gamma|/\gamma_0)\norm{\theta}^2\nonumber\\
    &\qquad +(a_2-|\gamma|\gamma_0)\norm{p}^2\big) +\int_0^t \big(b_1\norm{\theta}^2+c_1\norm{\nabla \theta}^2+\kappa \norm{\nabla p}^2 \big) \ds. \label{energy}
\end{align}

\noindent The following result states existence and uniqueness of solution to  \eqref{eq:coupled} in the sense of Definition 2.1 and establishes the boundedness of the energy  \eqref{energy}. The proof is  based on the approach outlined in p. 384 of Ref. \refcite{evans} (for second-order problems), and extended for coupled fourth- and second-order problems. Details are provided in   \ref{appendix}. 

\begin{thm}[Existence and uniqueness]\label{existence;thm}
    Let $f, \phi,g \in L^2(0,T;L^2(\Omega))$,  $u^0 \in H^2_0(\Omega)$, $u^{*0} \in H^1_0(\Omega)$,   and both $\theta^0,p^0 \in L^2(\Omega) $. Then, problem \eqref{eq:coupled} has a unique weak solution $(u,\theta,p)$ in the sense of Definition \ref{def:weak} and the solution satisfies
   \begin{align}
    \underset{t\in [0,T]}{\rm {ess \;sup}\;} E(u,\theta,p;t)
& \lesssim \norm{u^{*0}}^2+a_0\norm{u^{*0}}^2_{H^1(\Omega)}+d_0\norm{u^0}^2_{H^2(\Omega)}\nonumber\\
& \quad
+(a_1+|\gamma|/\gamma_0)\norm{\theta^0}^2 +(a_2+|\gamma|\gamma_0)\norm{p^0}^2+\norm{f}_{L^2(0,T;L^2(\Omega))}^2\nonumber \\
&\quad +\norm{\phi}^2_{L^2(0,T;L^2(\Omega))}+\norm{g}^2_{L^2(0,T;L^2(\Omega))}.\label{reg-ut}
\end{align}
\end{thm}
Next we present an alternate weak formulation under higher regularity assumptions on the initial data. This formulation is utilized later on, to design the fully discrete scheme.
\begin{thm}[Alternate weak formulation]\label{weak-sol-thm}
If $f, \phi, g \in H^1(0,T;L^2(\Omega))$,  $u^0 \in H^3(\Omega) \cap H^2_0(\Omega), \: u^{*0} \in H^2(\Omega) \cap H^1_0(\Omega)$,  and both $\theta^0,p^0 \in H^2(\Omega) \cap H^1_0(\Omega) $, then for the weak solution $(u,\theta,p) $, we have that 
   \begin{align}
\underset{t\in [0,T]}{\rm {ess \;sup}\;}E(u_t,\theta_t,p_t;t)  \text{ is bounded.} \label{reg-utt}
\end{align}
Furthermore, for any $0 \le t \le T$, the tuple $(u,\theta,p)$ satisfies
\begin{subequations}\label{weak form}
\begin{align}
    (u_{tt},v)+a_0 (\nabla u_{tt},\nabla v)+d_0(\nabla ^2 u, \nabla^2 v) -\alpha (\nabla \theta, \nabla v)-\beta (\nabla p, \nabla v)&=(f,v) \nonumber \\
    \quad  \forall\; v \in H^2_0(\Omega),& \label{u-weak}\\
    a_1(\theta_t, \psi)-\gamma (p_t, \psi)+b_1(\theta,\psi) +c_1 (\nabla \theta,\nabla \psi) +\alpha (\nabla u_t, \nabla \psi)&=(\phi,\psi)  \nonumber \\ \quad \forall\; \psi \in H^1_0(\Omega), &\label{theta-weak} \\
    a_2 (p_t,q)-\gamma (\theta_t,q)+\kappa (\nabla p,\nabla q)+\beta (\nabla u_t,\nabla q)&=(g,q) \nonumber \\  \quad \forall\; q \in H^1_0(\Omega).&\label{p-weak}
    \end{align}\end{subequations}    
\end{thm}
\begin{proof}
   Given  that $f_t, \phi_t, g_t \in L^2(0,T;L^2(\Omega))$, following Step 1 of the proof of Theorem~\ref{existence;thm} presented in  \ref{appendix}, note that, $({\rm d}_m^1(t), {\rm d}_m^2(t), \cdots, {\rm d}_m^m(t))$ (respectively, $(\eta_m^1(t), \eta_m^2(t), \cdots, \eta_m^m(t))$ and $({\rm l}_m^1(t),
{\rm l}_m^2(t), \cdots,$ ${\rm l}_m^m(t))$), are $C^3$ (resp. $C^2 $) functions and
 satisfy \eqref{weak-mo}-\eqref{weak-m} for $0\le t \le T$. Next, we  differentiate $\eqref{weak-m}$ with respect to $t$, and multiply the resulting equations by ${{\rm d}^k_m}''(t)$, ${\eta^k_m}'(t)$, and ${{\rm l}^k_m}'(t)$, respectively. Summing over $k = 1, 2, \ldots, m$ (for all three equations), readily yields
    \begin{align*}
    \frac{1}{2}&\frac{d}{dt}\big(\norm{u^m_{tt}}^2+a_0\norm{\nabla u_{tt}^m}^2+d_0 \norm{\nabla^2 u_t^m}^2+a_1\norm{\theta^m_t}^2+a_2\norm{p^m_t}^2\big)\\
    &\quad +b_1\norm{\theta^m_t}^2+c_1\norm{\nabla \theta^m_t}^2+\kappa\norm{\nabla p^m_t}^2-\gamma \frac{d}{dt}(\theta ^m_t,p^m_t)\\
    & =(f_t,u^m_{tt})+(\phi_t, \theta^m_t)+(g_t,p^m_t).
\end{align*}
Then, integrating from $0$ to $t$, and using the Cauchy--Schwarz, Young, and Gronwall's inequalities (similar to the proof of existence in Theorem~\ref{existence;thm} in  \ref{appendix}), we arrive at 
\begin{align}
E(u_{mt},\theta_{mt},p_{mt};t) &\lesssim \norm{u^m_{tt}(0)}^2+a_0\norm{\nabla u_{tt}^m(0)}^2+ d_0\norm{\nabla^2 u_t^m(0)}^2 \nonumber \\
&\quad +(a_1+|\gamma|/\gamma_0)\norm{\theta^m_t(0)}^2
+(a_2+|\gamma|\gamma_0)\norm{p^m_t(0)}^2 \label{uttm} \\
&\quad +\norm{f_{t}}^2_{L^2(0,T;L^2(\Omega))}+\norm{\phi_{t}}^2_{L^2(0,T;L^2(\Omega))}+\norm{g_{t}}^2_{L^2(0,T;L^2(\Omega))}. \nonumber
\end{align}
We now wish to bound the right-hand side of \rev{\eqref{uttm}} by known data. Multiply the equations \eqref{um}, \eqref{thetam}, and \eqref{pm} by ${{\rm d}_m^k}''(t),{\eta_m^k}'(t)$ and ${{\rm l}_m^k}'(t)$, respectively.
Sum up the resulting equations of  the system for $k=1,2,\cdots, m$ and  $t=0$,  and utilize the definitions in \eqref{rep} to obtain  
    \begin{align*}
 (u^m_{tt}(0),u^m_{tt}(0))+a_0 (\nabla u^m_{tt}(0),\nabla u^m_{tt}(0)) -d_0( \nabla \Delta u^m(0), \nabla u^m_{tt}(0)) & \\
+\alpha (\Delta \theta^m(0), u^m_{tt}(0))  +\beta (\Delta p^m(0),  u^m_{tt}(0))&=(f(0),u^m_{tt}(0)),  \\
     a_1(\theta^m_{t}(0), \theta^m_{t}(0))-\gamma (p^m_{t}(0), \theta^m_{t}(0)) 
     +b_1(\theta^m(0),\theta^m_{t}(0))  &\\
  -c_1 (\Delta \theta^m(0), \theta^m_{t}(0))    -\alpha (\Delta u^m_{t}(0), \theta^m_{t}(0))&=(\phi(0),\theta^m_{t}(0)),  \\
    a_2 (p^m_{t}(0),p^m_{t}(0))-\gamma (\theta^m_{t}(0),p^m_{t}(0))-\kappa (\Delta p^m(0), p^m_{t}(0))&\\
    -\beta (\Delta u^m_{t}(0),p^m_{t}(0))&=(g(0),p^m_{t}(0)) , 
    \end{align*}  
    where in the last step we have also used integration by parts and the fact that  $u^m_{tt}(0) \in H^2_0(\Omega)$, $\theta^m_{t}(0) \in H^1_0(\Omega)$ and $p^m_{t}(0) \in H^1_0(\Omega)$. 

    Next, we apply once more  Cauchy--Schwarz and Young's inequalities together with some elementary manipulation, which gives the following bounds  
    \begin{align}
    &   \norm{u^m_{tt}(0)}^2+a_0 \norm{\nabla u^m_{tt}(0)}^2 \nonumber \\ &\quad \le \frac{{d_0}^2}{a_0}\norm{ \nabla \Delta u^m(0)}^2 +3{\alpha^2}\norm{\Delta  \theta^m(0)}^2 +3\beta^2\norm{\Delta p^m(0)}^2+3\norm{f(0)}^2 , \nonumber
       \\
      & (a_1-\frac{|\gamma|}{\gamma_0})\norm{\theta ^m_{t}(0)}^2+ (a_2-|\gamma|\gamma_0)\norm{p^m_{t}(0)}^2 \nonumber\\
      &
        \quad  \le \frac{4}{a_1-\frac{|\gamma|}{\gamma_0}}\big(\alpha^2\norm{ \Delta u ^m_{t}(0)}^2+b_1^2\norm{\theta^m(0)}^2+c_1^2\norm{\Delta \theta^m(0)}^2+\norm{\phi(0)}\big)\nonumber\\
         & \qquad
         +\frac{3}{a_2-|\gamma|\gamma_0}\big(\beta^2\norm{ \Delta u ^m_{t}(0)}^2+\kappa^2\norm{ \Delta p ^m(0)}^2 +\norm{g(0)}^2\big) . \label{ptt0}
    \end{align} 
Then, \eqref{uttm}-\eqref{ptt0} \rev{lead} to
\begin{align}
E(u_{mt},\theta_{mt},q_m;t) & \lesssim d_0\norm{u^{*0}}_{H^2(\Omega)}+({d_0}^2/a_0)\norm{u^0}^2_{H^3(\Omega)}+3{\alpha^2}\norm{\theta^0}^2_{H^2(\Omega)}\nonumber \\
&\  +3\beta^2\norm{p^0}^2_{H^2(\Omega)}+3\norm{f(0)}^2
 +\Big(\frac{a_1+|\gamma|/\gamma_0}{a_1-|\gamma|/\gamma_0}+\frac{a_2+|\gamma|\gamma_0}{a_2-|\gamma|\gamma_0}\Big)
 \nonumber \\
 & \  \times \Big[
 \frac{4}{a_1-\frac{|\gamma|}{\gamma_0}}\big(\alpha^2\norm{ u^{*0}}^2_{H^2(\Omega)}+b_1^2\norm{\theta^0}^2+c_1^2\norm{\theta^0}^2_{H^2(\Omega)}+\norm{\phi(0)}\big)\nonumber\\
 &\ 
 +\frac{3}{a_2-|\gamma|\gamma_0}\big(\beta^2\norm{  u^{*0}}^2_{H^2(\Omega)}+\kappa^2\norm{p^0}^2_{H^2(\Omega)} +\norm{g(0)}^2\big) \nonumber \\
&\ +\norm{f_{t}}^2_{L^2(0,T;L^2(\Omega))}+\norm{\phi_{t}}^2_{L^2(0,T;L^2(\Omega))}+\norm{g_{t}}^2_{L^2(0,T;L^2(\Omega))}\Big]. \label{int3}
\end{align}
And this, in combination with \eqref{c1} and \eqref{c2}, readily implies  that  
\begin{align*}
    ( u^{m}_t,u_{tt}^{m},\theta_t^{m},p_t^{m})  &\xrightarrow{\text{weak*}}(u_t,u_{tt},\theta_t,p_t) \quad  \text{in}\ L^\infty\big(0,T;H_0^2(\Omega) 
 \times H^1_0(\Omega) \times (L^2(\Omega) )^2\big), 
 \\
   (\theta_t^{m}, p_t^{m})  &\xrightarrow{\text{weak}} (\theta_t,p_t) \quad \qquad \; \;  \text{in}\ L^2 \big(0,T; (H_0^1(\Omega))^2\big).
\end{align*}
Finally, the bounds in \eqref{reg-utt} are established by taking the limit $m\to \infty$   in \eqref{int3}. To confirm that $(u,\theta,p)$ satisfies \eqref{u-weak}-\eqref{p-weak}, we proceed as in the uniqueness proof of Theorem~\ref{existence;thm}, to obtain \eqref{uep}-\eqref{pep}, but with $(f,v)$, $(\phi,\psi)$, and $(g,q)$ representing the respective right-hand sides, after which we take the limit as $\varepsilon \rightarrow 0$.
\end{proof}

\subsection{Additional regularity}\label{sec:reg}
The next theorem establishes a priori bounds of the solution and its higher-order time derivatives, provided that the initial and source data are sufficiently smooth. While the specific approach followed here is relatively standard (see, e.g.,  Ref. \refcite{Rivera} and the references therein), its adaptation to the present model is novel. A summary of the regularity results is displayed in Table~\ref{table:summary}.

\medskip 
\noindent
\textbf{Regularity estimate.} It is well known\cite{AGMON,Rannachar} that if $\Phi^* \in H^{-r}(\Omega
)$ (resp. $F^* \in H^{-s}(\Omega
)$) are such that $-\Delta \chi   =\Phi^*$ (resp.   $\Delta^2 w  =F^*$) then 
\begin{align}
 \norm{\chi}_{H^{2-r}(\Omega)} \le C_{\text{reg}}(r) \norm{\Phi^*}_{H^{-r}(\Omega)}\; (\text{ resp.} \norm{w}_{H^{4-s}(\Omega)} \le C_{\text{reg}}(s) \norm{F^*}_{H^{-s}(\Omega)}),\label{ellip-reg}
\end{align}
for all $1-\sigma^1_{\text{reg}} \le r \le 1$ (resp. $2-\sigma^2_{\text{reg}} \le s \le 2$),  where $\sigma_{\text{reg}}^1>0$ (resp. $\sigma_{\text{reg}}^2>0$), is the elliptic  regularity index of the Laplace  (resp. biharmonic) operator, and the constants $C_{\text{reg}}(r)$  (resp. $C_{\text{reg}}(s)$) depend only on $\Omega$ and $s$ (resp. $r$). Lowest-order \ac{fe} schemes typically  achieve at most linear convergence in energy norm, so it is reasonable to assume throughout the paper that $\sigma= \text{min}\{1,\sigma_{\text{reg}}^1,\sigma_{\text{reg}}^2\}$, whence $0 < \sigma \le 1$. Note that if $\Omega$ is a convex polygon, then $\sigma =1$, whereas for non-convex polygons we have $1/2< \sigma < 1$. The elliptic regularity index 
$\sigma$ plays an important role in determining the rate of convergence presented in Section~\ref{sec:fully}. 
 We are now in a position to state the regularity of the weak solution. From the estimates  \eqref{reg-utt}, we can write (also using \eqref{model22}-\eqref{p2;model33}):
\begin{align*}
   -c_1 \Delta \theta &=\phi - a_1\theta_t+\gamma p_t-b_1\theta+\alpha \Delta u_t:=\Phi^* \in L^2(\Omega),\; \qquad\quad\text{ for all } 0\le t \le T ,\nonumber\\
   -\kappa \Delta p&=g  -a_2 p_t+\gamma \theta_t +\beta \Delta  u_t:=G^* \in L^2(\Omega),\qquad\;\;\quad\qquad\;\text{ for all } 0\le t \le T,\nonumber\\
   d_0\Delta^2 u&= f- u_{tt}+a_0 \Delta u_{tt} -\frac{\alpha}{c_1}\Phi^*-\frac{\beta}{\kappa}G^*:=F^* \in H^{-1}(\Omega),\quad \text{ for all } 0\le t \le T ,
\end{align*}
where in the last equation we have used the fact that $\norm{\nabla u_{tt}} $ is bounded (cf. \eqref{reg-utt}), and hence $a_0 \Delta u_{tt} \in H^{-1}(\Omega)$, whence $F^* \in H^{-1}(\Omega).$ Then we utilize \eqref{ellip-reg} to see that, for all $1/2< \sigma \le 1$, there holds 
\begin{equation}
    u \in L^\infty(0,T;H^{2+\sigma}(\Omega)) \text{ and } \theta , p  \in L^\infty(0,T;H^{1+\sigma}(\Omega)). \label{ell-reg-sol}
\end{equation}
The next theorem guarantees higher regularity of weak solution (needed for the error estimates in   Section~\ref{p2;fully-error-sec}). The proof is based on the arguments used in    Prop. 2.5.2 of Ref. \refcite{Lasiecka}, and details are postponed to  \ref{appendix}.

\begin{thm}[Regularity]\label{thm;regularity}
(a)  Let $f,\phi , g \in H^2(0,T;L^2(\Omega))$
, $u^0,u^{*0}\in H^3(\Omega) \cap H^2_0(\Omega),  \theta^{0}\in H^2(\Omega) \cap H^1_0(\Omega), \text{ and }p^{0}\in H^2(\Omega) \cap H^1_0(\Omega).$
Assume that the compatibility conditions 
\begin{subequations}\label{compat1}
\begin{align}
 & u_{tt}(0)-a_0 \Delta u_{tt}(0)=f(0)-d_0\Delta^2 u^0-\alpha\Delta \theta^0-\beta \Delta p^0,\\
    &  a_1 \theta_{t}(0)  -\gamma p_{t}(0)=\phi(0)-b_1 \theta^0 +c_1 \Delta \theta^0 +\alpha \Delta u^{*0},  \\
   &    a_2  p_{t}(0)-\gamma \theta_{t}(0)=g(0) +\kappa \Delta p^0+\beta \Delta  u^{*0},
\end{align} 
    \end{subequations}
hold and  $(u_{tt}(0), \theta_t(0) , p_t(0) )$ 
 belongs to $
   (H^2(\Omega) \cap H^1_0(\Omega))^3.$
Then,  
\begin{equation}\label{reg-uttt}
 \underset{t\in [0,T]}{\rm {ess \;sup}}\big(\norm{u_{t}}^2_{H^{2+\sigma}(\Omega)}+\norm{\theta_{t}}^2_{H^{1+\sigma}(\Omega)}+\norm{p_{t}}^2_{H^{1+\sigma}(\Omega)}+E(u_{tt},\theta_{tt},p_{tt};t)\big)\end{equation}
 is bounded. 

\medskip  
\noindent (b) Let $f,\phi , g \in  H^3(0,T;L^2(\Omega)$, $u^0,u^{*0}, u_{tt}(0) \in H^3(\Omega) \cap H^2_0(\Omega)$,  $\theta^{0},\theta_t{(0)},p^{0},p_t(0)\in H^2(\Omega) \cap H^1_0(\Omega)$.  
Assume that  the compatibility conditions 
\begin{subequations}\label{compat2}
\begin{align}
  u_{ttt}(0)-a_0 \Delta u_{ttt}(0)&=f_t(0)-d_0\Delta^2 u_t(0)-\alpha\Delta \theta_t(0)-\beta \Delta p_t(0),\\
      a_1 \theta_{tt}(0)  -\gamma p_{tt}(0)&=\phi_{t}(0)-b_1 \theta_t(0) +c_1 \Delta \theta_t(0) +\alpha \Delta u_{tt}(0),  \\
      a_2  p_{tt}(0)-\gamma \theta_{tt}(0)&=g_t(0) +\kappa \Delta p_t(0)+\beta \Delta  u_{tt}(0),
\end{align}
\end{subequations}
hold and $(u_{ttt}(0), \theta_{tt}(0) , p_{tt}(0) )$ belongs to $
   (H^2(\Omega) \cap H^1_0(\Omega))^3$.  Then, 
\[ \underset{t\in [0,T]}{\rm{ess \; sup}}\big(\norm{u_{tt}}^2_{H^{2+\sigma}(\Omega)}+\norm{\theta_{tt}}^2_{H^{1+\sigma}(\Omega)}+\norm{p_{tt}}^2_{H^{1+\sigma}(\Omega)}+E(u_{ttt},\theta_{ttt},p_{ttt};t)\big)\]    is bounded.  
\end{thm}

\begin{table}[ht]
\renewcommand{\arraystretch}{1.4}
\scriptsize
\centering
\begin{tabular}{|p{0.18\textwidth}|p{0.32\textwidth}|p{0.37\textwidth}|}
\hline
\textit{Description} & \textit{Assumptions on data} & \textit{Conclusions} \\
\hline
\textbf{Theorem~\ref{existence;thm}}

(Existence, Uniqueness, \& Energy bound for solution) 
& 
$u^0 \in H^2_0(\Omega),\ u^{*0} \in H^1_0(\Omega)$ \newline
$\theta^0, p^0 \in L^2(\Omega)$ \newline
$f(t), \phi(t), g(t) \in L^2(0,T;L^2(\Omega))$
& 
$u \in L^\infty(0,T;H^2_0(\Omega))$ \newline
$u_t \in L^\infty(0,T;H^1_0(\Omega))$ \newline
$\theta,\, p \in L^\infty(0,T;L^2(\Omega))$\newline 
\hphantom{$\int_X\ \int_X$} $ \cap L^2(0,T;H^1_0(\Omega))$ \\[0.5ex]
\hline
\textbf{Theorem~\ref{weak-sol-thm}} 

(Weak formulation \eqref{weak form} \& Energy bound for time derivative of the solution) 
& 
$u^0 \in H^3(\Omega) \cap H^2_0(\Omega)$ \newline
$u^{*0}, \theta^0, p^0 \in H^2(\Omega) \cap H^1_0(\Omega)$ \newline
$f(t), \phi(t), g(t) \in H^1(0,T;L^2(\Omega))$ 
& 
$u \in L^\infty(0,T;H^{2+\sigma}(\Omega) \cap H^2_0(\Omega))$ \newline
$u_t \in L^\infty(0,T;H^2_0(\Omega))$ \newline 
$u_{tt} \in L^\infty(0,T;H^1_0(\Omega))$ \newline
$\theta,\, p \in L^\infty(0,T;H^{1+\sigma}(\Omega) \cap H^1_0(\Omega))$ \newline
$\theta_t,\, p_t \in L^\infty(0,T;L^2(\Omega)) $\newline 
\hphantom{$\int_X\ \int_X$} $\cap L^2(0,T;H^1_0(\Omega))$ \\[0.5ex]
\hline
\textbf{Theorem~\ref{thm;regularity}(a)} 

(Additional regularity of solution \& Energy bound for second-order time derivative) 
& 
$u^0,\, u^{*0} \in H^3(\Omega) \cap H^2_0(\Omega)$ \newline
$\theta(0),\, p(0) \in H^2(\Omega) \cap H^1_0(\Omega)$ \newline
$u_{tt}(0),\, \theta_t(0),\, p_t(0) \in H^2(\Omega) \cap H^1_0(\Omega)$ \newline
$f(t), \phi(t), g(t) \in H^2(0,T;L^2(\Omega))$
\newline
\eqref{compat1} holds
& 
$u_t \in L^\infty(0,T;H^{2+\sigma}(\Omega) \cap H^2_0(\Omega))$ \newline
$u_{tt} \in L^\infty(0,T;H^2_0(\Omega))$\newline 
$u_{ttt} \in L^\infty(0,T;H^1_0(\Omega))$ \newline
$\theta_t,\, p_t \in L^\infty(0,T;H^{1+\sigma}(\Omega) \cap H^1_0(\Omega))$ \newline
$\theta_{tt},\, p_{tt} \in L^\infty(0,T;L^2(\Omega)) $\newline 
\hphantom{$\int_X\ \int_X$} $\cap L^2(0,T;H^1_0(\Omega))$
\\
[0.5ex]
\hline
\textbf{Theorem~\ref{thm;regularity}(b)} 

(Energy bound for third-order time derivative. Sufficient conditions for error analysis in Lemma~\ref{initial-error-lemm} \& Theorem~\ref{error-estimates thm}) 
& 
$u^0, u^{*0}, u_{tt}(0) \in H^3(\Omega) \cap H^2_0(\Omega)$ \newline
$u_{ttt}(0), \theta^0, \theta_t(0), p^0, p_t(0),$ \newline
$\theta_{tt}(0), p_{tt}(0) \in H^2(\Omega) \cap H^1_0(\Omega)$ \newline
$f(t), \phi(t), g(t) \in H^3(0,T;L^2(\Omega))$
\newline
\eqref{compat2} holds
& 
$u_{tt} \in L^\infty(0,T;H^{2+\sigma}(\Omega) \cap H^2_0(\Omega))$ \newline
$u_{ttt} \in L^\infty(0,T;H^2_0(\Omega))$\newline 
$u_{tttt} \in L^\infty(0,T;H^1_0(\Omega))$ \newline
$\theta_{tt},\, p_{tt} \in L^\infty(0,T;H^{1+\sigma}(\Omega) \cap H^1_0(\Omega))$ \newline
$\theta_{ttt},\, p_{ttt} \in L^\infty(0,T;L^2(\Omega)) $\newline 
\hphantom{$\int_X\ \int_X$} $\cap L^2(0,T;H^1_0(\Omega))$ \\[0.5ex]
\hline
\end{tabular}
\caption{Summary of regularity assumptions and corresponding results.}
\label{table:summary}
\end{table}

\begin{rem}
In accordance with the above regularity result, if we define 
\begin{equation}
F(t,\bx):=f(t,\bx)-u_{tt}+a_0 \Delta u_{tt}-\alpha \Delta \theta-\beta \Delta p,\label{F}
\end{equation}
then, there exist positive constants $C_F$ and $C_F'$, such that
\begin{equation}
(i) \; \norm{F}_{L^\infty(0,T;L^2(\Omega)} \le C_F  \quad \text{ and } \quad (ii) \; \norm{F_t}_{L^2(0,T;L^2(\Omega))} \le C_F'. \label{norm-F}
\end{equation}
\end{rem}

\section{Fully discrete scheme and stability}\label{sec:semi}
{This section develops the numerical framework for the coupled hyperbolic-parabolic system \eqref{p2;model11}-\eqref{p2;model33}. Subsection~\ref{P1 semi_schemes} introduces the \ac{fe} spaces and projection operators, highlighting the need for a modified Ritz projection for the displacement variable. Subsection~\ref{sec:fully} presents the {\it first} fully discrete scheme with explicit initial error estimates, contrasting with previous works that begin at the second time step. Subsection~\ref{p2;stability_sec} establishes the unconditional stability of the scheme.}
\subsection{Space discretization}\label{P1 semi_schemes}
We now define the \ac{fe} spaces and projection operators, and highlight their approximation properties. Additionally, we discuss the necessity for a \textit{modified} Ritz projection, specifically for the displacement variable.

Let $K \in \mathcal{T}$ be any triangle in the shape-regular triangulation $\mathcal{T}$ of $\bar{\Omega}$. We denote its diameter by $h_K$, its area by $|K|$, and use $\bn_K$ to refer to the outward unit normal vector on $\partial K$. Define $h := \max_{K \in \mathcal{T}} h_K$. The sets of interior and boundary vertices of $\mathcal{T}$ are denoted by $\mathcal{V}(\Omega)$ and $\mathcal{V}(\partial \Omega)$, respectively, with the combined set represented as $\mathcal{V} = \mathcal{V}(\Omega) \cup \mathcal{V}(\partial \Omega)$. Similarly, we use $\mathcal{E}(\Omega)$ and $\mathcal{E}(\partial \Omega)$ for the sets of interior and boundary edges, and write $\mathcal{E} = \mathcal{E}(\Omega) \cup \mathcal{E}(\partial \Omega)$. For any edge $e \in \mathcal{E}$, the corresponding edge patch $\omega(e)$ is defined as $\text{int}(K_+ \cup K_-)$ if $e = \partial K_+ \cap \partial K_- \in \mathcal{E}(\Omega)$, and $\text{int}(K)$ when $e \in \mathcal{E}(\partial \Omega)$. Consider two neighboring triangles, $K_+$ and $K_-$, with the unit normal vector along $e$ satisfying $\bn_{K_+}|_{e} = \bn|_{e} = -\bn_{K_-}|{_e}$, directed outward from $K_+$ towards $K_-$. The jump of a function $\varphi$, written as $ \big[\!\!\big[\varphi\big]\!\!\big]$, is defined by $  \varphi |_{K_{+}}-  \varphi |_{K_{-}}\text{ if } e=\partial K_+ \cap \partial K_-  \in {\cal E}(\Omega)$ and $\varphi|_e  \text{ if } e \in {\cal E}(\partial \Omega)$. The average $\big\{\!\!\!\big\{\varphi\big\}\!\!\!\big\}$ is defined by $\frac{1}{2} (\varphi |_{K_{+}}+  \varphi|_{K_{-}} )\text{ if } e=\partial K_+ \cap \partial K_-  \in {\cal E}(\Omega)$ and  $\varphi|_e  \text{ if } e \in {\cal E}(\partial \Omega)$.

Let $V_h:=\mathcal{P}_2({\cal T}) \cap H^1_0(\Omega) \subset {\mathcal H}^2({\cal T})$ and $W_h:=\mathcal{P}_1({\cal T}) \cap H^1_0(\Omega) \subset  H^1_0(\Omega) $ be  finite-dimensional subspaces and define the bilinear form $a_h(\bullet, \bullet):V_h \times V_h \rightarrow \mathbb{R}$ by
\begin{align*}
     a_h(w_h,v_h) &:=\int_\Omega D^2_{{\rm pw}}w_h:D^2_{\rm{pw}}v_h \dx-\sum_{e \in {\cal E}} \int_e \big[\!\!\big[\nabla w_h\big]\!\!\big] \cdot \big\{\!\!\!\big\{D^2_{\rm{pw}}v_h\big\}\!\!\!\big\}{\bn }\ds\\
   &\quad
   -\sum_{e \in {\cal E}} \int_e \big[\!\!\big[\nabla v_h\big]\!\!\big] \cdot \big\{\!\!\!\big\{D^2_{\text{\rm pw}}w_h\big\}\!\!\!\big\}\bn \ds+\sum_{e \in {\cal E}} \frac{\sigma_{\rm{IP}}}{h_e} \int_e  \Big[\!\!\Big[\frac{\partial w_h}{{\partial {\bn}}} \Big]\!\!\Big] \Big[\!\!\Big[\frac{\partial v_h}{{\partial {\bn}}} \Big]\!\!\Big]\ds,
\end{align*}
with respect to a mesh-dependent (broken) norm on $V_h$ defined by 
\begin{equation*}
\norm{v_h}_{h}^2:= \norm{D^2_{{\rm pw}}v_h}^2+\sum_{e \in {\cal E}} \frac{\sigma_{\rm{IP}}}{h_e} \int_e  \Big[\!\!\Big[\frac{\partial v_h}{{\partial {\bn}}} \Big]\!\!\Big]^2 \ds, 
\end{equation*}
where, $D^2_{\rm{pw}}$ is the piecewise Hessian and  the penalty parameter $\sigma_{\rm{IP}}>0$ is chosen sufficiently large\cite{Brener-sung}. 

It is well-known that  $a_h(\bullet ,\bullet)$ is  symmetric, continuous, and elliptic, i.e., there exist $ C_{\rm{Coer}}, C_{\rm{Cont}}>0$ such that for all $w_h, v_h \in V_h$ (see, for e.g.,  Ref. \refcite{MR3371900})
\begin{gather}
    a_h(w_h,v_h)=a_h(v_h,w_h), \quad C_{\rm{Coer}} \|w_h\|_h^2 \le a_h(w_h,w_h), \nonumber\\ a_h(w_h,v_h) \le C_{\rm{Cont}} \|w_h\|_h\|v_h\|_h.\label{P1 a_h_properties}
\end{gather}  
\noindent
The nonconforming Morley FE space\cite{CC} is defined as follows:
\begin{align*}
{\rm M}({\cal T})&:=\{ {v_{\rm M}} \in \rev{\mathcal{P}_2}({\cal T}): 
 {v_{\rm M}} \text{{\small is continuous at interior vertices and its normal derivatives}}\\
&\text{{\small are continuous at the midpoints of interior edges,} }{v_{\rm M}} \text{{\small  vanishes at the vertices of }}\partial \Omega \\
&\text{{\small  and its normal derivatives vanish at the midpoints of boundary edges of } $\partial \Omega$}\}.
\end{align*}
\begin{defn}[Morley interpolation\cite{CC}]\label{bhmorley_lem}
  For all $v_{\rm{pw}}\in H^2(\mathcal T)$, the extended Morley interpolation operator $I_{\rm M}:H^2({\mathcal T} ) \rightarrow {\rm M}({\mathcal T} )$ is defined by 
\[ 
 (I_{\rm {M}}v_{\rm{pw}})(z):= 
 |\mathcal T(z)|^{-1}
\sum_{K \in {\mathcal T}(z)} (v_{\rm{pw}}|_K)(z)\ \text{and}\ \displaystyle \fint_e\frac{\partial (I_{\rm M} v_{\rm{pw}})}{\partial {\bn}} \,\ds :=  \fint_e \Big\{\!\!\!\Big\{\frac{\partial v_{\rm{pw}}}{\partial {\bn}}\Big\}\!\!\!\Big\} \ds .\]
In case of an interior vertex $z$, ${\mathcal T}(z)$ represents the collection of attached triangles, and $|{\mathcal T}(z)|$ indicates the number of such triangles connected to vertex $z$.  
\end{defn}
\begin{lem}[Companion operator and properties\cite{carsput2020,CC}]\label{bhcompanion_lem} 
Let ${\rm HCT}(\mathcal{T})$  denote the Hsieh--Clough--Tocher  element.
There exists a linear mapping $J: {\rm M}(\mathcal{T})\to ({\rm HCT}(\mathcal{T})+\mathcal{P}_8(\mathcal{T})) \cap H^2_0(\Omega)$ such that any $w_{\rm M}\in {\rm M}(\mathcal{T})$ satisfies
\begin{align*}
&\text{(i) } Jw_{\rm M}(z)=w_{\rm M}(z) \quad \text{ for } z\in\mathcal{V}, \nonumber \\ 
&\text{(ii) } \nabla ({J}w_{\rm M})(z)=
|\mathcal{T}(z)|^{-1}\sum_{K\in\mathcal{T}(z)}(\nabla w_{\rm M}|_K)(z)
\quad \text{ for }z\in\mathcal{V}(\Omega),  \nonumber\\
&\text{(iii) } \fint_e \frac{\partial J w_{\rm M}}{{\partial {\bn}}} \ds=\fint_e \frac{ \partial w_{\rm M}}{{\partial {\bn}}} \ds \text{ for any } e\in\mathcal{E},\nonumber\\ 
&\text{(iv) }  \trinl w_{\rm M}- J w_{\rm M} \trinr_{\rm pw} \lesssim  \min_{v\in  H^2_0(\Omega)}  \trinl w_{\rm M}- v \trinr_{\rm pw}, \nonumber \\
&\text{(v) } \|{v_h -Qv_h}\|_{H^s({\cal T})} \le C_1 h^{2-s}\!\!\!\min_{v\in  H^2_0(\Omega)}\norm{v- v_h}_h\ \text{ for }v_h \in V_h,\, C_1>0, \, 0 \le s \le 2. 
\end{align*} 
Here  $Q =JI_{\rm M}$ is a smoother operator defined from $ V_h $ to $ H^2_0(\Omega)$.
\end{lem} 
\noindent
\textbf{Ritz projection operators}

 The error control associated with the fully discrete approximation employs Ritz projection operators defined from $H^2_0(\Omega)$ (resp. $H^1_0(\Omega))$ into $V_h$ (resp. $W_h$) for $u$ (resp. $\theta \text{ and } p$). It should be noted that since $V_h $ is not a subspace of $H^2_0(\Omega)$,  the standard definition 
 \begin{equation*}
     a_h({\cal R}_h w, v_h) =(\nabla^2 w, \nabla^2 v_h) \qquad \text{for all $v_h \in V_{h}$}, 
     \end{equation*}
 does not hold for   $v_h \in V_h \subset H^2({\mathcal T})$ for the nonstandard $C^0$IP scheme proposed herein. 
 
 Alternative approaches that define Ritz projections for nonstandard methods (see, e.g., Refs. \refcite{danumjaya2021morley,gudi} for the fourth-order nonlinear parabolic extended Fisher--Kolmogorov equation) often require higher regularity $u \in H^3(\Omega) \cap H^2_0(\Omega)$, which might not hold for non-convex domains. See the discussion in Section~\ref{sec:reg} for 
 non-convex polygons. Our recent work in Ref. \refcite{nry_preprint} addresses this issue by means of the \textit{modified} Ritz projection (see Definition~\ref{P1 ritz_projection} below), which utilizes a smoother operator $Q: V_h \rightarrow H^2_0(\Omega)$ defined as $JI_{\rm M}$, where $J$ (resp. $I_{\rm M}$) denotes the companion (resp. extended Morley interpolation) operator from Lemma~\ref{bhcompanion_lem} (resp. Lemma~\ref{bhmorley_lem}). 
The {\it modified Ritz projection} $\mathcal{R}_h: H^2_0(\Omega)\rightarrow V_h $  for the displacement variable is defined as follows:
\begin{equation}
    a_h({\cal R}_hw,v_h )=(\nabla ^2 w, \nabla^2 Qv_h )\qquad  \text{ for all } v_h \in V_h,\, w \in H^2_0(\Omega)\label{P1 ritz_projection}.
\end{equation}

\begin{lem}[Approximation properties for ${\cal R}_h$, Appendix in Ref. \refcite{mahata2025lowest}]\label{P1 ritz_lemma}
Let $w \in H^2_0(\Omega) \cap H^{2+\sigma}(\Omega)$, where 
$\sigma \in (1/2, 1]$, and let $\mathcal{R}_h w$ be its Ritz projection defined in \eqref{P1 ritz_projection}. Then, there exists a constant $C_2 > 0$ such that
\begin{equation}
    \|{w-\mathcal{R}_hw}\| + \|{\nabla (w-\mathcal{R}_h w)}\| +  h^{\sigma}\norm{w-\mathcal{R}_hw}_h \le C_2 h^{2\sigma}\norm{w}_{H^{2+\sigma }(\Omega)}.  \label{P1 norm_ritz}
\end{equation}
\end{lem}
\noindent
Next, we define the $H^1$-\textit{conforming elliptic projection}\cite{ern2}  $\Pi_h\hspace{-0.1cm}:H^1_0(\Omega) \rightarrow W_h$  for the first moments of temperature and pressure as:
\begin{equation}
    (\nabla \Pi_h \chi,\nabla \chi_h)=(\nabla \chi,\nabla \chi_h) \qquad\text{ for all }\chi_h \in W_h. \label{p2;ritzprojection2}
\end{equation}
\begin{lem}[Approximation properties for $\Pi_h$, Theorem~32.15 of Ref. \refcite{ern2}]\label{P1_pi_lemma}
    Let $\chi  \in H^1_0(\Omega) \cap H^{1+\sigma}(\Omega)$ for some $\sigma \in (1/2, 1]$. Then, there exists a constant $C_3 > 0$ such that
    \begin{equation}
        \|{\chi - \Pi_h \chi}\| + h^{\sigma}\|\nabla(\chi - \Pi_h \chi)\| \le C_3 h^{2\sigma}\|\chi\|_{H^{1+\sigma}(\Omega)}. \label{P1 norm_ritz2}
    \end{equation}
\end{lem}
\subsection{Fully discrete scheme}\label{sec:fully}
This subsection discusses a fully discrete scheme for   \eqref{weak form}. To the best of our knowledge, this is the first fully discrete scheme for a hyperbolic-parabolic coupled problem with explicit initial error estimates. In contrast with  Ref. \refcite{Madureira_num}, where the initial error is assumed to be bounded with the required convergence and the formulation begins from the second time   step, our approach starts from the initial step and provides a general framework for defining the fully discrete formulation for any coupled hyperbolic-parabolic system. 

For a positive integer $N$, consider the partition $ 0=t_0 < t_1<t_2< \cdots<t_N=T$ of the interval $[0,T]$
 with $t_n=n\Delta t$, and $\Delta t=T/N$ being the time step.
For any function $\upsilon(\bx,t)$, the following notations are adopted:  
\begin{align*}
& \upsilon^n  := \upsilon(\bx,t_n)= \upsilon(t_n), \quad \upsilon^{n+1/2}:= \frac{1}{2}\left(\upsilon^{n+1}+\upsilon^n \right), \\
& \upsilon^{n,1/4} :=\frac{1}{4} \left(\upsilon^{n+1}+2\upsilon^n+  \upsilon^{n-1}\right)= {\frac{1}{2} \left(v^{n+1/2} +v^{n-1/2} \right)},\\
& \bar{\partial}_t \upsilon^{n+1/2} :=\frac{\upsilon^{n+1}-\upsilon^{n}}{\Delta t} ,\quad 
{\bar{\partial}}^2_t \upsilon^n := \frac{\upsilon^{n+1}-2\upsilon^n+\upsilon^{n-1}}{(\Delta t)^2} ,\quad 
\delta_t \upsilon^n := \frac{\upsilon^{n+1}-\upsilon^{n-1}}{2\Delta t}. 
\end{align*}
 Let $(U^n,\Theta^n,P^n)=(U(t_n),\Theta(t_n),P(t_n))$ denote the approximation of the continuous solution $(u,\theta,p)$ at time $t_n$.  Considering the following approximation of the initial solution 
\begin{equation}
    (U^0,\Theta^0,P^0)=({\mathcal{R}_h}u^0,\Pi_h \theta^0,\Pi_h p^0) ,\label{U0,THETA0,P0}
    \end{equation}
    we compute $(U^1,\Theta^1,P^1) \in V_h \times W_h \times W_h$ by solving the following elliptic system for all $(v_h,\psi_h, q_h) \in V_h \times W_h \times W_h$ 
    \begin{subequations}\label{3.10}
\begin{align}
    2(\Delta t)^{-1}\big[(\bar{\partial}_tU^{1/2}-u^{*0},v_h)+a_0(\nabla \bar{\partial}_tU^{1/2}-\nabla u^{*0},\nabla v_h)\big] & \label{int_u0}\\
 +d_0 a_h(U^{1/2}, v_h)  -  \alpha( \nabla \Theta^{1/2}, \nabla v_h)-\beta( \nabla P^{1/2}, \nabla v_h) 
  & =( f^{1/2},v_h)
 \nonumber    ,\\
    a_1(\bar{\partial}_t{\Theta}^{1/2}, \psi_h)-\gamma (\bar{\partial}_t{P^{1/2}}, \psi_h)
+b_1(\Theta^{1/2},\psi_h)
   +c_1(\nabla\Theta^{1/2},\nabla\psi_h) \label{int_theta0} &\\
   +\alpha(\nabla \bar{\partial}_tU^{1/2},\nabla\psi_h) & =(\phi^{1/2},\psi_h)\nonumber  , \\
    a_2 (\bar{\partial}_tP^{1/2},q_h)-\gamma (\bar{\partial}_t{\Theta}^{1/2},q_h)+\kappa (\nabla P^{1/2},\nabla q_h)+\beta (\nabla \bar{\partial}_tU^{1/2},\nabla q_h)&=(g^{1/2},q_h) .\label{int_p0}
    \end{align}\end{subequations}
    The solution is calculated at $t_1$ using \eqref{int_u0}-\eqref{int_p0} in order to align the two numerical schemes, since the Newmark scheme \eqref{un} requires solutions at $t_0$ and $t_1$ to compute the solution at $t_2$, while the Crank--Nicolson scheme begins its computation from $t_1$.
    The idea of the discrete equation \eqref{int_u0} is based on  the discretization of  \rev{the biharmonic wave equation (see  Ref. \refcite{nry_preprint}) applied to the} coupled system, and \eqref{int_theta0}-\eqref{int_p0} are based on the Crank--Nicolson method to determine the solution at $t_1$. The construction of \eqref{3.10} guarantees quadratic convergence in time, implying that also the fully discrete scheme is quadratically convergent. 
    
\noindent For $n=1,2, \cdots,N-1$, the fully discrete scheme consists in finding $(U^{n+1},\Theta^{n+1},P^{n+1}) \in V_h \times W_h \times W_h$ such that for all $(v_h,\psi_h, q_h) \in V_h \times W_h \times W_h$
\begin{subequations}
   \label{fully-discret-scheme}
    \begin{align}
   (\bar{\partial}_t^2U^{n},v_h)+a_0 (\nabla \bar{\partial}_t^2U^{n},\nabla v_h)+d_0 a_h(U^{n,1/4}, v_h)&\label{un}\\
    -  \alpha (\nabla \Theta^{n,1/4},\nabla v_h) -\beta ( \nabla P^{n,1/4}, \nabla v_h) 
   & =( f^{n,1/4},v_h),\nonumber  \\
    a_1(\bar{\partial}_t{\Theta}^{n+1/2}, \psi_h)-\gamma (\bar{\partial}_t{P^{n+1/2}}, \psi_h)+b_1(\Theta^{n+1/2},\psi_h)&  \label{thetan} \\
+c_1 (\nabla \Theta^{n+1/2},\nabla \psi_h) +\alpha (\nabla \bar{\partial}_t U^{n+1/2},\nabla \psi_h) & =(\phi^{n+1/2},\psi_h),\nonumber \\
\label{pn}    a_2 (\bar{\partial}_tP^{n+1/2},q_h)
   -\gamma (\bar{\partial}_t{\Theta}^{n+1/2},q_h)
 +\kappa (\nabla P^{n+1/2},\nabla q_h)&\\
 +\beta (\nabla\bar{\partial}_t U^{n+1/2},\nabla q_h)&=(g^{n+1/2},q_h).\nonumber 
    \end{align}\end{subequations}
 This section and the rest of the paper uses the discrete Gronwall Lemma that is stated below.
\begin{lem}[Discrete Gronwall Lemma\cite{MR3003381}]\label{P1 d-gronwall}
 Let $\{v_n\}$, $\{w_n\}$, and $\{y_n\}$ be three non-negative sequences, with $\{y_n\}$ monotone, that satisfy 
  $\displaystyle    v_m+w_m \le y_m + \nu \sum_{n=0}^{m-1} v_n, \quad \nu >0, \ v_0+w_0 \le y_0 $. Then for $m \ge 0,$ it holds that 
  $    v_m+w_m \le y_m e^{m \nu}.$
\end{lem}
\begin{rem}[\rev{Useful} identities]
  Before proceeding further, we state the following identities, which lead to telescopic sums and are used in Theorem~\ref{p2;stability-thm} and Theorem~\ref{error-estimates thm}. For any discrete functions $Q^n \in V_h$ and $S^n \in W_h$, $n=0,1,2,\cdots,N$ there hold
\begin{subequations}
\begin{align}
    2\Delta t (\bar{\partial}_t^2Q^{n},\delta_t Q^n)&=\norm{\bar{\partial}_t Q^{n+1/2}}^2-\norm{\bar{\partial}_t Q^{n-1/2}}^2, \label{p2;stb1}\\
    2\Delta t (\nabla \bar{\partial}_t^2Q^{n},\nabla \delta_t Q^n)&=\norm{\nabla \bar{\partial}_t Q^{n+1/2}}^2-\norm{\nabla \bar{\partial}_t Q^{n-1/2}}^2,\label{p2;stb2}\\
     2\Delta t a_h(Q^{n,1/4}, \delta_t Q^n)&=a_h(Q^{n+1/2},Q^{n+1/2})-a_h(Q^{n-1/2},Q^{n-1/2}), \label{p2;stb3}\\ 
 2 \Delta t(\bar{\partial}_t S^{n+1/2},S^{n+1/2})&= \norm{S^{n+1}}^2-\norm{S^{n}}^2.
    \label{p2;stab5}
    \end{align}
    \end{subequations}
\end{rem}
\subsection{Stability}\label{p2;stability_sec}
Here we demonstrate the stability of  the fully discrete scheme {in \eqref{fully-discret-scheme}} and establish a uniform bound of the solution \( (U^{m+1},\Theta^{m+1},P^{m+1}) \) at \( t_{m+1} \) for \( 1 \le m \le N-1 \) in terms of the solution at \( t_0 \), \( t_1 \), and the load/source functions. 
 For any $\chi_h^n, Q_h^n \in W_h  ; n\in\{1,2,\cdots,m\}$ with $1 \le m \le N-1$, define
\begin{equation}
   \norm{(\chi_h^{m},Q_h^{m})}^2_{H}:= \Delta t \sum_{n=1}^m\Big[ {b_1}\norm{\chi_h^{n+1/2}}^2
   +c_1 \norm{\nabla\chi_h^{n+1/2}}^2 +{\kappa}\norm{\nabla Q_h^{n+1/2}}^2\Big].\label{t-norm}
\end{equation}
 Also, we define
\begin{align*}
    ((U^0,U^1,f))&:=6\norm{\bar{\partial}_t U^{1/2}}^2+4a_0\norm{\nabla \bar{\partial}_t U^{1/2}}^2+4d_0C_{\rm{Cont}}\norm{U^{1/2}}_h^2\\
    &\qquad +4T^2\norm{f}_{L^\infty(0,T;L^2(\Omega))}^2,\\
    ((\Theta^0,\Theta^1,\phi))&:= 
  {\frac{1}{2}(3a_1+|\gamma|/\gamma_0)}\norm{\Theta^1}^2 +{c_1}\Delta t\norm{\nabla \Theta^{1/2}}^2\\
  &\qquad +\frac{T^2}{a_1-|\gamma|/\gamma_0}\norm{\phi}^2_{L^\infty(0,T;L^2(\Omega))},\\
  ((P^0,P^1,g))&:=\frac{1}{2}{(3a_2+|\gamma|\gamma_0)}\norm{P^1}^2+{\kappa}\Delta t\norm{\nabla P^{1/2}}^2\!
+ \!\frac{T^2}{a_2-|\gamma|\gamma_0}\norm{g}^2_{L^\infty(0,T;L^2(\Omega))}.
\end{align*}
In relation with  \eqref{energy}, we define the discrete energy of  \eqref{eq:coupled} at time $t_n$, for 
$m=1,\cdots,N$,  as 
\begin{align*}
E_h(U^{m+1},\Theta^{m+1},P^{m+1})&:= \norm{\bar{\partial}_t U^{m+1/2}}^2+a_0\norm{\nabla \bar{\partial}_t U^{m+1/2}}^2+d_0 C_{\rm{Coer}}\norm{U^{m+1/2}}^2_h
   \\
  &  \qquad+({a_1-|\gamma|/\gamma_0}) \norm{\Theta^{m+1}}^2+{(a_2-|\gamma|\gamma_0)}\norm{P^{m+1}}^2\\
  &\qquad + \norm{(\Theta^{m},P^{m})}^2_{H}.
\end{align*}
Then, as in Theorem~\ref{weak-sol-thm}, the next theorem leads to the well-posedness  of \eqref{fully-discret-scheme}.
    \begin{thm}[Stability]\label{p2;stability-thm}
 Let $f, \phi,g \in L^\infty(0,T;L^2(\Omega))$,  $u^0 \in H^2_0(\Omega) ,u^{*0} \in H^1_0(\Omega)$,  and both $\theta^0,p^0 \in H^1_0(\Omega) $. Then, the scheme \eqref{fully-discret-scheme} is unconditionally stable. Moreover, for $1 \le m \le N-1$, the following bound holds:
\[ 
 E_h(U^{m+1},\Theta^{m+1},P^{m+1}) \lesssim ((U^0,U^1,f))+((\Theta^0,\Theta^1,\phi))+((P^0,P^1,g)),
\]  
where the constant hidden in "$\lesssim$''   depends on $T$ and on the model coefficients $a_0,c_1, \alpha,\beta, \kappa$.
\end{thm}
\begin{proof}
The proof follows in six steps as outlined below. 

\medskip
\noindent\textit{{Step 1} (Key inequality).} We multiply \eqref{un} by $8 \Delta t $, then choose $v_h= \delta_tU^{n}$ in \eqref{un},  and utilize \eqref{p2;stb1}-\eqref{p2;stb3} to show that
\begin{align}
    \nonumber4 & \big[\norm{\bar{\partial}_t U^{n+1/2}}^2-\norm{\bar{\partial}_t U^{n-1/2}}^2+ a_0\norm{\nabla \bar{\partial}_t U^{n+1/2}}^2-a_0\norm{\nabla \bar{\partial}_t U^{n-1/2}}^2\big] \\
\nonumber & \quad   
   +4d_0 a_h(U^{n+1/2},U^{n+1/2})  -4d_0 a_h(U^{n-1/2},U^{n-1/2})  \\
  & =8\Delta t (\alpha \nabla \Theta^{n,1/4}+\beta\nabla P^{n,1/4}, \nabla \delta_t U^n)+ 8 \Delta t( f^{n,1/4}, \delta_t U^n) \nonumber\\
   &= 4 \Delta t(\alpha \nabla \Theta^{n,1/4}+\beta\nabla P^{n,1/4}, \nabla (\bar{\partial}_tU^{n+1/2}+ \bar{\partial}_tU^{n-1/2})) \nonumber \\
  & \quad + 4 \Delta t( f^{n,1/4},\bar{\partial}_tU^{n+1/2}+ \bar{\partial}_tU^{n-1/2} ), 
   \label{st1}
\end{align}
with the identity $2\delta_t U^n=\bar{\partial}_tU^{n+1/2}+ \bar{\partial}_tU^{n-1/2}$ in the last equality.
Next we choose $\psi_h=2\Delta t \Theta^{n+1/2}$ in \eqref{thetan},  $q_h=2\Delta t P^{n+1/2}$ in \eqref{pn}, employ the identity \eqref{p2;stab5} and add the two resulting  equations to  obtain
\begin{align}\nonumber
   &  a_1\Delta t \bar{\partial}_t\norm{\Theta^{n+1/2}}^2+2\Delta t\big[{b_1} \norm{\Theta^{n+1/2}}^2
   +c_1  \norm{\nabla\Theta^{n+1/2}}^2+\kappa  \norm{\nabla P^{n+1/2}}^2\big]\\
   & \qquad+2\gamma\big[(P^{n},\Theta^{n})-(P^{n+1},\Theta^{n+1})\big] + a_2\Delta t \bar{\partial}_t\norm{P^{n+1/2}}^2 \nonumber \\
   &=-2\Delta t\big[ (\alpha \nabla \Theta^{n+1/2}+\beta\nabla P^{n+1/2}, \nabla \bar{\partial}_tU^{n+1/2})\big] \nonumber \\
  & \qquad + 2 \Delta t \big[(\phi^{n+1/2},{\Theta}^{n+1/2})+(g^{n+1/2},{P}^{n+1/2})\big].\label{st2}
\end{align}
We also combine the coupling terms on the right-hand sides of \eqref{st1}-\eqref{st2}, and utilize the term $\Theta^{n,1/4}:= \frac{1}{2}\left(\Theta^{n+1/2}+\Theta^{n-1/2} \right) $ twice (analogously for $P^{n,1/4}$). Elementary manipulations lead to the cancellation of some  terms,  and we eventually arrive at
\begin{align}
&4 (\alpha \nabla \Theta^{n,1/4}+\beta\nabla P^{n,1/4}, \nabla (\bar{\partial}_tU^{n+1/2}+ \bar{\partial}_tU^{n-1/2})) \nonumber \\
&\quad   - 2(\alpha \nabla \Theta^{n+1/2}+\beta\nabla P^{n+1/2}, \nabla \bar{\partial}_tU^{n+1/2})\nonumber\\
   &
   =\big[4(\alpha \nabla \Theta^{n,1/4}+\beta\nabla P^{n,1/4}, \nabla \bar{\partial}_tU^{n-1/2})\big]\nonumber \\
   &\quad +\big[ 2(\alpha \nabla \Theta^{n-1/2}+\beta\nabla P^{n-1/2}, \nabla \bar{\partial}_t U^{n+1/2})\big]\nonumber \\
   & :=A^n+B^n
.\label{p2-stab-all}
\end{align}
Then we add \eqref{st1}-\eqref{st2}, utilize \eqref{p2-stab-all} and then sum the resulting equation for $n=1,2,\cdots,m$, {for any $m=1,\cdots,N-1$}, and in turn use \eqref{P1 a_h_properties} and \eqref{t-norm} to arrive at the {\it key inequality}
\begin{align}   & 4\Big[\norm{\bar{\partial}_t U^{m+1/2}}^2+a_0\norm{\nabla \bar{\partial}_t U^{m+1/2}}^2
   +d_0 C_{\rm{Coer}}\norm{U^{m+1/2}}^2_h\Big]+{a_1} \norm{\Theta^{m+1}}^2\nonumber \\
   &\quad +{a_2}\norm{P^{m+1}}^2+2\norm{(\Theta^{m},P^{m})}^2_{H} \nonumber\\
   & \le 4\Big[\norm{\bar{\partial}_t U^{1/2}}^2+a_0\norm{\nabla \bar{\partial}_t U^{1/2}}^2
   +d_0 C_{\rm{Cont}}\norm{U^{1/2}}^2_h\Big]+{a_1} \norm{\Theta^{1}}^2\nonumber \\
   &\quad +{a_2}\norm{P^{1}}^2 +2\gamma\big[(P^{m+1},\Theta^{m+1})-(P^{1},\Theta^{1})\big]
 \nonumber  \\
  &\quad  +\Delta t \sum_{n=1}^{m}\Big[ A^n+ B^n + 4( f^{n,1/4},\bar{\partial}_tU^{n+1/2}+ \bar{\partial}_tU^{n-1/2})+(\phi^{n+1/2},\Theta^{n+1}+\Theta^{n})
  \nonumber \\
  & \quad +(g^{n+1/2},P^{n+1}+P^{n})\Big].\label{eq:key0}
\end{align}

\medskip 
\noindent\textit{Step 2 (Bound for $\Delta t \sum_{n=1}^m A^n$).}
Using the definition  $\Theta^{n,1/4}:= \frac{1}{2}\left(\Theta^{n+1/2}+\Theta^{n-1/2} \right) $ (and an analogous expression for $P^{n,1/4}$) yields
\begin{align}
 \sum_{n=1}^m A^n& = \sum_{n=1}^m 4(\alpha \nabla \Theta^{n,1/4}+\beta\nabla P^{n,1/4}, \nabla \bar{\partial}_tU^{n-1/2})\nonumber\\
 & = 2\alpha \sum_{n=1}^m  ( \nabla \Theta^{n+1/2}+ \nabla \Theta^{n-1/2},\nabla \bar{\partial}_tU^{n-1/2}
 )\nonumber \\
 &\quad +2\beta \sum_{n=1}^m ( \nabla P^{n+1/2}+ \nabla P^{n-1/2}, \nabla \bar{\partial}_tU^{n-1/2}).\label{p2;1/4}
\end{align}
Next we can apply Cauchy--Schwarz inequality and  Young's inequality $(ab \le a^2/2
\epsilon + b^2 \epsilon/2 )$ with $\epsilon = 4/c_1$ to bound the first term on the right-hand side of \eqref{p2;1/4} by 
\begin{align*}
& 2\alpha \sum_{n=1}^m  ( \nabla \Theta^{n+1/2}+ \nabla \Theta^{n-1/2},\nabla \bar{\partial}_tU^{n-1/2}
 )\\
 &\quad  \le \frac{c_1}{2}
 \sum_{n=1}^m \left( \|\nabla \Theta^{n+1/2}\|^2 +\|\nabla \Theta^{n-1/2}\|^2  \right) + \sum_{n=1}^m  \frac{4\alpha^2}{c_1} \|\nabla \bar{\partial}_tU^{n-1/2} \|^2 \nonumber \\
 &\quad  \le c_1 \sum_{n=1}^m \|\nabla \Theta^{n+1/2}\|^2 + \frac{c_1}{2}
 \|\nabla \Theta^{1/2}\|^2 +  \sum_{n=1}^m
 \frac{4\alpha^2}{c_1} \|\nabla \bar{\partial}_tU^{n-1/2} \|^2
\end{align*}
with elementary manipulations and addition of $\frac{4\alpha^2}{c_1}  \|\nabla \Theta^{m+1/2}\|^2$ in the last step. Similar arguments \rev{can be used to bound} the second term on the right-hand side of \eqref{p2;1/4}. A combination of all this  in \eqref{p2;1/4} (after multiplying by $\Delta t$) shows
\begin{align}
  \Delta t \sum_{n=1}^{m} A^n&
  \le  \Delta t \Big({c_1}\sum_{n=1}^{m} \norm{\nabla \Theta^{n+1/2}}^2+{\kappa} \sum_{n=1}^{m}\norm{\nabla P^{n+1/2}}^2 \Big)\nonumber \\
  &\quad +\frac{\Delta t}{2} \Big( c_1 \norm{\nabla \Theta^{1/2}}^2+{\kappa} \norm{\nabla P^{1/2}}^2 \Big)\nonumber\\
  &\quad
  +4\Delta t\Big( \frac{\alpha^2}{c_1}+\frac{\beta^2}{\kappa}\Big)\sum_{n=1}^{m}\norm{\nabla \bar{\partial}_t U^{n-1/2}}^2.  \label{p2;nabla-pm+1-stab}
  \end{align}
 \noindent \textit{{Step 3} (Bound for $\Delta t \sum_{n=1}^{m}B^n$).}
First, we rewrite $\Delta t \sum_{n=1}^{m}B^n $ as 
 \begin{align*}& \Delta t \sum_{n=1}^{m-1} (2\alpha \nabla \Theta^{n-1/2}+2\beta\nabla P^{n-1/2}, \nabla \bar{\partial}_t U^{n+1/2})\\ &\quad + 2\Delta t  (\alpha \nabla \Theta^{m-1/2}+\beta\nabla P^{m-1/2}, \nabla \bar{\partial}_t U^{m+1/2}).  \end{align*}
 Then, it suffices to apply  Cauchy--Schwarz inequality and Young's inequality  with $\epsilon = 2/c_1$ (resp. $\epsilon = 2/\kappa$) to the first (resp. second) term in the summation on the right-hand side above,   to obtain
\begin{align*}
  &   \sum_{n=1}^{m-1} 2(\alpha \nabla \Theta^{n-1/2}, \nabla \bar{\partial}_t U^{n+1/2}) 
    \le \sum_{n=1}^{m-1} \left(\frac{c_1}{2}\| {\nabla}\Theta^{n-1/2}\|^2 + \frac{2 \alpha^2}{c_1} \| \nabla \bar{\partial}_t U^{n+1/2}\|^2  \right) \\
    &\qquad  \le \sum_{n=1}^{m-1} \frac{c_1}{2}\| {\nabla} \Theta^{n+1/2}\|^2 +\frac{c_1}{2}\| {\nabla}\Theta^{1/2}\|^2 + \frac{2 \alpha^2}{c_1} \sum_{n=1}^{m-1}  \| \nabla\bar{\partial}_t U^{n+1/2}\|^2 \\
&\biggl(\text{resp. }     \sum_{n=1}^{m-1} 2(\beta \nabla P^{n-1/2}, \nabla \bar{\partial}_t U^{n+1/2})  
    \le 
     \sum_{n=1}^{m-1} \frac{\kappa}{2}\| 
 {\nabla} P^{n+1/2}\|^2 +\frac{\kappa}{2}\|{\nabla} P^{1/2}\|^2 \\
 &\qquad \qquad \qquad \qquad \qquad \qquad \qquad \qquad \qquad + \frac{2 \beta^2}{\kappa} \sum_{n=1}^{m-1}  \| \nabla\bar{\partial}_t U^{n+1/2}\|^2  \biggr),
\end{align*}
with an addition of a non-negative term $\frac{c_1}{2}\|\nabla \Theta ^{m-1/2}\|^2$ (resp. $\frac{\kappa}{2}\|\nabla P ^{m-1/2}\|^2$) on the right-hand side. An analogous simplification (with $\epsilon= a_0$) in the Young's inequality   leads to 
\begin{align*}
 & 2 \Delta t (\alpha \nabla \Theta^{m-1/2}+\beta\nabla P^{m-1/2}, \nabla \bar{\partial}_t U^{m+1/2})  \\
 &\quad  \le a_0^{-1} (\Delta t)^2\left(\alpha^2 
  \|\nabla \Theta^{m-1/2}\|^2 + \beta^2 \|\nabla P^{m-1/2}\|^2\right) + 2a_0 \| \nabla\bar{\partial}_t U^{m+1/2}\|^2 \\
  & \quad \le a_0^{-1} (\Delta t)^2 \sum_{n=1}^{m-1}\left(\alpha^2 
  \|\nabla \Theta^{n+1/2}\|^2 + \beta^2\|\nabla P^{n+1/2}\|^2\right) + 2a_0 \| \nabla\bar{\partial}_t U^{m+1/2}\|^2,
\end{align*}
where there is an over bound by $a_0^{-1} (\Delta t)^2 \sum_{n=1}^{m-2}\left(\alpha^2 
  \|\nabla \Theta^{n+1/2}\|^2 +{ \beta^2} \|\nabla P^{n+1/2}\|^2\right)$ in the last step. 
  A combination of all this yields
  \begin{align}
 \Delta t \sum_{n=1}^{m}B^n 
 & \le  2{a_0} \norm{\nabla \bar{\partial}_t U^{m+1/2}}^2+ \frac{\Delta t}{2} \left(c_1 \norm{\nabla \Theta^{1/2}}^2+{\kappa} \norm{\nabla P^{1/2}}^2 \right) \nonumber\\
 &\quad +\frac{\Delta t}{2} \sum_{n=1}^{m-1}\Big({c_1}   \norm{\nabla \Theta^{n+1/2}}^2+{\kappa}  \norm{\nabla P^{n+1/2}}^2\Big)\nonumber \\
 &\quad +2 \Delta t\Big( \frac{\alpha^2}{c_1}+\frac{\beta^2}{\kappa}\Big) \sum_{n=1}^{m-1}\norm{\nabla \bar{\partial}_t U^{n+1/2}}^2\nonumber\\
  &\quad 
  + \frac{(\Delta t)^2}{a_0} \sum_{n=1}^{m-1}\Big(\alpha^2\norm{\nabla \Theta^{n+1/2}}^2 +\beta^2\norm{\nabla P^{n+1/2}}^2\Big) .
\label{-T2}
  \end{align}
\textit{{Step 4} (Bounds for load and source terms).}
One more application of Cauchy--Schwarz inequality and Young's inequality  with $\epsilon = 1/2T$, results in the following bound
\begin{align*}
 & 4\Delta t \sum_{n=1}^{m}( f^{n,1/4}, \bar{\partial}_t U^{n+1/2}+\bar{\partial}_t U^{n-1/2}) \\
 & \le  {4T}\Delta t \sum_{n=1}^{m} \norm{f^{n,1/4}}^2+\frac{\Delta t }{T}\sum_{n=1}^{m} \norm{\bar{\partial}_t U^{n+1/2}+\bar{\partial}_t U^{n-1/2}}^2.
\end{align*}
Note that $\Delta t \sum_{n=1}^{m} \norm{f^{n,1/4}}^2 \le m\Delta t  \norm{f}_{L^\infty(0,T;L^2(\Omega))}^2 \le T \norm{f}_{L^\infty(0,T;L^2(\Omega))}^2$.
Moreover, \rev{the following bound} $$\norm{\bar{\partial}_t U^{n+1/2}+\bar{\partial}_t U^{n-1/2}}^2 \le 2 \norm{\bar{\partial}_t U^{n+1/2}+}^2+2\norm{\bar{\partial}_t U^{n-1/2}}^2,$$ shows \rev{that}  
\begin{align*}
&    \frac{\Delta t}{T} \sum_{n=1}^{m} \norm{\bar{\partial}_t U^{n+1/2}+\bar{\partial}_t U^{n-1/2}}^2 \\
& \le  2\frac{\Delta t }{T}\norm{\bar{\partial}_t U^{1/2}}^2+2\frac{\Delta t}{T} \norm{\bar{\partial}_t U^{m+1/2}}^2+ 4\frac{\Delta t}{T} \sum_{n=1}^{m-1} \norm{\bar{\partial}_t U^{n+1/2}}^2.\end{align*}
\rev{Putting together all this} with   $\frac{\Delta t }{T} \le 1$  yields 
\begin{align*}
4\Delta t \sum_{n=1}^{m}( f^{n,1/4}, \bar{\partial}_t U^{n+1/2}+\bar{\partial}_t U^{n-1/2})&\le {4T}^{2}\norm{f}_{L^\infty(0,T;L^2(\Omega))}^2 +{2} \norm{\bar{\partial}_t U^{1/2}}^2\\
&\quad
+ {2}\norm{\bar{\partial}_t U^{m+1/2}}^2+4\frac{\Delta t}{T} \sum_{n=1}^{m} \norm{\bar{\partial}_t U^{n-1/2}}^2.
\end{align*}
 Moreover, the same arguments, with $\epsilon = \frac{1}{2T}(a_1 - |\gamma|/\gamma_0)$ (resp. $\epsilon = \frac{1}{2T}(a_2 - |\gamma|\gamma_0)$) used in Young's inequality, also lead to the following bounds 
\begin{subequations} \label{eq:a}
\begin{align}
&\Delta t\sum_{n=1}^{m}(\phi^{n+1/2},\Theta^{n+1}+\Theta ^n) \le \frac{T^2}{a_1-|\gamma|/\gamma_0}\norm{\phi}^2_{L^\infty(0,T;L^2(\Omega))}+\frac{a_1-|\gamma|/\gamma_0}{2}  \norm{ \Theta^{1}}^2\nonumber\\
&\qquad \qquad \qquad +\frac{a_1-|\gamma|/\gamma_0}{2}  \norm{ \Theta^{m+1}}^2+{(a_1-|\gamma|/\gamma_0)} \rev{\frac{\Delta t}{T}\sum_{n=1}^{m}\norm{ \Theta^{n}}^2,}\\
& \bigg(\text{resp. } \Delta t\sum_{n=1}^{m}(g^{n+1/2},P^{n+1}+P^n)
 \le \frac{T^2}{a_2-|\gamma|\gamma_0}\norm{g}^2_{L^\infty(0,T;L^2(\Omega))}+\frac{a_1-|\gamma|\gamma_0}{2}  \norm{ P^{1}}^2\nonumber\\
 &\qquad \qquad\qquad +\frac{a_1-|\gamma|\gamma_0}{2}  \norm{ P^{m+1}}^2+{(a_1-|\gamma|\gamma_0)}\frac{\Delta t}{T}\sum_{n=1}^{m}\norm{ P^{n}}^2\rev{\bigg).} \label{P2;Pm+1}
\end{align}
\end{subequations}
\textit{{Step 5} (bound for $2\gamma(P^{m+1},\Theta^{m+1}) -2\gamma(P^{1},\Theta^{1})$).} A triangle inequality plus Cauchy--Schwarz and Young's inequalities with $\epsilon=1/\gamma_0$ lead to 
\begin{align*}& 2\gamma(P^{m+1},\Theta^{m+1}) -2\gamma(P^{1},\Theta^{1})\\
& \qquad \le{|\gamma|\gamma_0}\norm{P^{m+1}}^2+{|\gamma|}/{\gamma_0}\norm{\Theta^{m+1}}^2+{|\gamma|}{\gamma_0} \norm{P^{1}}^2+{|\gamma|}/{\gamma_0}\norm{\Theta^{1}}^2.\end{align*}
\textit{{Step 6} (Consolidation).}
A combination  of \eqref{p2;nabla-pm+1-stab}-\eqref{eq:a} and \eqref{eq:key0} together with elementary manipulations (adding the non-negative term $2 \Delta t \Big( \frac{\alpha^2}{c_1}+\frac{\beta^2}{\kappa}\Big) \norm{\nabla \bar{\partial}_t U^{1/2}}^2 $ on the right-hand side), yields the bound: 
\begin{align}   &2\norm{\bar{\partial}_t U^{m+1/2}}^2+2a_0\norm{\nabla \bar{\partial}_t U^{m+1/2}}^2
   +4d_0 C_{\rm{Coer}}\norm{U^{m+1/2}}^2_h+2\norm{(\Theta^{m},P^{m})}^2_{H} \nonumber\\
   &\quad +\frac{1}{2}({a_1-|\gamma|/\gamma_0}) \norm{\Theta^{m+1}}^2+\frac{1}{2}({a_2-|\gamma|\gamma_0})\norm{P^{m+1}}^2  - (\Delta t)    {c_1}\norm{\nabla \Theta^{m+1/2}}^2 \nonumber\\
& \quad - (\Delta t)   {\kappa} \norm{\nabla P^{m+1/2}}^2  -\frac{3\Delta t}{2} \sum_{n=1}^{m-1}\Big({c_1}   \norm{\nabla \Theta^{n+1/2}}^2+{\kappa}  \norm{\nabla P^{n+1/2}}^2\Big)\nonumber\\
   &  \le  ((U^0,U^1,f))+  ((\Theta^0,\Theta^1,\phi))+  ((P^0,P^1,g)) \nonumber \\ & \quad   +2\frac{\Delta t}{T} \Big(\sum_{n=1}^{m} 2\norm{\bar{\partial}_t U^{n-1/2}}^2+\frac{1}{2}{(a_1-\frac{|\gamma|}{\gamma_0})} \sum_{n=1}^{m}\norm{ \Theta^{n}}^2  +\frac{1}{2}{(a_2-|\gamma|\gamma_0)} \sum_{n=1}^{m}\norm{ P^{n}}^2 \Big) \nonumber \\
  & \quad+ 6 \Delta t \Big( \frac{\alpha^2}{c_1}+\frac{\beta^2}{\kappa}\Big)\sum_{n=0}^{m-1}\norm{\nabla \bar{\partial}_t U^{n+1/2}}^2 \nonumber \\
 & \quad 
    +\frac{(\Delta t)^2}{a_0} \sum_{n=1}^{m-1}\Big(\alpha^2\norm{\nabla \Theta^{n+1/2}}^2 +\beta^2\norm{\nabla P^{n+1/2}}^2\Big). \label{consol}
\end{align}
Utilize the  definition \eqref{t-norm} to obtain
\begin{align*}
    &2\norm{(\Theta^{m},P^{m})}^2_{H} - \Delta t \Big( {c_1}  \norm{\nabla \Theta^{m+1/2}}^2+ {\kappa}  \norm{\nabla P^{m+1/2}}^2\Big)\\
    &\quad -\frac{3\Delta t}{2} \sum_{n=1}^{m-1}\Big({c_1}   \norm{\nabla \Theta^{n+1/2}}^2+{\kappa}  \norm{\nabla P^{n+1/2}}^2\Big)\\
    &=\frac{\Delta t}{2}\Big[{c_1} \norm{\nabla\Theta^{m+1/2}}^2+{\kappa} \norm{\nabla P^{m+1/2}}^2\\
    &\quad + \sum_{n=1}^m \big(4b_1\norm{\Theta^{n+1/2}}^2+c_1\norm{\nabla\Theta^{n+1/2}}^2+\kappa\norm{\nabla P^{n+1/2}}^2\big)\Big]
     \\
    & \ge \frac{\Delta t}{2}\Big({c_1} \norm{\nabla\Theta^{m+1/2}}^2+{\kappa} \norm{\nabla P^{m+1/2}}^2 \Big) + \frac{1}{2} \norm{(\Theta^{m},P^{m})}^2_{H},
\end{align*}
and the elementary manipulations 
$$6 \Delta t \Big( \frac{\alpha^2}{c_1}+\frac{\beta^2}{\kappa}\Big)\sum_{n=0}^{m-1}\norm{\nabla \bar{\partial}_t U^{n+1/2}}^2 = \frac{\Delta t}{T} \Big(\frac{3T\alpha^2}{a_0 c_1} +
\frac{3T\beta^2}{a_0\kappa}\Big)
\sum_{n=0}^{m-1} 2 a_0\norm{\nabla \bar{\partial}_t U^{n+1/2}}^2,
     $$
     \rev{and} 
\begin{align*}&\frac{(\Delta t)^2}{a_0} \sum_{n=1}^{m-1}\Big(\alpha^2\norm{\nabla \Theta^{n+1/2}}^2 +\beta^2\norm{\nabla P^{n+1/2}}^2\Big) \\
&\qquad \le \frac{\Delta t}{T} \Big(\frac{2T\alpha^2}{a_0 c_1} +
\frac{2T\beta^2}{a_0\kappa}\Big)
\frac{\Delta t}{2}\sum_{n=1}^{m-1}   \Big(c_1\norm{\nabla \Theta^{n+1/2}}^2 +\kappa\norm{\nabla P^{n+1/2}}^2\Big),\end{align*}
in \eqref{consol} to show that 
\begin{align*}
&2\norm{\bar{\partial}_t U^{m+1/2}}^2+2a_0\norm{\nabla \bar{\partial}_t U^{m+1/2}}^2
   +\frac{1}{2}({a_1-|\gamma|/\gamma_0}) \norm{\Theta^{m+1}}^2\\
   &\quad +\frac{1}{2}({a_2-|\gamma|\gamma_0})\norm{P^{m+1}}^2 +\frac{\Delta t}{2}\Big({c_1} \norm{\nabla\Theta^{m+1/2}}^2+{\kappa} \norm{\nabla P^{m+1/2}}^2\Big)\\
   &\quad + \frac{1}{2} \norm{(\Theta^{m},P^{m})}^2_{H}
     +4d_0 C_{\rm{Coer}}\norm{U^{m+1/2}}^2_h \nonumber \\
    & \le   ((U^0,U^1,f))+  ((\Theta^0,\Theta^1,\phi))+  ((P^0,P^1,g)) 
    \\
&\quad     +
    C\frac{\Delta t}{T} \sum_{n=0}^{m-1}\Big[2\norm{\bar{\partial}_t U^{n+1/2}}^2
  +2a_0\norm{\bar{\partial}_t \nabla U^{n+1/2}}^2 +\frac{1}{2}{(a_1-|\gamma|/\gamma_0)}\norm{ \Theta^{n+1}}^2\\
  &\qquad 
   +\frac{1}{2}{(a_2-|\gamma|\gamma_0)}\norm{ P^{n+1}}^2+\frac{\Delta t}{2}  \Big(c_1\norm{\nabla \Theta^{n+1/2}}^2 +\kappa\norm{\nabla P^{n+1/2}}^2\Big) \Big],
\end{align*}
where $C=\max\{2, \frac{3T\alpha^2}{a_0c_1}+ \frac{3T\beta^2}{a_0\kappa}\}$. 
Then we invoke Lemma~\ref{P1 d-gronwall} and \eqref{t-norm} to arrive at
\begin{align*}
&2\norm{\bar{\partial}_t U^{m+1/2}}^2+2a_0\norm{\nabla \bar{\partial}_t U^{m+1/2}}^2
  +\frac{1}{2}(a_1-\!\frac{|\gamma|}{\gamma_0}) \norm{\Theta^{m+1}}^2+\frac{1}{2}({a_2\!-\!|\gamma|\gamma_0})\norm{P^{m+1}}^2\quad& \nonumber\\
    &\quad\! +\frac{\Delta t}{2}\Big({c_1} \norm{\nabla\Theta^{m+1/2}}^2+{\kappa} \norm{\nabla P^{m+1/2}}^2\Big)\! + \frac{1}{2} \|(\Theta^m, P^m) \|^2_H \! +\! 4d_0C_{\rm{Coer}}\norm{U^{m+1/2}}^2_h 
    \nonumber \\
     & \le e^{m\frac{C\Delta t}{T}} \Big(((U^0,U^1,f))+  ((\Theta^0,\Theta^1,\phi))+  ((P^0,P^1,g))\Big)\\
    & \le  e^{C} \Big(((U^0,U^1,f))+  ((\Theta^0,\Theta^1,\phi))+  ((P^0,P^1,g))\Big),
\end{align*}
with  ${m\Delta t}\le T$ in the last step. Then, we 
 ignore the non-negative term $\frac12\Delta t[c_1 \norm{\nabla\Theta^{m+1/2}}^2+{\kappa} \norm{\nabla P^{m+1/2}}^2]$ on the left-hand side \rev{to conclude}.
\end{proof}
\begin{rem}
    For the existence of \rev{the} unique solution, it suffices to show that \( (0,0,0) \) is the only solution of the fully discrete problem \eqref{U0,THETA0,P0}-\eqref{fully-discret-scheme} with homogeneous initial conditions and load/source functions. From \eqref{U0,THETA0,P0}, it is evident that if \( u^0=\theta^0=p^0 = 0 \), then \( U^0=\Theta^0=P^0 = 0 \), which, together with \( u^{*0} = 0 \), leads to \( U^1=\Theta^1=P^1 = 0 \) from \eqref{3.10}. Then, we can utilize  $U^0=\Theta^0=P^0=U^1=\Theta^1=P^1= 0$ in Theorem~\ref{p2;stability-thm} to show that $U^{m+1}=\Theta^{m+1}=P^{m+1} = 0$ for all $1\le m \le N-1.$
\end{rem}
\section{Error analysis}\label{p2;fully-error-sec}
This section establishes the error estimates for the fully discrete scheme presented in the previous section. In Subsection~\ref{subsec-initial}, we prove the error estimates at the initial time steps $t_0$ and $t_1$ for the scheme \eqref{U0,THETA0,P0}-\eqref{3.10}. The subsequent subsection provides error estimates for the scheme \eqref{fully-discret-scheme} in different norms. 
Let us consider the following  decomposition of errors at time $t_n$, for $n=1,\ldots, N$ 
\begin{subequations}\label{spilit}
\begin{align} 
    &
    u(t_n)-U^n=\big(u(t_n)-\mathcal{R}_h u(t_n)\big)+\big(\mathcal{R}_hu(t_n)-U^n \big):=\rho^n+\zeta^n,\label{p2;splitting1}\\
    &
    \theta(t_n)-\Theta^n=\big(\theta(t_n)-\Pi_h\theta(t_n)\big)+\big(\Pi_h\theta(t_n)-\Theta^n\big):=\eta^n+\Psi^n,\label{p2;splitting2}\\
    &
     p(t_n)-P^n=\big(p(t_n)-\Pi_hp(t_n)\big)+\big(\Pi_hp(t_n)-P^n\big):=\varrho^n+\xi^n,\label{p2;splitting3}
\end{align}
\end{subequations}
where $\mathcal{R}_h$ and $\Pi_h$ are the projections defined in \eqref{P1 ritz_projection} and \eqref{p2;ritzprojection2}, respectively.
\subsection{Initial error bounds}\label{subsec-initial}
Since our discrete formulation is split into two parts,  solutions at $t_0$ and $t_1$ are determined using \eqref{U0,THETA0,P0}-\eqref{3.10}, whereas the solutions at $t_2, t_3, \dots, t_n$ are computed using \eqref{fully-discret-scheme}---it is thus necessary to estimate the initial error at time levels $t_0$ and $t_1$ before we proceed to derive the error estimates. To do so, we take the average of the equations in system \eqref{weak form} at $t_0$ and $t_1$ as
\begin{align*}
    (u_{tt}^{1/2},v)+a_0 (\nabla u_{tt}^{1/2},\nabla v)+d_0(\nabla ^2 u^{1/2}, \nabla^2 v) -\alpha (\nabla \theta^{1/2}, \nabla v)&\\-\beta (\nabla p^{1/2}, \nabla v)&=(f^{1/2},v), \\
    a_1(\theta_t^{1/2}, \psi)-\gamma (p_t^{1/2}, \psi)+b_1(\theta^{1/2},\psi) +c_1 (\nabla \theta^{1/2},\nabla \psi)&\\ +\alpha (\nabla u_t^{1/2}, \nabla \psi)&=(\phi^{1/2},\psi), \\
    a_2 (p_t^{1/2},q)-\gamma (\theta_t^{1/2},q)+\kappa (\nabla p^{1/2},\nabla q)&\\+\beta (\nabla u_t^{1/2},\nabla q)&=(g^{1/2},q),
    \end{align*}
   for all $v \in H^2_0(\Omega) $ and both $\psi,q \in H^1_0(\Omega).$
      Let us observe that for the smoother $Q$ defined in Section~\ref{sec:semi} there holds $\text{Range}\,(Q) \subset H^2_0(\Omega)$, and,  readily from the definitions, we have that $W_h \subset  H^1_0(\Omega)$. Then, for any $v_h \in V_h$  and $\psi_h,q_h \in W_h$, we can choose  $ Q v_h\in H^2_0(\Omega)$ and $\psi_h,q_h \in W_h \subset  H^1_0(\Omega)$ as test functions in the last system of equations and employ the definitions of the projections $\mathcal{R}_h $ and $\Pi_h$ from  \eqref{P1 ritz_projection} and  \eqref{p2;ritzprojection2}, respectively to arrive at 
\begin{subequations}\label{cont-int}
\begin{align}
    (u_{tt}^{1/2},Qv_h)+a_0(\nabla u_{tt}^{1/2},\nabla Qv_h)+d_0a_h(\mathcal{R}_h u^{1/2}, v_h)&\label{int-u}\\
  - \alpha(\nabla  \theta^{1/2}, \nabla Qv_h) +\beta (\nabla  p^{1/2}, \nabla Qv_h) & =(f^{1/2} ,Qv_h),\nonumber \\
    a_1(\theta_t^{1/2}, \psi_h)-\gamma ({p_t^{1/2}}, \psi_h)+b_1(\theta^{1/2},\psi_h) +c_1 (\nabla \Pi_h\theta^{1/2},\nabla \psi_h) &\label{int-t} \\
    +\alpha (\nabla  u_t^{1/2}, \nabla \psi_h)&=(\phi^{1/2},\psi_h ),\nonumber \\
    a_2 (p_t^{1/2},q_h)-\gamma ( \theta_t^{1/2},q_h)+\kappa (\nabla \Pi_hp^{1/2},\nabla q_h)+\beta (\nabla  u_t^{1/2},\nabla q_h)&=(g^{1/2},q_h)\label{int-p}.
    \end{align}
    \end{subequations}
Since all the terms $u_{tt}^{1/2}, \theta^{1/2} , p^{1/2}, \text{ and }(Q-I)v_h $  belong to  $ H^1_0(\Omega)$, an integration by parts leads to 
    \begin{align}
  \nonumber &  (a_0\nabla u_{tt}^{1/2}-\beta \nabla p^{1/2}-\alpha \nabla \theta^{1/2},\nabla (Q-I)v_h)\\
  &\quad =(-a_0\Delta  u_{tt}^{1/2}+\beta\Delta  p^{1/2}+\alpha \Delta \theta^{1/2},(Q-I)v_h). \label{p2;IBP}
    \end{align}
Next we recall  $F:= F(t,\bx)=f(t,\bx)-u_{tt}+a_0 \Delta u_{tt}-\alpha \Delta \theta-\beta \Delta p $ from \eqref{F}. This and 
    { \eqref{p2;IBP} in \eqref{int-u} with some basic manipulations yields}
\begin{align}
    &2 (\Delta t)^{-1}(\bar{\partial}_tu^{1/2},v_h)+2a_0  (\Delta t)^{-1}(\nabla \bar{\partial}_tu^{1/2},\nabla v_h)+d_0a_h(\mathcal{R}_h u^{1/2},  v_h)\\
    &\qquad-\alpha (\nabla \theta^{1/2}, \nabla v_h)-\beta (\nabla p^{1/2}, \nabla v_h)\nonumber\\ 
    &
   \;  =(F^{1/2},(Q-I)v_h)+2 (\Delta t)^{-1}(\bar{\partial}_tu^{1/2},v_h)
    -(u_{tt}^{1/2},v_h)\\
      &\qquad+2a_0 (\Delta t)^{-1}(\nabla \bar{\partial}_tu^{1/2},\nabla v_h)-a_0(\nabla u_{tt}^{1/2},\nabla v_h).\label{int-u0}
    \end{align}
   Let us now define the initial truncation terms $R^0, r_0, \tau_0$, and $s_0$ as follows:
   \begin{subequations}\label{tranc-terms}
   \begin{align}
  &R^0:={2}(\Delta t)^{-1}(\bar{\partial}_t u^{1/2}-u^{*0}) -u_{tt}^{1/2},\quad  r_0:=\bar{\partial}_tu^{1/2}-u_{t}^{1/2},\\
  &   \tau_0:=\bar{\partial}_t \theta^{1/2}-\theta_{t}^{1/2}, \quad s_0:=\bar{\partial}_t p^{1/2}-p_{t}^{1/2}.
    \end{align}
    \end{subequations}
Utilizing these definitions and subtracting the equation  \eqref{int_u0} form \eqref{int-u0}, \eqref{int_theta0} from \eqref{int-t}, and \eqref{int_p0} from \eqref{int-p}, 
we can obtain the following system   
    \begin{align*}
    &
    2 (\Delta t)^{-1}(\bar{\partial}_t(u^{1/2}-U^{1/2}),v_h)+2 (\Delta t)^{-1}a_0 (\nabla \bar{\partial}_t(u^{1/2}-U^{1/2}),\nabla v_h)\\
    &\quad  +d_0 a_h(\mathcal{R}_hu^{1/2}-U^{1/2}, v_h)
     -\alpha (\nabla (\theta^{1/2}-\Theta^{1/2}), \nabla v_h) -\beta (\nabla (p^{1/2}-P^{1/2}), \nabla v_h) \\
    &  
     =(F^{1/2},(Q-I)v_h)+(R_0,v_h)+a_0(\nabla R_0,\nabla v_h) ,\nonumber\\
   &a_1(\bar{\partial}_t(\theta^{1/2}-{\Theta}^{1/2}), \psi_h)-\gamma (\bar{\partial}_t(p^{1/2}-{P^{1/2}}), \psi_h)+b_1(\theta^{1/2}-\Theta^{1/2},\psi_h) \\
    & \quad +c_1 (\nabla (\Pi_h\theta^{1/2}-\Theta^{1/2}),\nabla \psi_h)
   +\alpha (\nabla \bar{\partial}_t(u^{1/2}-U^{1/2}) ,\nabla \psi_h)\\
    &  
    =a_1(\tau_0, \psi_h)-\gamma (s_0, \psi_h)+\alpha (\nabla r_0, \nabla \psi_h),\\
   &a_2 (\bar{\partial}_t(p^{1/2}-P^{1/2}),q_h)-\gamma (\bar{\partial}_t(\theta^{1/2}-{\Theta}^{1/2}),q_h) +\kappa (\nabla (\Pi_h p^{1/2}-P^{1/2}),\nabla q_h)\\
    &  
  \quad+\beta (\nabla \bar{\partial}_t (u^{1/2}-U^{1/2}),\nabla q_h)=a_2 (s_0,q_h)-\gamma(\tau_0,p_h)+\beta (\nabla r_0,\nabla q_h).
    \end{align*}
In turn, {for $(v_h,\psi_h,q_h)\in V_h\times W_h \times W_h$},  the error decomposition  described in \eqref{spilit} leads to
\begin{subequations}
    \begin{align}
         &2 (\Delta t)^{-1}(\bar{\partial}_t\zeta^{1/2},v_h)+2 (\Delta t)^{-1}a_0 (\nabla \bar{\partial}_t\zeta^{1/2},\nabla v_h)\nonumber\\
    &\quad  +d_0 a_h(\zeta^{1/2}, v_h) -\alpha (\nabla \Psi^{1/2}, \nabla v_h) -\beta (\nabla \xi^{1/2}, \nabla v_h)  \nonumber\\
    & 
    = -2 (\Delta t)^{-1}(\bar{\partial}_t\rho^{1/2},v_h)-2 (\Delta t)^{-1}a_0 (\nabla \bar{\partial}_t\rho^{1/2},\nabla v_h)+\alpha (\nabla \eta^{1/2}, \nabla v_h)\nonumber\\
    &
    \quad+\beta (\nabla \varrho^{1/2}, \nabla v_h) 
    +(F^{1/2},(Q-I)v_h)+(R_0,v_h)+a_0(\nabla R_0,\nabla v_h), \label{zeta1/2}\\
    & 
    a_1(\bar{\partial}_t \Psi^{1/2}, \psi_h)-\gamma (\bar{\partial}_t \xi^{1/2}, \psi_h)+b_1(\Psi^{1/2},\psi_h) \nonumber\\
   & \quad+c_1 (\nabla \Psi^{1/2},\nabla \psi_h)+\alpha (\nabla \bar{\partial}_t\zeta^{1/2} ,\nabla \psi_h)\nonumber\\
   &
    =-a_1(\bar{\partial}_t \eta^{1/2}, \psi_h)+\gamma (\bar{\partial}_t \varrho^{1/2}, \psi_h)-b_1(\eta^{1/2},\psi_h)  -\alpha (\nabla \bar{\partial}_t\rho^{1/2} ,\nabla \psi_h)\nonumber\\
    &\quad
    +a_1(\tau_0, \psi_h)-\gamma (s_0, \psi_h)+\alpha (\nabla r_0, \nabla \psi_h),\label{psi1/2}\\
   &
    a_2 (\bar{\partial}_t\xi^{1/2},q_h)-\gamma (\bar{\partial}_t \Psi^{1/2},q_h)+\kappa (\nabla \xi^{1/2},\nabla q_h)+\beta (\nabla \bar{\partial}_t \zeta^{1/2},\nabla q_h)\nonumber\\
    &
    =-a_2 (\bar{\partial}_t\varrho^{1/2},q_h)
 +\gamma (\bar{\partial}_t \eta^{1/2},q_h)-\beta (\nabla \bar{\partial}_t \rho^{1/2},\nabla q_h)
   \nonumber\\
    &\quad +a_2 ( s_0,q_h)-\gamma( \tau_0,q_h)+\beta (\nabla r_0, \nabla q_h).\label{xi1/2}
    \end{align}\end{subequations}
    

\noindent Note that   $\norm{\bar{\partial}_t\rho^{1/2}}=(1/ \Delta t)\norm{\int_{0}^{t_1}\rho_t(t)\ds}$ and $\norm{\nabla \bar{\partial}_t\rho^{1/2}}=(1/ \Delta t)\norm{\int_{0}^{t_1}\nabla \rho_t(t)\ds}$. Then, definition 
    \eqref{p2;splitting1}  and the approximation property in \eqref{P1 norm_ritz}  yield the bounds 
{ \begin{subequations}\begin{align}
\norm{\bar{\partial}_t\rho^{1/2}} +\norm{\nabla \bar{\partial}_t\rho^{1/2}} \le C_2 h^{2\sigma} \norm{ u_t}_{L^\infty (0,t_1;H^{2+\sigma}(\Omega))},
\label{rho-half}\\
\sqrt{\Delta t }\big(\norm{\bar{\partial}_t\rho^{1/2}} + 
 \norm{\nabla \bar{\partial}_t\rho^{1/2}} \big)
\le C_2h^{2\sigma} \norm{ u_t}_{L^2 (0,t_1;H^{2+\sigma}(\Omega))}.
\label{rho-half-l2}
\end{align}\end{subequations}}
Further, the definitions \eqref{p2;splitting2}–\eqref{p2;splitting3} and the approximation property in \eqref{P1 norm_ritz2} lead to
\begin{subequations}\label{eta-half}
\begin{align}
    \norm{ \eta^{1}-\eta^{0}}+\norm{\eta^{1/2}} +h^\sigma \norm{\nabla\eta^{1/2}} &\le 3C_3h^{2\sigma}\norm{\theta}_{L^\infty( 0,t_1; H^{1+\sigma}(\Omega))},\\
    \text{ and } \qquad \norm{ \varrho^{1}-\varrho^{0}}+\norm{ \varrho^{1/2}}+h^\sigma \norm{\nabla \varrho^{1/2}} &\le 3C_3h^{2\sigma}\norm{p}_{L^\infty(0,t_1;H^{1+\sigma}(\Omega))}.
\end{align}
\end{subequations}
      The following lemma, whose proof involves Taylor series expansion and the Cauchy--Schwarz inequality, provides truncation error estimates that will be utilized later in this section.
 \begin{lem}[Truncation error bounds\cite{MR3003381,deka}]\label{p2;trancation-lemma}
For $\varphi \in H^4(0,T;L^2(\Omega))$ , the following inequalities hold
\begin{subequations}
 \begin{align}
(a) & \quad \norm{2\Delta t^{-1}(\bar{\partial}_t \varphi^{1/2}-\varphi_t(0))-\varphi_{tt}^{1/2}} \le  \Delta t\norm{\varphi_{ttt}}_{L^\infty(0,t_1;L^2(\Omega))}, \label{p2;R^0}\\
(b)& \quad 
\norm{\bar{\partial}_t \varphi^{n+1/2}- \varphi_t^{n+1/2}}\lesssim (\Delta t)^{3/2} \norm{\varphi_{ttt}}_{L^2(t_{n},t_{n+1};L^2(\Omega))} \\
&\qquad\qquad\qquad\qquad\qquad\qquad \text{ for }n=0,1,\ldots, N-1 ,\nonumber\\
(c) & \quad \sum_{n=1}^{N-1} 
\norm{\bar{\partial}^2_t \varphi^n -\varphi_{tt}^{n,1/4} }^2 \lesssim (\Delta t)^{3} \norm{\varphi_{tttt}}^2_{L^2(0,T;L^2(\Omega))}.\label{p2;Rn}
 \end{align}\end{subequations}
 \end{lem}
     \noindent Next, we present the initial error estimates for \eqref{3.10}, which will be used to prove the  next theorem. Before proceeding we define the following quantities, all bounded thanks to  Theorems~\ref{weak-sol-thm} and  \ref{thm;regularity}
    \begin{align*}  L_{(u,\theta,p,t_1)}&:= \norm{ u_t}^2_{L^2 (0,t_1;H^{2+\sigma}(\Omega))}+\norm{u_t}^2_{L^\infty(0,t_1;H^{2+\sigma}(\Omega))}+\norm{\theta}_{L^\infty( 0,t_1; H^{1+\sigma}(\Omega))}^2\\
    &\qquad +\norm{p}^2_{L^\infty(0,t_1;H^{1+\sigma}(\Omega))}, \\  M_{(u,\theta,p,t_1)}&:=\norm{u_{ttt}}_{L^\infty(0,t_1;H^1(\Omega))}^2+\norm{\theta_{ttt}}^2_{L^2(0,t_1; L^2(\Omega))}+\norm{p_{ttt}}^2_{L^2(0,t_1; L^2(\Omega))}.\end{align*}
\begin{lem}[Initial error bounds]\label{initial-error-lemm}
 Under the regularity assumptions on given data as stated in Theorems~\ref{existence;thm}-\ref{thm;regularity}, the following estimates are satisfied:
\begin{align*}
&\norm{\bar{\partial}_t\zeta^{1/2}}^2+  a_0\norm{\nabla \bar{\partial}_t\zeta^{1/2}}^2+d_0C_{\rm{Coer}}\norm{\zeta^{1/2}}_h^2+(a_1-|\gamma|/\gamma_0)\norm{\Psi^{1/2}}^2 \\&
      \quad+(a_2-|\gamma|\gamma_0)\norm{\xi^{1/2}}^2
      +\Delta t \big[b_1\norm{\Psi^{1/2}}^2+ c_1 \norm{\nabla \Psi^{1/2}}^2
       + \kappa \norm{\nabla \xi^{1/2}}^2\big]\\
       &
        \lesssim h^{4\sigma} +L_{(u,\theta,p,t_1)} h^{4\sigma}+ M_{(u,\theta,p,t_1)} (\Delta t)^4,
\end{align*}
where the absorbed constant in "$\lesssim$" depends on $C_F$, $C_{\rm{Coer}}$, $C_{\rm{Cont}}$, $C_1$,$C_2$, $C_3$,  $T$, and the model coefficients.
\end{lem}
\begin{proof}
The proof follows in five steps below. \\
\textit{{Step 1} (Key inequality).}
       From the choice of test function  \( v_h = \zeta^{1/2} \in V_h \) in \eqref{zeta1/2} and  the identity  $\zeta^{1/2}=\frac{\zeta^1+\zeta^0}{2}=\frac{
    \zeta^1-\zeta^0}{2}=\frac{\Delta t}{2} \bar{\partial}_t \zeta^{1/2}$ that follows from $\zeta^0=0$ (see \eqref{U0,THETA0,P0}), we obtain 
    \begin{align}
&\norm{\bar{\partial}_t\zeta^{1/2}}^2+a_0  \norm{\nabla \bar{\partial}_t\zeta^{1/2}}^2+d_0a_h(\zeta^{1/2},\zeta^{1/2})- (\nabla(\alpha  \Psi^{1/2}+\beta \xi^{1/2}), \nabla \zeta^{1/2}) \nonumber\\
     &
    \qquad\; = -(\bar{\partial}_t\rho^{1/2},\bar{\partial}_t\zeta^{1/2})-a_0 (\nabla \bar{\partial}_t\rho^{1/2},\nabla \bar{\partial}_t\zeta^{1/2})+ (\nabla (\alpha \eta^{1/2}+\beta \varrho^{1/2}), \nabla \zeta^{1/2}) \nonumber\\
   &
   \qquad\quad+(F^{1/2},(Q-I)\zeta^{1/2})+\frac{1}{2}\Delta t (R_0,\bar{\partial}_t\zeta^{1/2})+\frac{a_0}{2}\Delta t(\nabla R_0,\nabla \bar{\partial}_t\zeta^{1/2}). \label{eq:1}
   \end{align}
    We can then proceed to multiply \eqref{psi1/2}  by $\Delta t/2$ and  choose \( \psi_h =  \Psi^{1/2} \in W_h \) as test function,  and utilize $$\zeta^{1/2}=\frac{\Delta t}{2} \bar{\partial}_t \zeta^{1/2}, \quad \Psi^{1/2}=\frac{\Delta t}{2} \bar{\partial}_t \Psi^{1/2}, \quad \xi^{1/2}=\frac{\Delta t}{2} \bar{\partial}_t \xi^{1/2},$$ (from \eqref{U0,THETA0,P0}) to get 
    \begin{align}
    &a_1\norm{\Psi^{1/2}}^2-\gamma ( \xi^{1/2},  \Psi^{1/2})+\frac{b_1}{2} \Delta t\norm{\Psi^{1/2}}^2+ \frac{c_1}{2}\Delta t \norm{\nabla \Psi^{1/2}}^2+\alpha (\nabla \zeta^{1/2} ,\nabla \Psi^{1/2})\nonumber\\
   &=-\frac{a_1}{2}( \eta^{1}-\eta^0, \Psi^{1/2})+\frac{\gamma}{2} (\varrho^{1}-\varrho^{0}, \Psi^{1/2})-\frac{b_1}{2} \Delta t (\eta^{1/2}, \Psi^{1/2}) \nonumber\\
   &\quad-\frac{\alpha}{2}\Delta t (\nabla \bar{\partial}_t\rho^{1/2} ,\nabla
    \Psi^{1/2})+\frac{a_1}{2}\Delta t(\tau_0, \Psi^{1/2}) \nonumber\\
   &\quad-\frac{\gamma}{2} \Delta t(s_0,  \Psi^{1/2})+\frac{\alpha}{2} \Delta t (\nabla r_0, \nabla  \Psi^{1/2}). \label{eq:2}
     \end{align}
    Similarly, we multiply \eqref{xi1/2}  by $\Delta t/2$,  choose  \( { q_h} =  \xi^{1/2} \in W_h \), { and use $\Psi^{1/2}=\frac{\Delta t}{2} \bar{\partial}_t \Psi^{1/2}, \; \xi^{1/2}= \frac{\Delta t}{2} \bar{\partial}_t \xi^{1/2}$} to arrive at 
     \begin{align}
   & a_2 \norm{\xi^{1/2}}^2-\gamma ( \Psi^{1/2},\xi^{1/2})+\frac{\kappa}{2}\Delta t \norm{\nabla \xi^{1/2}}^2 +\beta (\nabla  \zeta^{1/2},\nabla \xi^{1/2})\nonumber\\
   &=-\frac{a_2}{2} (\varrho^{1}-\varrho^{0},\xi^{1/2})+\frac{\gamma}{2}  ( \eta^{1}-\eta^{0},\xi^{1/2})-\frac{\beta}{2} \Delta t (\nabla \bar{\partial}_t \rho^{1/2},\nabla \xi^{1/2})\nonumber\\
   &
    \quad+ \frac{\Delta t}{2} ( a_2s_0-\gamma\tau_0,\xi^{1/2})+\frac{\beta}{2} \Delta t  (\nabla  r_0, \nabla \xi^{1/2}).\label{eq:3}
    \end{align}
   A summation of  \eqref{eq:1}--\eqref{eq:3} leads to the cancellation of the term $(\nabla(\alpha  \Psi^{1/2}+\beta \xi^{1/2}), \nabla \zeta^{1/2})$. This, the coercivity of \( a_h(\bullet, \bullet) \) from \eqref{P1 a_h_properties}, and an appropriate regrouping of the terms lead to
    \begin{align}
    &\norm{\bar{\partial}_t\zeta^{1/2}}^2+a_0  \norm{\nabla \bar{\partial}_t\zeta^{1/2}}^2+d_0C_{\rm Coer} \|\zeta^{1/2}\|^2_h+{a_1}\norm{\Psi^{1/2}}^2  +{a_2}\norm{\xi^{1/2}}^2 
    \nonumber\\
    &\quad  
   +\frac{\Delta t}{2} \big[ b_1\norm{\Psi^{1/2}}^2+ {c_1} \norm{\nabla \Psi^{1/2}}^2+\kappa \norm{\nabla \xi^{1/2}}^2\big] \nonumber\\
   &\le  (F^{1/2},(Q-I)\zeta^{1/2}) +(\nabla (\alpha \eta^{1/2}+\beta \varrho^{1/2}), \nabla \zeta^{1/2})
    \nonumber\\
   &\quad
    +\big[-(\bar{\partial}_t\rho^{1/2},\bar{\partial}_t\zeta^{1/2}) -a_0 (\nabla (\bar{\partial}_t\rho^{1/2},\nabla \bar{\partial}_t\zeta^{1/2})\nonumber \\
   &\qquad\qquad 
   +\frac{1}{2}\Delta t (R_0,\bar{\partial}_t\zeta^{1/2})+\frac{1}{2}\Delta t (R_0,\nabla \bar{\partial}_t\zeta^{1/2})\big]\nonumber\\
   &\quad
   +\frac{\Delta t}{2}\big[ -\alpha (\nabla \bar{\partial}_t\rho^{1/2},\nabla \Psi^{1/2}) -{\beta} (\nabla \bar{\partial}_t \rho^{1/2},\nabla \xi^{1/2})\nonumber\\
   &\qquad\qquad +{\alpha}(\nabla r_0, \nabla  \Psi^{1/2})+{\beta} (\nabla  r_0, \nabla \xi^{1/2}) \big]
  \nonumber\\
    & \quad
    +\frac{1}{2}\big[-{a_1}( \eta^{1}-\eta^0,  \Psi^{1/2})+{a_1}\Delta t(\tau_0, \Psi^{1/2})+{\gamma}( \varrho^{1}-\varrho^{0}, \Psi^{1/2}) \nonumber\\
   &\qquad\qquad  -{\gamma}\Delta t(s_0,  \Psi^{1/2})-{b_1} \Delta t (\eta^{1/2}, \Psi^{1/2}) \big]\nonumber\\
&\quad
+\frac{1}{2}\big[ -{a_2}(\varrho^{1}-\varrho^{0},\xi^{1/2})+{\gamma} ( \eta^{1}-\eta^{0},\xi^{1/2})+{a_2} \Delta t ( s_0,\xi^{1/2})\nonumber\\
   &\qquad\qquad -{\gamma}\Delta t (\tau_0,\xi^{1/2}) \big] + 2{\gamma} ( \Psi^{1/2},\xi^{1/2})\nonumber\\
  &\quad
:=T_1+T_2+T_3+T_4+T_5+T_6+T_7.\label{p2;sum22}
    \end{align}
\textit{{Step 2} (Bound for $T_1$).}
   An application of Cauchy--Schwarz inequality and the bounds from Lemma~\ref{bhcompanion_lem}(v) yields
    \begin{align*}
        T_1:=(F^{1/2},(Q-I)\zeta^{1/2})\le \norm{F^{1/2}}\norm{(Q-I)\zeta^{1/2}}& \le C_1h^2\norm{F^{1/2}}\norm{\zeta^{1/2}}_h.
    \end{align*}
Then we can utilize Young's inequality with $\epsilon=2 ({d_0C_{\rm{Coer}}})^{-1}$ to show that 
   \begin{align*}
       T_1& \le {C_1^2}({d_0C_{\rm{Coer}}})^{-1}h^4\norm{F^{1/2}}^2+\frac{d_0}{4}C_{\rm{Coer}}\norm{\zeta^{1/2}}^2_h \\&\le {C_1^2C_F^2}({d_0C_{\rm{Coer}}})^{-1}h^4+\frac{d_0}{4}C_{\rm{Coer}}\norm{\zeta^{1/2}}^2_h,
   \end{align*}
   with the bound $\norm{F^{1/2}}\le C_F$ from the regularity result \eqref{norm-F} in the last step.
   
    \noindent \textit{{Step 3} (Bound for $T_2$).}
{Note that $\eta^{1/2}, \varrho^{1/2} \in H^1_0(\Omega)$, and $Q\zeta^{1/2} \in H^2_0(\Omega)$.
     Some elementary manipulations and an integration by parts show 
     \begin{align}
       T_2&:=\alpha (\nabla \eta^{1/2}, \nabla\zeta^{1/2})+\beta (\nabla \varrho^{1/2}, \nabla \zeta^{1/2}) 
       \nonumber \\
       & =
    \alpha (\nabla \eta^{1/2}, \nabla(I-Q)\zeta^{1/2})+\beta (\nabla \varrho^{1/2}, \nabla (I-Q)\zeta^{1/2})\nonumber\\
    &\quad-\alpha ( \eta^{1/2}, \Delta (Q \zeta^{1/2}))-\beta ( \varrho^{1/2}, \Delta (Q\zeta^{1/2})). \label{ist-T2}
     \end{align}
Using Cauchy--Schwarz's  inequality, $\norm{\nabla(I-Q)\zeta^{1/2}} \le C_1 h \|\zeta^{1/2}\|_h$ from  Lemma~\ref{bhcompanion_lem}(v) (with $s=1$ and $v=0$), and  Young's inequality (applied twice with $\epsilon = 8(d_0 C_{\rm Coer})^{-1}$), we can readily bound the first two terms on the right-hand side  of \eqref{ist-T2} as
     \begin{align}
         &\alpha (\nabla \eta^{1/2}, \nabla(I-Q)\zeta^{1/2})+\beta (\nabla \varrho^{1/2}, \nabla (I-Q)\zeta^{1/2}) \nonumber\\
         &\le C_1 h \big(\alpha \norm{\nabla\eta^{1/2}}\norm{\zeta^{1/2}}_h +\beta\norm{\nabla \varrho^{1/2}}\norm{\zeta^{1/2}}_h\big)\nonumber\\
         &\le 4C^2_1h^2({d_0C_{\rm Coer}})^{-1}\big(\alpha^2 \norm{\nabla\eta^{1/2}}^2 +\beta^2\norm{\nabla \varrho^{1/2}}^2\big)+\frac{d_0}{8} C_{\rm Coer}\norm{\zeta^{1/2}}^2_h
        \nonumber\\
        &\le Ch^{2+2\sigma}\big(\norm{\theta}^2_{L^\infty(0,t_1;H^{1+\sigma}(\Omega))}+\norm{p}^2_{L^\infty(0,t_1;H^{1+\sigma}(\Omega))}\big)+\frac{d_0}{8}C_{\rm Coer} \norm{\zeta^{1/2}}^2_h,
        \label{2nd-T2}
     \end{align}
     where we have utilized \eqref{eta-half} in the last inequality.
     
     Note that $\norm{\Delta (Q\zeta^{1/2})} \le \|Q\zeta^{1/2} \|_h $ and  a triangle inequality with Lemma~\ref{bhcompanion_lem}(v) shows
     $\norm{\Delta (Q\zeta^{1/2})} \le\Lambda \|\zeta^{1/2}\|_h$ for $\Lambda >0$.     Therefore, combining this with the Cauchy--Schwarz  and  Young's inequality (as in the last step) lead to
    \begin{align} 
     &-\alpha ( \eta^{1/2}, \Delta (Q\zeta^{1/2}))-\beta ( \varrho^{1/2}, \Delta (Q\zeta^{1/2}))
     \le   \alpha\Lambda\norm{\eta^{1/2}} \norm{\zeta^{1/2}}_h +\beta\Lambda\norm{\varrho^{1/2}} \norm{\zeta^{1/2}}_h \nonumber\\
     &\le 4\Lambda^2 {(d_0C_{\rm Coer})^{-1}}\big(\alpha^2\norm{\eta^{1/2}}^2+\beta^2\norm{\varrho^{1/2}}^2\big)+\frac{d_0}{8}C_{\rm Coer}\norm{\zeta^{1/2}}^2_h\nonumber\\
     &\le Ch^{4\sigma}\big(\norm{\theta}^2_{L^\infty(0,t_1;H^{1+\sigma}(\Omega))}+\norm{p}^2_{L^\infty(0,t_1;H^{1+\sigma}(\Omega))}\big)+\frac{d_0}{8}C_{\rm Coer}\norm{\zeta^{1/2}}^2_h,
     \label{3rd-T2}
      \end{align}
      with estimates in the last step  from \eqref{eta-half}. In addition, 
      a combination of \eqref{2nd-T2}-\eqref{3rd-T2} in \eqref{ist-T2} yields
      \begin{align}
         T_2 
         &\le C (h^{2+2\sigma}+h^{4\sigma})\big( \norm{\theta}^2_{L^\infty(0,t_1;H^{{1+\sigma}}(\Omega))}+\norm{p}^2_{L^\infty(0,t_1;H^{1+{\sigma}}(\Omega))}\big)+\frac{d_0}{4}C_{\rm{Coer}}\norm{\zeta^{1/2}}^2_h. 
     \end{align}
  }
  
  \noindent  \textit{{Step 4} (Bounds for $T_3$--$T_7$).}
The estimates for $T_3$--$T_7$ follow a similar approach.  {We apply the Cauchy--Schwarz and Young's inequalities (with $\epsilon=2,2,1,1$ for the four terms, respectively, in $T_3$}) to reveal the bound 
    \begin{align*}
    T_3&:=-(\bar{\partial}_t\rho^{1/2},\bar{\partial}_t\zeta^{1/2})-a_0 (\nabla \bar{\partial}_t\rho^{1/2},\nabla \bar{\partial}_t\zeta^{1/2})\\
    &\quad+\frac{1}{2}\Delta t(R_0,\bar{\partial}_t\zeta^{1/2})+\frac{a_0}{2}\Delta t(\nabla R_0,\nabla \bar{\partial}_t\zeta^{1/2})\nonumber\\
          &\le C_{T_3} \big( \norm{\bar{\partial}_t\rho^{1/2}}^2+ \norm{\nabla \bar{\partial}_t\rho^{1/2}}^2+{(\Delta t)^2 }\norm{R_0}^2+ {(\Delta t)^2 }\norm{\nabla R_0}^2 \big)  \\
          &\quad+\frac{1}{2}\norm{\bar{\partial}_t\zeta^{1/2}}^2+\frac{a_0}{2}\norm{\nabla \bar{\partial}_t\zeta^{1/2}}^2,\nonumber
          \end{align*}
           where $C_{T_3}=\max\{1,a_0\}$.  
         We then employ \eqref{rho-half} to bound the first two terms and \eqref{tranc-terms} and Lemma~\ref{p2;trancation-lemma} to bound the third and fourth terms on the right-hand side above to obtain
          \begin{align*}
          T_3 &\le C \big( h^{4\sigma} \norm{u_t}^2_{L^\infty(0,t_1;H^{2+\sigma}(\Omega))}+(\Delta t)^4 \norm{u_{ttt}}_{L^\infty(0,t_1;H^1(\Omega))}^2\big)\\
          &\quad
         +\frac{1}{2}\norm{\bar{\partial}_t\zeta^{1/2}}^2+\frac{a_0}{2}\norm{\nabla \bar{\partial}_t\zeta^{1/2}}^2.
         \end{align*}
      \noindent   (Here $C$  denotes a generic constant   independent of the discretization parameters).
       \noindent   Now, similar arguments ({with $\epsilon= 2/c_1, 2/\kappa, 2/c_1, 2/\kappa$ }in the Young's inequalities for  the terms in $T_4$) lead to  
          \begin{align*}
    T_4& :=\frac{\Delta t}{2}\big[ -\alpha (\nabla \bar{\partial}_t\rho^{1/2},\nabla \Psi^{1/2}) -{\beta} (\nabla \bar{\partial}_t \rho^{1/2},\nabla \xi^{1/2})\nonumber\\
    &\quad+{\alpha}(\nabla r_0, \nabla  \Psi^{1/2})+{\beta} (\nabla  r_0, \nabla \xi^{1/2}) \big]\nonumber\\
    &\le C_{T_4} \Big(\Delta t\norm{\nabla \bar{\partial}_t\rho^{1/2}}^2+\Delta t\norm{\nabla r_0}^2\Big)+\frac{c_1}{4}\Delta t \norm{\nabla \Psi^{1/2}}^2+\frac{\kappa}{4} \Delta t \norm{\nabla \xi^{1/2}}^2,\nonumber
    \end{align*}
    with {$C_{T_4}=\max\{\frac{1}{2}\alpha^2c_1^{-1},\frac{1}{2}\beta^2\kappa^{-1}\}$}. An application of \eqref{rho-half-l2} leads to 
    \begin{align*}
    T_4 
    &\le C \Big( h^{4\sigma}\norm{ u_t}^2_{L^2 (0,t_1;H^{2+\sigma}(\Omega))}+(\Delta t)^4 \norm{ u_{ttt}}^2_{L^\infty (0,t_1;H^{1}(\Omega))}\Big)\\
    &\quad+\frac{c_1}{4}\Delta t \norm{\nabla \Psi^{1/2}}^2+\frac{\kappa}{4} \Delta t \norm{\nabla \xi^{1/2}}^2.
    \end{align*}
     A further application of  Cauchy--Schwarz and Young's inequalities (details on the choice of $\epsilon$ in the rest of the proof are skipped for brevity) leads to 
          \begin{align*}
        T_5&:=\frac{1}{2}\big[-{a_1}( \eta^{1}-\eta^0,  \Psi^{1/2})+{a_1}\Delta t(\tau_0, \Psi^{1/2})+{\gamma}( \varrho^{1}-\varrho^{0}, \Psi^{1/2}) \\
        &\qquad\qquad-{\gamma}\Delta t(s_0,  \Psi^{1/2})-{b_1} \Delta t (\eta^{1/2}, \Psi^{1/2})\big]  
       %
        \\
        &\le  C\big(\norm{ \eta^{1}-\eta^{0}}^2+\norm{\eta^{1/2}}^2+\norm{\varrho^{1}-\varrho^{0}}^2+(\Delta t)^2\big[\norm{\tau_0}^2+\norm{s_0}^2\big]\big)\\
        &\quad
        +\frac{a_1-|\gamma|/\gamma_0}{2}\norm{ \Psi^{1/2}}^2+\frac{b_1}{4} \Delta t\norm{\Psi^{1/2}}^2\\
       &\le C\big(h^{4\sigma}\norm{\theta}_{L^\infty( 0,t_1; H^{1+\sigma}(\Omega))}^2 +h^{4\sigma} \norm{p}^2_{L^\infty(0,t_1;H^{1+\sigma}(\Omega))}\\
       &\quad+(\Delta t)^4\norm{\theta_{ttt}}^2_{L^2(0,t_1;L^2(\Omega))}+(\Delta t)^4\norm{p_{ttt}}^2_{L^2(0,t_1;L^2(\Omega))}\big)\nonumber\\
       &
       \quad +\frac{a_1-|\gamma|/\gamma_0}{2}\norm{ \Psi^{1/2}}^2+\frac{b_1}{4} \Delta t\norm{\Psi^{1/2}}^2.
        \end{align*}
       In the last inequality, the bounds for the first three terms on the right-hand side are obtained from \eqref{eta-half}, and the last two from Lemma~\ref{p2;trancation-lemma}, respectively.
       Then, using  similar arguments as above also show that
           \begin{align*}
        T_6&
    := \frac{1}
    {2}\big[ -a_2(\varrho^{1}-\varrho^0,\xi^{1/2})+{\gamma}  ( \eta^{1}-\eta^0,\xi^{1/2})+{a_2} \Delta t ( s_0,\xi^{1/2})-{\gamma}\Delta t (\tau_0,\xi^{1/2}) \big] \nonumber\\
    &\le C\Big(h^{4\sigma}\norm{\theta}_{L^\infty(0,t_1;H^2(\Omega))}^2 +h^{4\sigma}\norm{p}^2_{L^\infty(0,t_1; H^2(\Omega))}\\
    &\quad\quad+(\Delta t)^4\norm{\theta_{ttt}}^2_{L^2(0,t_1; L^2(\Omega))}+(\Delta t)^4\norm{p_{ttt}}^2_{L^2(0,t_1; L^2(\Omega))}\Big)
     +\frac{a_2-|\gamma|\gamma_0}{2}\norm{ \xi^{1/2}}^2.
    \end{align*}
The  Cauchy--Schwarz and Young's inequalities
     lead to
$
T_7:=2{\gamma} ( \Psi^{1/2},\xi^{1/2})\le \frac{|\gamma|}{\gamma_0}\norm{\Psi^{1/2}}^2+|\gamma|\gamma_0\norm{\xi^{1/2}}^2.
    $

   \noindent \textit{{Step 5} (Conclusion). } 
  It suffices to put together the bounds for $T_1$-$T_7$ in 
     \eqref{p2;sum22} and the fact that $h^2 \le |\Omega|^{1-\sigma} h^{2\sigma} \lesssim  h^{2\sigma} $, 
to finish the proof.
      \end{proof}
\begin{rem}\label{rem-super}
The decomposition \eqref{ist-T2} and the analysis in Step 3 of Lemma~\ref{initial-error-lemm} are aimed at proving the superconvergence ($h^{4\sigma}$ rates) of the projected errors $\zeta^{1/2}$ and $\Psi^{1/2}$, $\xi^{1/2}$ in the norms $\norm{\zeta^{1/2}}^2_h$  and $\norm{(\Psi^1, \xi^1)}^2_H$, respectively. This is achieved using the approximation properties of the smoother $Q$ from Lemma~\ref{bhcompanion_lem}(v). The same arguments are also applied in Step 2 of Theorem~\ref{error-estimates thm} to obtain the superconvergence of the projected errors $\zeta^{m+1/2}$ and $\Psi^{m+1/2}$, $\xi^{m+1/2}$ for all $1 \le m \le N-1$ in $\norm{\zeta^{m+1/2}}^2_h$  and $\norm{(\Psi^m, \xi^m)}^2_H$, respectively; and is achieved in \eqref{proj;estimate}. This superconvergence also yields  more elegant lower  $H^s-$ order estimates with $s=0,1$ (resp. $s=0$)  for $u$ (resp. $\theta$ and $p$) established in Corollary~\ref{corr} (resp. Theorem~\ref{error-estimates thm}).
\end{rem}

\subsection{Error estimates}
We present the error estimates for \eqref{fully-discret-scheme}. To do this, we first derive the error equations, which will subsequently be used in Theorem~\ref{error-estimates thm}, with appropriate choices of test functions, to establish the estimates. 

\medskip
\noindent {\bf Error equations.} 
A linear combination of the equations in the system \eqref{weak form}}, evaluated at $t=t_{n-1}$, $t=t_n,$ and $t=t_{n+1}$, for $n=1,2,\cdots, N-1$,  yields
\begin{subequations}
\begin{align}
    & (u_{tt}^{n,1/4},Qv_h)+a_0 (\nabla u_{tt}^{n,1/4},\nabla Qv_h)+d_0(\nabla ^2 u^{n,1/4}, \nabla^2 Qv_h)\nonumber  \\
    & \quad-\alpha (\nabla \theta^{n,1/4}, \nabla Qv_h)
    -\beta (\nabla p^{n,1/4}, \nabla Qv_h) =(f^{n,1/4},Qv_h), \label{u^n}\\
    &a_1(\theta_t^{n+1/2}, \psi_h)-\gamma (p_t^{n+1/2}, \psi_h)+b_1(\theta^{n+1/2},\psi_h) \nonumber
\\&\quad+c_1 (\nabla \theta^{n+1/2},\nabla \psi_h) +\alpha (\nabla u_t^{n+1/2}, \nabla \psi_h)=(\phi^{n+1/2},\psi_h), \label{theta^n} \\
  &  a_2 (p_t^{n+1/2},q_h)-\gamma (\theta_t^{n+1/2},q_h)+\kappa (\nabla p^{n+1/2},\nabla q_h)\nonumber\\
   &\quad +\beta (\nabla u_t^{n+1/2},\nabla q_h)=(g^{n+1/2},q_h),\label{p^n}
    \end{align}
    \end{subequations}
    $\text{for all } v_h \in V_h $ and $\psi_h,q_h \in  W_h$, 
   where as earlier we have used  Range$(Q) \subset H^2_0(\Omega)$ and $W_h \subset H^1_0(\Omega)$ .
   
\noindent   Next we recall the definition of $F =f(t,\bx)-u_{tt}+a_0 \Delta u_{tt}-\alpha \Delta \theta-\beta \Delta p $ from \eqref{F} and define the truncation terms  as follows:
   \begin{subequations}\label{Tr-n}
 \begin{align}
  &R^n:= \bar{\partial}_t^2 u^{n}- u_{tt}^{n,1/4},\; r^n:=\bar{\partial}_tu^{n+1/2}- u_t^{n+1/2},\\
  &\;\tau^n:=\bar{\partial}_t{\theta^{n+1/2}}-\theta_t^{n+1/2}, \text{ and }s^n:=\bar{\partial}_t{p^{n+1/2}}-p_t^{n+1/2} .
    \end{align}
    \end{subequations}
 Subtracting   \eqref{un}  from \eqref{u^n} and employing $a_h(\mathcal{R}_hu^{n,1/4}, v_h) =(\nabla^2 u^{n,1/4}, \nabla^2 Qv_h )$ from \eqref{P1 ritz_projection}, we readily obtain
 \begin{align*}
    &(\bar{\partial}_t^2 u^n-\bar{\partial}^2_tU^n,v_h)+a_0 (\nabla (\bar{\partial}_t^2 u^n-\bar{\partial}^2_tU^n),\nabla v_h)
    +d_0 a_h(\mathcal{R}_hu^{n,1/4}-U^{n,1/4}, v_h)
     \\
    &-\alpha (\nabla (\theta^{n,1/4}-\Theta^{n,1/4}), \nabla v_h)-\beta (\nabla (p^{n,1/4}-P^{n,1/4}), \nabla v_h) \\
    &     =(F^{n,1/4},(Q-I)v_h) +(R^n,v_h)+a_0(\nabla R^n,\nabla v_h).
   \end{align*}
An appeal to the splitting in \eqref{spilit} leads to the  first relation of the error equation  of the system as
\begin{align}
    &
    (\bar{\partial}_t^2 \zeta^n,v_h)+a_0 (\nabla \bar{\partial}_t^2 \zeta^n,\nabla v_h)+d_0 a_h(\zeta^{n,1/4}, v_h)-\alpha (\nabla \Psi^{n,1/4}, \nabla v_h)\nonumber\\
    & -\beta (\nabla \xi^{n,1/4}, \nabla v_h)
    =-(\bar{\partial}_t^2 \rho^n,v_h)-a_0 (\nabla \bar{\partial}_t^2 \rho^n,\nabla v_h)+\alpha (\nabla \eta^{n,1/4}, \nabla v_h)\nonumber\\
    &
   +\beta (\nabla \varrho^{n,1/4}, \nabla v_h)
    +(F^{n,1/4},(Q-I)v_h)+( R^n,v_h)+a_0(\nabla  R^n,\nabla v_h).\label{errror-eqn1}
    \end{align}
We can then subtract \eqref{thetan} from  \eqref{theta^n}  and utilize the definition of $\Pi_h$ from 
\eqref{p2;ritzprojection2} to obtain
 \begin{align*}
    &a_1(\bar{\partial}_t(\theta^{n+1/2}-{\Theta}^{n+1/2}), \psi_h)\\
    &\quad-\gamma (\bar{\partial}_t(p^{n+1/2}-{P^{n+1/2}}), \psi_h)+b_1(\theta^{n+1/2}-\Theta^{n+1/2},\psi_h)
   \\
   & \quad+c_1 (\nabla (\Pi_h\theta^{n+1/2}-\Theta^{n+1/2}),\nabla \psi_h)+\alpha (\nabla \bar{\partial}_t(u^{n+1/2}-U^{n+1/2}) ,\nabla \psi_h)  \\
   & =a_1(\tau^n, \psi_h)-\gamma (s^n, \psi_h)+\alpha (\nabla r^n, \nabla \psi_h).
   \end{align*}
The splitting from \eqref{spilit} reveals the  second  relation in the error equation as
\begin{align}
    &a_1(\bar{\partial}_t\Psi^{n+1/2}, \psi_h)-\gamma (\bar{\partial}_t{\xi^{n+1/2}}, \psi_h)+b_1(\Psi^{n+1/2},\psi_h)\nonumber\\
    &\quad+c_1 (\nabla \Psi^{n+1/2},\nabla \psi_h)+\alpha (\nabla \bar{\partial}_t\zeta^{n+1/2} ,\nabla \psi_h)\nonumber\\
    &\nonumber
   =-a_1(\bar{\partial}_t\eta^{n+1/2}, \psi_h) +\gamma (\bar{\partial}_t{\varrho^{n+1/2}}, \psi_h) -b_1(\eta^{n+1/2},\psi_h)
   \\
   &\quad-\alpha (\nabla \bar{\partial}_t \rho^{n+1/2} ,\nabla \psi_h)+a_1 (\tau^n, \psi_h)-\gamma (s^n, \psi_h)+\alpha (\nabla r^n, \nabla \psi_h).\label{errror-eqn2}
    \end{align}
Furthermore, we subtract \eqref{pn} from \eqref{p^n}, and utilize the same arguments as above along with  \eqref{spilit} to obtain  the third error relation of the error equation as
\begin{align}
&
    a_2 (\bar{\partial}_t \xi^{n+1/2},q_h)-\gamma (\bar{\partial}_t\Psi^{n+1/2},q_h)  +\kappa (\nabla \xi^{n+1/2},\nabla q_h)+\beta (\nabla \bar{\partial}_t \zeta^{n+1/2},\nabla q_h)\nonumber\\
    &
    =-a_2 (\bar{\partial}_t \varrho^{n+1/2},q_h)+\gamma (\bar{\partial}_t \eta^{n+1/2},q_h)
    -\beta (\nabla \bar{\partial}_t \rho^{n+1/2},\nabla q_h)\nonumber\\
    &\quad
    +a_2 (s^n,q_h)-\gamma(\tau^n,q_h)+\beta (\nabla r^n,\nabla q_h). \label{errror-eqn3}
    \end{align}

\noindent{\bf Some useful bounds.} 
Next we present some bounds that will be useful in the proof of Theorem~\ref{error-estimates thm}. 
The estimates from \eqref{P1 norm_ritz} for \( \rho(t) \), \eqref{P1 norm_ritz2} for \( \eta(t) \) and $\varrho(t)$ are used to demonstrate {the bounds}, for any $1 \le m \le N$ 
\begin{subequations}\label{one}
\begin{align}
& \norm{ \eta^{m,1/4}}+h^{ \sigma} \norm{\nabla \eta^{m,1/4}} \le C_3 h^{2\sigma}\norm{\theta}_{L^\infty(t_{m-1},t_{m+1};H^{1+\sigma}(\Omega))},\label{m-1}\\
&\norm{ \varrho^{m,1/4}}+h^{ \sigma} \norm{\nabla \varrho^{m,1/4}} \le C_3h^{2\sigma}\norm{p}_{L^\infty(t_{m-1},t_{m+1};H^{1+\sigma}(\Omega))}, \label{m-2}\\
&\Big(\Delta t \sum_{n=1}^m \norm{\nabla\bar{\partial}_t \rho^{n+1/2}}^2\Big)^{1/2}+\Big(\Delta t\sum_{n=1}^m \norm{\eta^{n+1/2}}^2\Big)^{1/2}\nonumber \\
     &\quad+\Big(\Delta t\sum_{n=1}^m \norm{\bar{\partial}_t 
     \eta^{n+1/2}}^2\Big)^{1/2} +\Big(\Delta t \sum_{n=1}^m  \norm{\bar{\partial}_t \varrho^{n+1/2}}^2\Big)^{1/2}\nonumber\\
     & \lesssim T^{1/2} h^{2\sigma} \big[ \norm{u_t}_{L^\infty(0,T;H^{2+\sigma}(\Omega))}+\norm{\theta}_{L^\infty(0,T;H^{1+\sigma}(\Omega))}\nonumber\\
     &\qquad\qquad+\norm{\theta_t}_{L^\infty(0,T;H^{1+\sigma}(\Omega))}+\norm{p_t}_{L^\infty(0,T;H^{1+\sigma}(\Omega))}\big],\label{m-3}
\end{align}
\end{subequations}
where in last  inequality we have used $m \Delta t \le T$.
The Taylor series {estimate }$\norm{\bar{\partial}_t^2 \rho^n}^2  \le \frac{2}{3}(\Delta t)^{-1}\int_{t_{n-1}}^{t_{n+1}} \norm{\rho_{tt}(t)}^2 \dt$ (resp. $\norm{\nabla \bar{\partial}_t^2 \rho^n}^2  \le \frac{2}{3}(\Delta t)^{-1}\int_{t_{n-1}}^{t_{n+1}} \norm{\nabla \rho_{tt}(t)}^2 \dt$) along with \eqref{P1 norm_ritz} reveals that 
\begin{align}
  \Big(\Delta t \sum_{n=1}^m\big(\norm{\bar{\partial}_t^2 \rho^n}^2+a_0\norm{\nabla \bar{\partial}_t^2 \rho^n}^2\big)\Big)^{1/2}
  &\lesssim
h^{2\sigma}\norm{u_{tt}}_{L^2(0,T;H^{2+\sigma}(\Omega))}.  \label{rho-n}
\end{align}
{Also,} the definition \eqref{Tr-n}, and the truncation  estimates from Lemma~\ref{p2;trancation-lemma}  are given by
\begin{subequations}\label{truncation}
\begin{align}
&  \Big(\Delta t\sum_{n=1}^m \norm{\ R^n}^2\Big)^{1/2} \lesssim (\Delta t)^2  \norm{u_{tttt}} _{L^2(0,T;L^2(\Omega))}, \nonumber \\
& \Big(\Delta t \sum_{n=1}^m\norm{\nabla R^n}^2\Big)^{1/2} \lesssim  (\Delta t)^2  \norm{\nabla u_{tttt}}_{L^2(0,T;L^2(\Omega))}, \quad \text{and} \label{R-n}\\
  &\Big(\Delta t \sum_{n=1}^m \norm{\nabla r^n}^2\Big)^{1/2}+\Big(\Delta t \sum_{n=1}^m \norm{\tau^n}^2\Big)^{1/2}+\Big(\Delta t \sum_{n=1}^m \norm{s^n}^2\Big)^{1/2} \nonumber \\
&\lesssim (\Delta t)^2\Big[\norm{u_{ttt}}_{L^2(0,T;H^1(\Omega))}+\norm{\theta_{ttt}}_{L^2(0,T;L^2(\Omega))}+\norm{p_{ttt}}_{L^2(0,T;L^2(\Omega))}\Big].\label{rnsn}
 \end{align}
 \end{subequations}
 Note that $\sum_{k=1}^\ell\norm{b^{k-1}+b^k}^2 \le 2 \sum_{k=1}^\ell\norm{b^{k-1}}^2+2 \sum_{k=1}^\ell\norm{b^k}^2 \le 4 \sum_{k=1}^\ell\norm{b^{k-1}}^2+2\norm{b^\ell}^2$, with addition of $2\norm{b^0}^2$ on the right-hand side in the last expression.
This and an application of Cauchy--Schwarz and Young's inequalities lead to
\begin{align}
     \pm \sum_{k=1}^\ell  (a^k,b^{k-1}+b^k)& \le \frac{\epsilon}{2}
    \sum_{k=1}^\ell \norm{a^k}^2+\frac{1}{2\epsilon}\sum_{k=1}^\ell\norm{b^{k-1}+b^k}^2\nonumber\\
     & \le \frac{\epsilon}{2}
    \sum_{k=1}^\ell \norm{a^k}^2+ \frac{2}{\epsilon}\sum_{k=1}^\ell \norm{b^{k-1}}^2+ \frac{1}{\epsilon}\norm{b^{\ell}}^2.\label{an-bn}
\end{align}


\noindent{\bf Main result.} Before proceeding to establish the error estimates at $t=t_2,t_3, \cdots, t_N$, we first  note that the following quantities are bounded, thanks to Table~\ref{table:summary}
\begin{subequations}  \label{lm}
  \begin{align} L_{(u,\theta,p,T)}&:=\norm{u}^2_{L^\infty(0,T;H^{2+\sigma}(\Omega))}+\norm{u_t}^2_{L^\infty(0,T;H^{2+\sigma}(\Omega))}+\norm{u_{tt}}^2_{L^2(0,T;H^{2+\sigma}(\Omega))}\nonumber\\
    &\quad+\norm{\theta}^2_{L^\infty(0,T;H^{1+\sigma}(\Omega))}
    + \norm{\theta_t}^2_{L^\infty(0,T;H^{1+\sigma}(\Omega))}\nonumber\\
    &\quad+\norm{p}^2_{L^\infty(0,T;H^{1+\sigma}(\Omega))}+ \norm{p_t}^2_{L^\infty(0,T;H^{1+\sigma}(\Omega))},\\   
M_{(u,\theta,p,T)}&:=\norm{u_{ttt}}^2_{L^2(0,T;H^1(\Omega))}+\norm{u_{tttt}}^2_{L^2(0,T;H^1(\Omega))}\nonumber\\
&\quad+\norm{\theta_{ttt}}^2_{L^2(0,T;L^2(\Omega))}+\norm{p_{ttt}}^2_{L^2(0,T;L^2(\Omega))}.\label{last}
    \end{align}
    \end{subequations}

\begin{thm}[Error estimates] \label{error-estimates thm}
Under the regularity assumptions on given data as stated in Theorem \ref{thm;regularity},  for $1 \le m \le N-1$, the following estimates are satisfied:
   \begin{align*}
   &\norm{\bar{\partial}_t (u^{m+1/2}-U^{m+1/2})}^2+a_0\norm{\nabla \bar{\partial}_t (u^{m+1/2}-U^{m+1/2})}^2\\
   &\quad+{(a_1-|\gamma|/\gamma_0)} \norm{\theta^{m+1}-\Theta^{m+1}}^2
    +{(a_2-|\gamma|\gamma_0)}\norm{p^{m+1}-P^{m+1}}^2\nonumber\\ 
    & \quad+{d_0}h^{2\sigma}\norm{u^{m+1/2}-U^{m+1/2}}^2_h+h^{2\sigma}\norm{(\theta^{m}-\Theta^{m},p^{m}-P^{m})}^2_{H}  \\
    &\lesssim h^{4\sigma}+\big[L_{(u,\theta,p,t_1)}+L_{(u,\theta,p,T)}\big]h^{4\sigma} +M_{(u,\theta,p,T)}(\Delta t)^4\nonumber,
   \end{align*}
   where the absorbed constant in "$\lesssim$" depends on $C_{\rm{Coer}}$, $C_{\rm{Cont}}$, $C_1$,$C_2$, $C_3$,  $T$, and the model coefficients.
\end{thm}
\begin{proof}
The proof is divided into six steps-the first step derives a key inequality, this is followed by bounds for the terms in the key inequality in Steps 2-5, and Step 6 consolidates the proof.

\noindent \textit{{Step 1} (Key inequality).}
Let us multiply \eqref{errror-eqn1} by $2 \Delta t$, choose $v_h= \delta_t \zeta^n$ and utilize the identities 
\eqref{p2;stb1}-\eqref{p2;stb3}. 
This yields 
\begin{align*}
     &
    \norm{\bar{\partial}_t \zeta^{n+1/2}}^2-\norm{\bar{\partial}_t \zeta^{n-1/2}}^2+a_0\norm{\nabla \bar{\partial}_t \zeta^{n+1/2}}^2-a_0\norm{\nabla\bar{\partial}_t \zeta^{n-1/2}}^2\\
    &\quad
    +d_0 a_h(\zeta^{n+1/2},\zeta^{n+1/2})-d_0 a_h(\zeta^{n-1/2},\zeta^{n-1/2})\nonumber\\
    &
    =2 \Delta t \Big[(F^{n,1/4},(Q-I)\delta_t \zeta^n)\\
    & \qquad \qquad  
     +(\alpha \nabla \eta^{n,1/4} +\beta \nabla \varrho^{n,1/4}, \nabla \delta_t \zeta^n)+  (\alpha\nabla \Psi^{n,1/4} +\beta \nabla \xi^{n,1/4}, \nabla \delta_t \zeta^n) \nonumber\\
    &
 \qquad \qquad     -(\bar{\partial}_t^2 \rho^n,\delta_t \zeta^n)-a_0 (\nabla \bar{\partial}_t^2 \rho^n,\nabla \delta_t \zeta^n)
     +( R^n,\delta_t \zeta^n)+a_0(\nabla  R^n,\nabla \delta_t \zeta^n)\Big] \\
     & = 2  (F^{n,1/4},(Q-I)(\zeta^{n+1/2}-\zeta^{n-1/2}))\\
    &
 \qquad \qquad  +2 (\alpha \nabla \eta^{n,1/4} +\beta \nabla \varrho^{n,1/4},\nabla (\zeta^{n+1/2}-\zeta^{n-1/2}) )
     \nonumber \\
  & \qquad \qquad
+  \Delta t(\alpha\nabla \Psi^{n,1/4} +\beta \nabla \xi^{n,1/4},\bar{\partial}_t (\zeta^{n+1/2} +  \zeta^{n-1/2}) )\\
    &
 \qquad \qquad  
 - \Delta t (\bar{\partial}_t^2 \rho^n - R^n ,\bar{\partial}_t (\zeta^{n+1/2} +  \zeta^{n-1/2})) \nonumber \\
    &  \qquad \qquad - \Delta t a_0 (\nabla( \bar{\partial}_t^2 \rho^n - R^n),\nabla (\bar{\partial}_t (\zeta^{n+1/2} +  \zeta^{n-1/2})))
\end{align*}
with $\Delta t \delta_t \zeta^n=\zeta^{n+1/2}-\zeta^{n-1/2}$ for the first two terms on the right-hand side and
$\delta_t \zeta^n = \frac{1}{2} (\bar{\partial}_t \zeta^{n+1/2} + \bar{\partial}_t \zeta^{n-1/2})$ for the remaining terms.
Next, we again multiply the equation \eqref{errror-eqn2} by $2 \Delta t$ and choose   $\psi_h= \Psi^{n+1/2}$ as the test function. Then utilize \eqref{p2;stab5}(i) 
to obtain
\begin{align*}
 & {a_1} \norm{\Psi^{n+1}}^2-{a_1}\norm{\Psi^{n}}^2\\
 &\quad+2\Delta t \Big[{b_1}\norm{\Psi^{n+1/2}}^2+{c_1} \norm{\nabla \Psi^{n+1/2}}^2+\alpha   (\nabla \bar{\partial}_t\zeta^{n+1/2} ,\nabla \Psi^{n+1/2})\Big]\nonumber\\
    &
    =2\Delta t \Big[-a_1(\bar{\partial}_t\eta^{n+1/2}, \Psi^{n+1/2})+\gamma (\bar{\partial}_t{\xi^{n+1/2}}, \Psi^{n+1/2})+\gamma (\bar{\partial}_t{\varrho^{n+1/2}}, \Psi^{n+1/2}) \nonumber \\
    &
    \qquad\qquad-b_1(\eta^{n+1/2},\Psi^{n+1/2})
    -\alpha (\nabla \bar{\partial}_t\rho^{n+1/2} ,\nabla \Psi^{n+1/2})
   +a_1 (\tau^n, \Psi^{n+1/2})\nonumber \\
    &
    \qquad\qquad-\gamma (s^n, \Psi^{n+1/2})+\alpha (\nabla r^n, \nabla \Psi^{n+1/2})\Big].
    \end{align*}
   Similarly, we multiply \eqref{errror-eqn3} by $2 \Delta t$, select the test function $q_h= \xi^{n+1/2}$, and employ \eqref{p2;stab5}(ii) to get
    \begin{align*}
    &{a_2} \norm{\xi^{n+1}}^2-{a_2}\norm{\xi^{n}}^2+2\Delta t\Big[{\kappa} \norm{\nabla \xi^{n+1/2}}^2+\beta  (\nabla \bar{\partial}_t \zeta^{n+1/2},\nabla \xi^{n+1/2})
   \Big]\nonumber\\
    &
   =2\Delta t \Big[-a_2 (\bar{\partial}_t \varrho^{n+1/2},\xi^{n+1/2})+\gamma   (\bar{\partial}_t{\Psi}^{n+1/2},\xi^{n+1/2}) +\gamma (\bar{\partial}_t \eta^{n+1/2},\xi^{n+1/2})\nonumber\\
    &
   \quad
-\beta (\nabla \bar{\partial}_t \rho^{n+1/2},\nabla \xi^{n+1/2})
    +a_2 (s^n,\xi^{n+1/2})-\gamma(\tau^n,\xi^{n+1/2})+\beta (\nabla r^n,\nabla \xi^{n+1/2})\Big].
\end{align*}
After adding the last three displayed equations (after multiplying the first equation by 4)  and summing  for $n = 1, 2, \cdots, m$, where $1 \le m \le N-1$,  we can then  use \eqref{P1 a_h_properties} and \eqref{t-norm} to produce
\begin{align}
   & 4\norm{\bar{\partial}_t \zeta^{m+1/2}}^2+4a_0\norm{\nabla \bar{\partial}_t \zeta^{m+1/2}}^2+   4d_0 C_{\rm Coer}\norm{\zeta^{m+1/2}}^2_h\nonumber\\
   &\quad+{a_1} \norm{\Psi^{m+1}}^2+{a_2}\norm{\xi^{m+1}}^2 
 +2 \norm{(\Psi^{m},\xi^{m})}_H^2\nonumber\\
  &\le
  8 \sum_{n=1}^m \Big[( F^{n,1/4}, (Q-I)(\zeta^{n+1/2} -\zeta^{n-1/2}))\nonumber\\
  &\qquad\qquad\qquad+(\nabla (\alpha \eta^{n,1/4}+\beta \varrho^{n,1/4}), \nabla ( \zeta^{n+1/2}- \zeta^{n-1/2}))\Big]
  \nonumber\\
  &\quad
   + \Delta t \sum_{n=1}^m\Big[4(\alpha \nabla \Psi^{n,1/4}+\beta \nabla \xi^{n,1/4}, \nabla( \bar{\partial}_t\zeta^{n+1/2}+\bar{\partial}_t\zeta^{n-1/2})\nonumber\\
  &\qquad\qquad\qquad-  2( \alpha \nabla \Psi^{n+1/2}+\beta \nabla \xi^{n+1/2},\nabla \bar{\partial}_t\zeta^{n+1/2}) \Big]\nonumber\\
  &\quad
  -4\Delta t \sum_{n=1}^m  \Big[ (\bar{\partial}_t^2 \rho^n - R^n ,\bar{\partial}_t (\zeta^{n+1/2} +  \zeta^{n-1/2}))\nonumber\\
  &\qquad\qquad\qquad + a_0 (\nabla( \bar{\partial}_t^2 \rho^n - R^n),\nabla (\bar{\partial}_t (\zeta^{n+1/2} +  \zeta^{n-1/2})))\Big]\nonumber\\
 &\quad+ 2\Delta t \sum_{n=1}^m\Big[-a_1(\bar{\partial}_t\eta^{n+1/2}, \Psi^{n+1/2})
    +\gamma (\bar{\partial}_t{\varrho^{n+1/2}}, \Psi^{n+1/2})\nonumber\\
  &\qquad\qquad\qquad-b_1(\eta^{n+1/2},\Psi^{n+1/2}) 
   -\alpha (\nabla \bar{\partial}_t  \rho^{n+1/2} ,\nabla \Psi^{n+1/2}) +a_1 (\tau^n, \Psi^{n+1/2})\nonumber\\
  &\qquad\qquad\qquad-\gamma (s^n, \Psi^{n+1/2})+\alpha (\nabla r^n, \nabla \Psi^{n+1/2})\Big]\nonumber\\
    &\quad +2\Delta t \sum_{n=1}^m\Big[-a_2 (\bar{\partial}_t \varrho^{n+1/2},\xi^{n+1/2})\nonumber\\
    &\qquad\qquad\qquad+\gamma (\bar{\partial}_t\eta ^{n+1/2},\xi^{n+1/2})-\beta (\nabla\bar{\partial}_t \rho^{n+1/2},\nabla \xi^{n+1/2})\nonumber\\
        &\qquad\qquad\qquad
    +a_2 (s^n,\xi^{n+1/2})-\gamma(\tau^n,\xi^{n+1/2})+\beta (\nabla r^n,\nabla \xi^{n+1/2})\Big]\nonumber\\
    &\quad+2\gamma\sum_{n=1}^m \Big[ (\Psi^{n+1}, \xi^{n+1})-(\Psi^{n},\xi^{n} )\Big]\nonumber\\
    &\quad + \Big[4\norm{\bar{\partial}_t \zeta^{1/2}}^2+4a_0\norm{\nabla\bar{\partial}_t \zeta^{1/2}}^2+4d_0C_{\rm Cont} \norm{\zeta^{1/2}}^2_h+{a_1}\norm{\Psi^{1}}^2+{a_2}\norm{\xi^{1}}^2 \Big]\nonumber\\
    &\quad =: T_1+T_2+T_3+T_4+T_5+T_6+T_7.\label{p2;error-all-terms}
   \end{align}
\textit{{Step 2} (Bound for $T_1$).}
 The summation by parts formula 
 $\sum_{n=1}^m g^{n,1/4}(h^{n+1/2}-h^{n-1/2})= g^{m,1/4}h^{m+1/2}-g^{1,1/4}h^{1/2}-\sum_{n=1}^{m-1} (g^{n+1,1/4}-g^{n,1/4})h^{n+1/2},$
{together with the observation} $g^{n+1,1/4}-g^{n,1/4}=1/4\int_{t_{n-1}}^{t_{n+2}}g_t \dt+1/4 \int_{t_{n}}^{t_{n+1}}g_t \dt$ reveals  
\begin{align}
T_1 & =   8 ( F^{m,1/4}, (Q-I) \zeta^ {m+1/2})- 8(  F^{1,1/4}, (Q-I) \zeta^ {1/2} ) \nonumber\\
&\;\;
-2\sum_{n=1}^{m-1}\int_{t_{n-1}}^{t_{n+2}} (F_{t}(t) ,
   (Q-I)\zeta^{n+1/2})\dt-2\sum_{n=1}^{m-1}\int_{t_{n}}^{t_{n+1}}(F_{t}(t) ,
   (Q-I)\zeta^{n+1/2})\dt
   \nonumber\\
 &\;\;
+8(\nabla (\alpha \eta^{m,1/4}+\beta\varrho^{m,1/4}), \nabla  \zeta^{m+1/2}) -8(\nabla(\alpha  \eta^{1,1/4} +\beta \varrho^{1,1/4}), \nabla  \zeta^{1/2})
\nonumber\\
         &\;\;
        -2\sum_{n=1}^{m-1}\int_{t_{n-1}}^{t_{n+2}}(\nabla (\alpha \eta_t(t)+\beta \varrho_t(t)), \nabla \zeta^{n+1/2})\dt\nonumber\\
         &\;\;
         -2\sum_{n=1}^{m-1}\int_{t_{n}}^{t_{n+1}}(\nabla(\alpha \eta_t(t)+\beta \varrho_t(t)), \nabla \zeta^{n+1/2})\dt.\label{eq:t1}
     \end{align}
     An application of Cauchy--Schwarz inequality, {the bounds form Lemma~\ref{bhcompanion_lem}(v),  the}  Young inequality (with $\epsilon = 4(d_0C_{\rm{Coer}})^{-1},4(d_0C_{\rm{Coer}})^{-1}$, respectively for the first two terms), and \eqref{norm-F}(i) result in
\begin{align*} 
    &8( F^{m,1/4}, (Q-I) \zeta^ {m+1/2}) \le 8C_1h^2\norm{F^{m,1/4}}\norm{ \zeta^ {m+1/2}}_h\\
 & \qquad\qquad\qquad\qquad\qquad\quad\;\le {16C_1^2}C_F^2(d_0C_{\rm{Coer}})^{-1}h^4 
    +{{d_0}}C_{\rm{Coer}}\norm{ \zeta^ {m+1/2}}^2_h,\\
    &
     8(  F^{1,1/4}, (I-Q) \zeta^ {1/2} )\le 8C_1h^2 \norm{F^{1,1/4}}\norm{ \zeta^ {1/2}}_h\\
     & \qquad\qquad\qquad\qquad\quad\;\;\le {16C_1^2}C_F^2(d_0C_{\rm{Coer}})^{-1}h^4 
     +{{d_0}}C_{\rm{Coer}}\norm{ \zeta^ {1/2}}^2_h . 
     \end{align*}
     Analogous arguments with $\epsilon = 6T(d_0C_{\rm{Coer}})^{-1},2T(d_0C_{\rm{Coer}})^{-1}$, respectively for the third and fourth terms give 
     \begin{align*}
    &
    2 \int_{t_{n-1}}^{t_{n+2}}(F_{t}(t) ,
    (I-Q)\zeta^{n+1/2})\dt\\
    &\le 6{{TC_1^2}}(d_0C_{\rm{Coer}})^{-1}h^4 \norm{F_t }^2_{L^2(t_{n-1},t_{n+2} ;L^2(\Omega))}+\frac{{d_0 }}{2T}C_{\rm{Coer}}\Delta t\norm{ \zeta^ {n+1/2}}^2_h,\\
    &  
     2\int_{t_{n}}^{t_{n+1}}(F_{t}(t) ,
     (I-Q)\zeta^{n+1/2})\dt\\
     &\le {{2 TC_1^2}}(d_0C_{\rm{Coer}})^{-1}h^4 \norm{F_t }^2_{L^2(t_{n},t_{n+1} ;L^2(\Omega))}+\frac{{d_0 }}{2T}C_{\rm{Coer}}\Delta t\norm{ \zeta^ {n+1/2}}^2_h.
\end{align*}
In the last two displayed inequalities, we have used $\int_{t_{n-1}}^{t_{n+2}}\norm{ \zeta^ {n+1/2}}^2_h\dt =3\Delta t \norm{ \zeta^ {n+1/2}}^2_h$ and $\int_{t_{n}}^{t_{n+1}}\norm{ \zeta^ {n+1/2}}^2_h\dt =\Delta t \norm{ \zeta^ {n+1/2}}^2_h$, respectively. Hence, it follows from \eqref{norm-F}(ii) that
\begin{align*}
  & 2 \sum_{n=1}^{m-1} \int_{t_{n-1}}^{t_{n+2}}(F_{t}(t) ,
    (I-Q)\zeta^{n+1/2})\dt\\
    &\le 6{{TC_1^2}}(C_F')^2(d_0C_{\rm{Coer}})^{-1}h^4 +\frac{{d_0 }}{2T}C_{\rm{Coer}}\Delta t
    \sum_{n=1}^{m-1}\norm{ \zeta^ {n+1/2}}^2_h,\\
    &  
     2\sum_{n=1}^{m-1} \int_{t_{n}}^{t_{n+1}}(F_{t}(t) ,
     (I-Q)\zeta^{n+1/2})\dt\\
     &\le {{2 TC_1^2}}(C_F')^2(d_0C_{\rm{Coer}})^{-1}h^4 +\frac{{d_0 }}{2T}C_{\rm{Coer}}\Delta t\sum_{n=1}^{m-1}\norm{ \zeta^ {n+1/2}}^2_h.
\end{align*}
Elementary manipulations analogous to \eqref{ist-T2} show
\begin{align}
      &8(\nabla ( \alpha\eta^{m,1/4}+\beta \varrho^{m,1/4}), \nabla  \zeta^{m+1/2})\nonumber\\
      & =
    8 (\nabla ( \alpha\eta^{m,1/4}+\beta \varrho^{m,1/4}), \nabla(I-Q)\zeta^{m+1/2})\nonumber\\ &\qquad-8 (  \alpha\eta^{m,1/4}+\beta \varrho^{m,1/4}, \Delta (Q \zeta^{m+1/2})). \label{ist-T2-ful}
     \end{align}
Then, similar to \eqref{2nd-T2}, utilize Cauchy--Schwarz's  inequality, $\norm{\nabla(Q-I)\zeta^{m+1/2}} \le C_1 h \|\zeta^{m+1/2}\|_h$ from  Lemma~\ref{bhcompanion_lem}(v) (with $s=1$ and $v=0$), and  Young's inequality for first two terms on the right-hand side  of \eqref{ist-T2-ful} as
     \begin{align}
       &  8 (\nabla ( \alpha\eta^{m,1/4}+\beta \varrho^{m,1/4}), \nabla(I-Q)\zeta^{m+1/2}) \nonumber\\
        & \le 8C_1 h \big(\alpha \norm{\nabla\eta^{m,1/4}}\norm{\zeta^{m+1/2}}_h +\beta\norm{\nabla \varrho^{m,1/4}}\norm{\zeta^{m+1/2}}_h\big)\nonumber\\
        &\le C h^2\big(\norm{\nabla \eta^{m,1/4}}^2 + \norm{\nabla \varrho^{m,1/4}}^2\big) +\frac{d_0}{2}C_{\rm{Coer}}\norm{\zeta^{m+1/2}}_h^2.\label{T2-ist-ful}
     \end{align}
     For third and fourth terms of right-hand side of \eqref{ist-T2-ful}, we note that $\norm{\Delta (Q\zeta^{m+1/2})} \le \|Q\zeta^{m+1/2} \|_h $ and  a triangle inequality with Lemma~\ref{bhcompanion_lem}(v) shows
     $\norm{\Delta (Q\zeta^{m+1/2})} \le\Lambda \|\zeta^{m+1/2}\|_h$ for $\Lambda >0$.
     Therefore, combining this with the Cauchy--Schwarz  and  Young's inequality (as in \eqref{3rd-T2}) lead to
    \begin{align} 
     &-8(  \alpha\eta^{m,1/4}+\beta \varrho^{m,1/4}, \Delta (Q\zeta^{m+1/2}))\nonumber\\
     &
     \le   8\alpha\Lambda\norm{\eta^{m,1/4}} \norm{\zeta^{m+1/2}}_h +8\beta\Lambda\norm{\varrho^{m,1/4}} \norm{\zeta^{m+1/2}}_h\nonumber\\
     &\le 64\Lambda^2 {(d_0C_{\rm Coer})^{-1}}\big(\alpha^2\norm{\eta^{m,1/4}}^2+\beta^2\norm{\varrho^{m,1/4}}^2\big)+\frac{d_0}{2}C_{\rm Coer}\norm{\zeta^{m+1/2}}^2_h.\label{T2-2nd-ful}
      \end{align}
      A combination of \eqref{T2-ist-ful}-\eqref{T2-2nd-ful} in \eqref{ist-T2-ful} and bounds from \eqref{m-1}-\eqref{m-2} yields
\begin{align*}
&    8(\nabla ( \alpha\eta^{m,1/4}+\beta \varrho^{m,1/4}), \nabla  \zeta^{m+1/2})\\
 & \le C(h^{2+2\sigma}+h^{4\sigma}) \big(\norm{\theta}^2_{L^\infty(t_{m-1},t_{m+1};H^{1+\sigma}(\Omega))} +\norm{p}^2_{L^\infty(t_{m-1},t_{m+1};H^{1+\sigma}(\Omega))}  \big)\\
 &\qquad+{d_0}C_{\rm{Coer}}\norm{\zeta^{m+1/2}}^2_h.
\end{align*}
Analogous arguments lead to
\begin{align*}
   &      -8 (\nabla (\alpha \eta^{1,1/4}+\beta \varrho^{1,1/4}), \nabla  \zeta^{1/2})\\
   &
          \le C(h^{2+2\sigma}+h^{4\sigma}) \big(\norm{\theta}^2_{L^\infty(0,t_2;H^{1+\sigma}(\Omega))} +\norm{p}^2_{L^\infty(0,t_2;H^{1+\sigma}(\Omega))}  \big)
          +{d_0}C_{\rm{Coer}}\norm{\zeta^{1/2}}^2_h,\\
&-\sum_{n=1}^{m-1}\int_{t_{n-1}}^{t_{n+2}}\Big[(\nabla(\alpha \eta_t(t)+\beta \varrho_t(t)), \nabla \zeta^{n+1/2})\Big] \dt\\
&\quad-\sum_{n=1}^{m-1}\int_{t_{n}}^{t_{n+1}}\Big[(\nabla(\alpha \eta_t(t)+\beta \varrho_t(t)), \nabla \zeta^{n+1/2})\Big] \dt\\
&\le  C(h^{2+2\sigma}+h^{4\sigma})\big(\norm{\theta_t}^2_{L^2(0,T;H^{1+\sigma}(\Omega))}+\norm{p_t}^2_{L^2(0,T;H^{1+\sigma}(\Omega))}\big) \\
&\quad+\frac{d_0}{T}C_{\rm{Coer}}\Delta t \sum_{n=1}^{m-1}\norm{\zeta^{n+1/2}}^2_h.
\end{align*}
\rev{Putting together the previous bounds back into} \eqref{eq:t1} establishes
\begin{align}
    T_1&\le C\Big(h^4+(h^{2+2\sigma}+h^{4\sigma})\big(\norm{\theta}^2_{L^\infty(t_{0},t_{2};H^{1+\sigma}(\Omega))} +\norm{p}^2_{L^\infty(t_{0},t_{2};H^{1+\sigma}(\Omega))}\nonumber
\\
&\qquad+\norm{\theta}^2_{L^\infty(t_{m-1},t_{m+1};H^{1+\sigma}(\Omega))} +\norm{p}^2_{L^\infty(t_{m-1},t_{m+1};H^{1+\sigma}(\Omega))}\nonumber \\
&\qquad +\norm{\theta_t}^2_{L^2(0,T;H^{1+\sigma}(\Omega))}+\norm{p_t}^2_{L^2(0,T;H^{1+\sigma}(\Omega))}\big)\Big)\nonumber\\
&\qquad+2{d_0C_{\rm{Coer}}}\norm{\zeta^{1/2}}^2_h+2{d_0C_{\rm{Coer}}}\norm{\zeta^{m+1/2}}^2_h \nonumber\\
&\qquad+ \frac{2d_0 }{T}C_{\rm{Coer}}\Delta t\sum_{n=1}^{m-1}\norm{\zeta^{n+1/2}}^2_h,\label{Tm1-Tm8}
\end{align}
where the generic constant $C$ depends on $C_1, C_F, C_F', d_0^{-1}$, and $C_{\rm Coer}^{-1}$.

\noindent\textit{{Step 3} (Bound for $T_2$).}  Utilize the arguments similar to \eqref{p2-stab-all} (with  $\Theta^n$,   $U^n$, $P^n$ replaced by $\Psi^n$, $\zeta^n$, $\xi^n$, respectively) to obtain
\begin{align*}
    T_2&:= \Delta t \sum_{n=1}^m\Big[4(\alpha \nabla \Psi^{n,1/4}+\beta \nabla \xi^{n,1/4}, \nabla( \bar{\partial}_t\zeta^{n+1/2}+\bar{\partial}_t\zeta^{n-1/2}))\\
    &\qquad\qquad\qquad- 2 ( \alpha \nabla \Psi^{n+1/2}+\beta \nabla \xi^{n+1/2},\nabla \bar{\partial}_t\zeta^{n+1/2})\Big]\\
    &\;\;= 4\Delta t \sum_{n=1}^m (\alpha \nabla \Psi^{n,1/4}+\beta \nabla \xi^{n,1/4}, \nabla \bar{\partial}_t \zeta^{n-1/2})\\
    &\qquad+2\Delta t \sum_{n=1}^m (\alpha \nabla \Psi^{n-1/2}+\beta \nabla \xi^{n-1/2},\nabla \bar{\partial}_t\zeta^{n+1/2}).
\end{align*}
{Follow the approach used in Steps 2-3 of Theorem~\ref{p2;stability-thm} (more precisely see the bounds \eqref{p2;nabla-pm+1-stab}-\eqref{-T2})  to show} 
\begin{align}
T_2&\le 
  {c_1} \Delta t \norm{\nabla \Psi^{1/2}}^2+{\kappa} \Delta t\norm{\nabla \xi^{1/2}}^2
  +\Delta t \sum_{n=1}^{m}\Big[{c_1} \norm{\nabla \Psi^{n+1/2}}^2+\kappa\norm{\nabla \xi^{n+1/2}}^2\Big]\nonumber\\
 &\qquad+\frac{\Delta t}{2} \sum_{n=1}^{m-1}\Big[{c_1} \norm{\nabla \Psi^{n+1/2}}^2+\kappa\norm{\nabla \xi^{n+1/2}}^2\Big]\nonumber\\
 &\qquad
 + \frac{(\Delta t)^2}{a_0} \sum_{n=1}^{m-1}\Big[ \alpha^2\norm{\nabla \Psi^{n-1/2}}^2 +\beta^2\norm{\nabla \xi^{n-1/2}}^2\Big]\nonumber\\
 &\qquad+ 2{a_0} \norm{\nabla \bar{\partial}_t \zeta ^{m+1/2}}^2
+4\Delta t \bigg(\frac{\alpha^2}{c_1}+\frac{\beta^2}{\kappa}\bigg)\sum_{n=1}^{m}\norm{\nabla \bar{\partial}_t \zeta^{n-1/2}}^2\nonumber\\
 &\qquad+2\Delta t\bigg(\frac{\alpha^2}{c_1}+\frac{\beta^2}{\kappa}\bigg)\sum_{n=1}^{m-1}\norm{\nabla \bar{\partial}_t \zeta^{n+1/2}}^2.
  \end{align}
\textit{{Step 4} (Bound for $T_3-T_5$).}
A repeated application of \eqref{an-bn}  with $\epsilon= 8T$ yields
\begin{align*}
    \nonumber
    T_3&:=-4\Delta t \sum_{n=1}^m  \Big( (\bar{\partial}_t^2 \rho^n - R^n ,\bar{\partial}_t (\zeta^{n+1/2} +  \zeta^{n-1/2}))  \nonumber\\
    &
    \qquad+ a_0 (\nabla( \bar{\partial}_t^2 \rho^n - R^n),\nabla (\bar{\partial}_t (\zeta^{n+1/2} +  \zeta^{n-1/2})))\Big] \nonumber\\
   &\le 16{T}\Delta t \sum_{n=1}^m\Big[\norm{\bar{\partial}_t^2 \rho^n}^2+a_0\norm{\nabla  \bar{\partial}_t^2 \rho^n}^2+\norm{R^n}^2
    +a_0\norm{\nabla R^n}^2\Big]\nonumber\\
    &
    \qquad
    + 2\frac{\Delta t}{T}\sum_{n=1}^m \Big[\norm{\bar{\partial}_t\zeta^{n-1/2}}^2+a_0\norm{\nabla\bar{\partial}_t\zeta^{n-1/2}}^2\Big]\nonumber\\
    &
    \qquad+\frac{\Delta t}{T}\Big[\norm{\bar{\partial}_t\zeta^{m+1/2}}^2+a_0 \norm{\nabla\bar{\partial}_t\zeta^{m+1/2}}^2\Big]\nonumber
    \\
    & \le C\left(h^{4\sigma} \norm{u_{tt}}^2_{L^2(0,T;H^{2+\sigma}(\Omega))} +
    (\Delta t)^4 \norm{\nabla u_{tttt}}^2_{L^2(0,T;L^2(\Omega))}\right) \nonumber\\
    &\qquad+ 2\frac{\Delta t}{T}\sum_{n=1}^m \Big[\norm{\bar{\partial}_t\zeta^{n-1/2}}^2+a_0\norm{\nabla\bar{\partial}_t\zeta^{n-1/2}}^2\Big]\nonumber\\
    &
    \qquad+\Big[\norm{\bar{\partial}_t\zeta^{m+1/2}}^2+a_0 \norm{\nabla\bar{\partial}_t\zeta^{m+1/2}}^2\Big]
\end{align*}
with \eqref{rho-n},\eqref{R-n} and $\Delta t/T \le 1$ applied for the last term (in the last line).
  Another repeated application of \eqref{an-bn}  with $\epsilon = 8T(a_1-|\gamma|/\gamma_0)^{-1} $ and $\Psi^{n+1/2}=\frac{1}{2}(\Psi^{n}+\Psi^{n+1})$ leads to a bound for four terms in $T_4$ as 
    \begin{align*}
         & 2 \Delta t \sum_{n=1}^m\big[-a_1(\bar{\partial}_t\eta^{n+1/2}, \Psi^{n+1/2})
    +\gamma (\bar{\partial}_t{\varrho^{n+1/2}}, \Psi^{n+1/2})\\
    &\qquad\qquad\qquad
    +a_1 (\tau^n, \Psi^{n+1/2})-\gamma (s^n, \Psi^{n+1/2})\big]\\
    & \le 4 \Delta t T(a_1-|\gamma|/\gamma_0)^{-1} \sum_{n=1}^m\big[a_1^2\norm{\bar{\partial}_t\eta^{n+1/2}}^2+\gamma^2\norm{\bar{\partial}_t{\varrho^{n+1/2}}}^2
    +a_1^2\norm{\tau^n}^2+\gamma^2\norm{s^n}^2 \big]\\
    &\qquad\quad+\frac{ \Delta t }{T}(a_1-|\gamma|/\gamma_0) \sum_{n=1}^m \norm{\Psi^{n}}^2+\frac{ \Delta t }{2T}(a_1-|\gamma|/\gamma_0)\norm{\Psi^{m+1}}^2.
    \end{align*}
On the other hand,  Cauchy--Schwarz and Young's inequalities (with $\epsilon = 1/4,c_1/8,c_1/8$) bound the remaining terms of $T_4$ as
\begin{align*}
    &2 \Delta t \sum_{n=1}^m \big[-b_1(\eta^{n+1/2},\Psi^{n+1/2}) -\alpha (\nabla \bar{\partial}_t  \rho^{n+1/2} ,\nabla \Psi^{n+1/2})+\alpha (\nabla r^n, \nabla \Psi^{n+1/2})\big]\\
    &\quad \le
    \frac{ \Delta t }{4}\sum_{n=1}^m \big[b_1 \norm{\Psi^{n+1/2}}^2+c_1 \norm{\nabla \Psi^{n+1/2}}^2\big]\\
    &\quad+4 \Delta t \sum_{n=1}^m \big[ b_1\norm{\eta^{n+1/2}}^2+2\alpha^2 c_1^{-1}\norm{\nabla \bar{\partial}_t \rho^{n+1/2}}^2
    +2\alpha^2 c_1^{-1}\norm{\nabla r^n}^2\big].
\end{align*}
\rev{Then, thanks to the} last two inequalities, the following bound holds 
    \begin{align*}
         T_4&:=2\Delta t \sum_{n=1}^m\big[-a_1(\bar{\partial}_t\eta^{n+1/2}, \Psi^{n+1/2})
    +\gamma (\bar{\partial}_t{\varrho^{n+1/2}}, \Psi^{n+1/2}) \nonumber\\
    &\qquad\quad\quad\qquad-b_1(\eta^{n+1/2},\Psi^{n+1/2}) 
   -\alpha (\nabla \bar{\partial}_t  \rho^{n+1/2} ,\nabla \Psi^{n+1/2}) +a_1 (\tau^n, \Psi^{n+1/2}) \nonumber\\
    &\qquad\quad\quad\qquad-\gamma (s^n, \Psi^{n+1/2})+\alpha (\nabla r^n, \nabla \Psi^{n+1/2})\big]\nonumber\\
   & \le  \Delta t \sum_{n=1}^m\Big[\frac{(a_1-|\gamma|/\gamma_0)}{T}\norm{\Psi^{n}}^2+\frac{b_1}{4} \norm{\Psi^{n+1/2}}^2+\frac{c_1}{4} \norm{\nabla \Psi^{n+1/2}}^2 \Big]\nonumber\\
   &\quad+\frac{\Delta t}{2T}(a_1-|\gamma|/\gamma_0)\norm{\Psi^{m+1}}^2 
  +C\Delta t\sum_{n=1}^m\Big[\norm{\bar{\partial}_t\eta^{n+1/2}}^2+\norm{\bar{\partial}_t{\varrho^{n+1/2}}}^2\nonumber\\
  &\qquad\quad\quad\qquad+\norm{\eta^{n+1/2}}^2+\norm{\nabla \bar{\partial}_t \rho^{n+1/2}}^2
    +\norm{\tau^n}^2+\norm{s^n}^2+\norm{\nabla r^n}^2\Big],
    \end{align*}
    with a generic constant that depends on the material parameters.
  Analogous steps are now employed to bound 
    $ T_5:=2 \Delta t \sum_{n=1}^m\big[ (-a_2\bar{\partial}_t \varrho^{n+1/2} +\gamma \bar{\partial}_t\eta ^{n+1/2},\xi^{n+1/2})
        -\beta (\nabla\bar{\partial}_t \rho^{n+1/2},\nabla \xi^{n+1/2})
    + (a_2s^n-\gamma\tau^n,\xi^{n+1/2})+\beta (\nabla r^n,\nabla \xi^{n+1/2})\big]$~as
    \begin{align*}
        T_5&\le  \Delta t \sum_{n=1}^m\Big[\frac{(a_2-|\gamma|\gamma_0)}{T}\norm{\xi^{n}}^2+\frac{\kappa}{4} \norm{\nabla \xi^{n+1/2}}^2 \Big]+\frac{\Delta t}{2T}(a_2-|\gamma|\gamma_0)\norm{\xi^{m+1}}^2\nonumber\\
   &\; \quad+ C \Delta t\sum_{n=1}^m\Big[\norm{\bar{\partial}_t \varrho^{n+1/2}}^2+\norm{\bar{\partial}_t \eta^{n+1/2}}^2+\norm{\nabla\bar{\partial}_t \rho^{n+1/2}}^2\\
   &\qquad\qquad\qquad\qquad+\norm{\tau^n}^2+\norm{s^n}^2+\norm{\nabla r^n}^2\Big].
    \end{align*}
 \noindent
Next we put together  the bounds for $T_3-T_5$, use \eqref{m-3} and \eqref{rnsn} for controlling the terms in the last lines of $T_4$ and $T_5$, recall  the definitions  \eqref{t-norm} and \eqref{lm} to arrive at
 \begin{align}
     &T_3+T_4+T_5 \nonumber\\
     &\le C\big( L_{(u,\theta,p,T)}h^{4\sigma}+ M_{(u,\theta,p,T)}(\Delta t)^4\big)\nonumber\\
     &\qquad+\norm{\bar{\partial}_t\zeta^{m+1/2}}^2+a_0\norm{\nabla\bar{\partial}_t\zeta^{m+1/2}}^2+\frac{1}{4}\norm{(\Psi^m,\xi^m)}^2_{H}\nonumber\\
     &\qquad
     +\frac{1}{2}(a_1-|\gamma|/\gamma_0)\norm{\Psi^{m+1}}^2+\frac{1}{2}(a_2-|\gamma|\gamma_0)\norm{\xi^{m+1}}^2\nonumber\\
     &\qquad+2\frac{\Delta t}{T}\sum_{n=1}^m \Big[\norm{\bar{\partial}_t\zeta^{n-1/2}}^2+a_0\norm{\nabla\bar{\partial}_t\zeta^{n-1/2}}^2\Big]\nonumber
\\&\qquad+\frac{\Delta t}{T} \sum_{n=1}^m\Big[{(a_1-|\gamma|/\gamma_0)}\norm{\Psi^{n}}^2+{(a_2-|\gamma|\gamma_0)}\norm{\xi^{n}}^2\Big],
 \end{align}
with $\Delta t/T \le 1$ utilized in two terms above that involve $\|\Psi^{m+1}\|^2$ and $\|\xi^{m+1}\|^2$.

\medskip
\noindent\textit{{Step 5} (Bounds for $T_6$ and $T_7$).}
Elementary manipulations show
\begin{align}
T_6&:=2\gamma\sum_{n=1}^m \Big[ (\Psi^{n+1}, \xi^{n+1})-(\Psi^{n},\xi^{n} )\Big]\nonumber\\
&\;\le \frac{|\gamma|}{\gamma_0}\norm{\Psi^{m+1}}^2+ {|\gamma|\gamma_0}\norm{\xi^{m+1}}^2+{|\gamma|}\Big(\norm{\Psi^{1}}^2+\norm{\xi^{1}}^2\Big).\nonumber
\end{align}
This, $\Psi^0=0$ and $\xi^0=0$, and the definition $T_7:=4\norm{\bar{\partial}_t \zeta^{1/2}}^2+4a_0\norm{\nabla\bar{\partial}_t \zeta^{1/2}}^2+4d_0C_{\rm Cont} \norm{\zeta^{1/2}}^2_h+{a_1}\norm{\Psi^{1}}^2+{a_2}\norm{\xi^{1}}^2 $ lead to
\begin{align}
&T_6+T_7 \nonumber\\
&\le 4\norm{\bar{\partial}_t \zeta^{1/2}}^2+4a_0\norm{\nabla\bar{\partial}_t \zeta^{1/2}}^2+4d_0C_{\rm Cont} \norm{\zeta^{1/2}}^2_h\nonumber\\
&\qquad+2\Big({a_1}+{|\gamma|}\Big)\norm{\Psi^{1/2}}^2+2\Big({a_2}+{|\gamma|}\Big)\norm{\xi^{1/2}}^2+\frac{|\gamma|}{\gamma_0}\norm{\Psi^{m+1}}^2+ {|\gamma|\gamma_0}\norm{\xi^{m+1}}^2\nonumber\\
& \le C\big(h^{4 \sigma} +L_{(u,\theta,p,t_1)}h^{4\sigma}+M_{(u,\theta,p,t_1)}(\Delta t)^4\big)\nonumber\\
&\qquad+\frac{|\gamma|}{\gamma_0}\norm{\Psi^{m+1}}^2+ {|\gamma|\gamma_0}\norm{\xi^{m+1}}^2,
    \label{Tm1623}
\end{align}
where elementary manipulations and Lemma~\ref{initial-error-lemm} were used in the last step. 

\noindent\textit{{Step 6} (Consolidation).}
First note that, by definition  \eqref{t-norm} and some basic manipulations, we can assert that 
\begin{align*}
&\frac{7}{4}\norm{(\Psi^{m},\xi^{m})}_H^2-\Delta t \sum_{n=1}^{m}\Big[{c_1} \norm{\nabla \Psi^{n+1/2}}^2+\kappa\norm{\nabla \xi^{n+1/2}}^2\Big]\\
&\quad-\frac{\Delta t}{2} \sum_{n=1}^{m-1}\Big[{c_1} \norm{\nabla \Psi^{n+1/2}}^2+\kappa\norm{\nabla \xi^{n+1/2}}^2\Big]\\
&= \frac{\Delta t}{4} \sum_{n=1}^m\Big[7 b_1\norm{\Psi^{n+1/2}}^2 +c_1 \norm{\nabla\Psi^{n+1/2}}^2+\kappa \norm{\nabla \xi^{n+1/2}}^2 \Big]\nonumber\\
&\quad+ \frac{\Delta t}{2} \Big[c_1 \norm{\nabla\Psi^{m+1/2}}^2+\kappa \norm{\nabla \xi^{m+1/2}}^2 \Big]\\
&\ge \frac{\Delta t}{2} \Big[c_1 \norm{\nabla\Psi^{m+1/2}}^2+\kappa \norm{\nabla \xi^{m+1/2}}^2 \Big]+\frac{1}{4}\norm{(\Psi^{m},\xi^{m})}_H^2.
\end{align*}
This, a combination of \eqref{Tm1-Tm8}-\eqref{Tm1623} in \eqref{p2;error-all-terms}  with 
$2d_0 C_{\rm Coer}\|\zeta^{1/2} \|_h^2 + {c_1} \Delta t \norm{\nabla \Psi^{1/2}}^2+{\kappa} \Delta t\norm{\nabla \xi^{1/2}}^2 \lesssim h^{4\sigma} +L_{(u,\theta,p,t_1)}h^{4\sigma}+M_{(u,\theta,p,t_1)}(\Delta t)^4$ from Lemma~\ref{initial-error-lemm} to bound the terms that involve the initial bounds in $T_1$ and $T_2$, and some elementary manipulations of the constants yield 
\begin{align*}
   &3\norm{\bar{\partial}_t \zeta^{m+1/2}}^2+a_0\norm{\nabla \bar{\partial}_t \zeta^{m+1/2}}^2+2{d_0C_{\rm{Coer}}}\norm{\zeta^{m+1/2}}^2_h\nonumber\\
    & \quad +\frac{1}{2}({a_1-|\gamma|/\gamma_0}) \norm{\Psi^{m+1}}^2
  +\frac{1}{2}({a_2-|\gamma|\gamma_0})\norm{\xi^{m+1}}^2 \nonumber\\
  &\quad + \frac{\Delta t}{2} c_1 \norm{\nabla\Psi^{m+1/2}}^2+ \frac{\Delta t}{2}\kappa \norm{\nabla \xi^{m+1/2}}^2 +\frac{1}{4}\norm{(\Psi^{m},\xi^{m})}_H^2\\
  &\lesssim  h^{4\sigma}+ \big[L_{(u,\theta,p,t_1)}+L_{(u,\theta,p,T)}\big]h^{4\sigma}+\big[M_{(u,\theta,p,t_1)}+M_{(u,\theta,p,T)}\big](\Delta t)^4\nonumber\\
  \qquad
  &\quad+  \frac{\Delta t }{T}\nu \sum_{n=0}^{m-1}\Big[3\norm{\bar{\partial}_t\zeta^{n+1/2}}^2+a_0\norm{\nabla\bar{\partial}_t\zeta^{n-1/2}}^2\\
  &\quad+2d_0C_{\rm{Coer}}\norm {\zeta^{n+1/2}}^2_h +\frac{1}{2}(a_1-|\gamma|/\gamma_0)\norm{\Psi^{n+1}}^2\nonumber  \\
  &\quad  +\frac{1}{2}(a_2-|\gamma|\gamma_0)\norm{\xi^{n+1}}^2+\frac{\Delta t}{2} c_1\norm{\nabla \Psi^{n+1/2}}^2+\frac{\Delta t}{2}{\kappa} \Delta t \norm{\nabla \xi^{n+1/2}}^2\big)\Big].
   \end{align*}
The constant $\nu$ \rev{on} the right-hand side of the above expression is manipulated for an easy application of  Gronwall's Lemma \ref{P1 d-gronwall}. Now, we apply  Lemma~\ref{P1 d-gronwall} to arrive at 
\begin{align}
 &  \norm{\bar{\partial}_t \zeta^{m+1/2}}^2+a_0\norm{\nabla \bar{\partial}_t \zeta^{m+1/2}}^2+{d_0}\norm{\zeta^{m+1/2}}^2_h+({a_1-|\gamma|/\gamma_0}) \norm{\Psi^{m+1}}^2\nonumber\\
   &\quad +({a_2-|\gamma|\gamma_0})\norm{\xi^{m+1}}^2+
   \norm{(\Psi^{m},\xi^{m})}_H^2\nonumber \\
   &\quad + \frac{\Delta t}{2} \Big[c_1 \norm{\nabla\Psi^{m+1/2}}^2+\kappa \norm{\nabla \xi^{m+1/2}}^2 \Big] \nonumber\\
  &\lesssim  h^{4\sigma}+\big[L_{(u,\theta,p,t_1)}+L_{(u,\theta,p,T)}\big]h^{4\sigma}+\big[M_{(u,\theta,p,t_1)}+M_{(u,\theta,p,T)}\big](\Delta t)^4.\label{proj;estimate}
   \end{align}
   We ignore the non-negative term $\frac{\Delta t}{2} \left[c_1 \|\nabla\Psi^{m+1/2}\|^2 + \kappa \|\nabla \xi^{m+1/2}\|^2 \right]$ from the left-hand side and apply the definitions \eqref{p2;splitting1}-\eqref{p2;splitting3},   triangle inequality and the projections estimates from \eqref{P1 norm_ritz}-\eqref{P1 norm_ritz2},   eventually leading to the desired estimates. 
\end{proof}
\noindent
The $L^2$-estimates for $\theta$ and $p$ have already been derived in the above {theorem, while  for  $u$, we present the following result.}
\begin{cor}[{$L^2$ and $H^1$-estimates for deflection}]\label{corr}
   Suppose that $(u,\theta,p)$ and $(U^n,\Theta^n, P^n)$ solve \eqref{u-weak}-\eqref{p-weak} and \eqref{un}-\eqref{pn}, respectively. Then, under the assumptions of Theorem~\ref{error-estimates thm}, for $1\le m\le N-1$, the following error estimate holds
\begin{align*}
  \norm{ u^{m+1}- U^{m+1} }^2+a_0\norm{\nabla(u^{m+1}- U^{m+1})}^2 \lesssim  h^{4\sigma} +(\Delta t)^4 .
\end{align*}
\end{cor}
\begin{proof}
        Ignoring the last four non-negative terms on the left-hand side of \eqref{proj;estimate}, we can obtain
\[
    \norm{\bar{\partial}_t \zeta^{m+1/2}}^2+a_0\norm{\nabla \bar{\partial}_t \zeta^{m+1/2}}^2+\norm{\zeta^{m+1/2}}^2_h \lesssim h^{{4\sigma} } +\Delta t ^4,  \qquad\text{ for } 1\le m\le N-1.\] 
Note that $\zeta^{m+1}= \zeta^{m+1/2}+\frac{1}{2}\Delta t \bar{\partial}_t\zeta^{m+1/2}$ (resp. $\nabla \zeta^{m+1}=\nabla \zeta^{m+1/2}+\frac{1}{2}\Delta t \nabla \bar{\partial}_t\zeta^{m+1/2}$). 
Then, { by the} discrete Poincar\'e   inequality we have $ \norm{\zeta^{m+1/2} } \le \norm{\zeta^{m+1/2} }_h$ (resp.$\norm{\nabla \zeta^{m+1/2} } \le \norm{\zeta^{m+1/2} }_h$), and hence  
\begin{align*}&\norm{ \zeta^{m+1}}^2+a_0\norm{\nabla \zeta^{m+1}}\\
& \lesssim \norm{\bar{\partial}_t \zeta^
{m+1/2}}^2+ \norm{\zeta^{m+1/2} }^2_h+a_0\norm{\nabla \bar{\partial}^2_t \zeta^
{m+1/2}}+ a_0\norm{\zeta^{m+1/2} }_h\\ &\lesssim (1+a_0)( h^{4\sigma}+(\Delta t)^4).\end{align*}
Therefore, simply using the triangle inequality we can obtain 
\begin{align*}&\norm{ u^{m+1}- U^{m+1} }^2+a_0\norm{\nabla(u^{m+1}- U^{m+1})}^2 \\ &  \lesssim \| \rho^{m+1}  \|^2 +\| \zeta^{m+1}  \|^2+ a_0\|\nabla \rho^{m+1}  \|^2 +a_0\| \nabla \zeta^{m+1}  \|^2 . \end{align*}
\rev{The} last three inequalities  \rev{together with}  \eqref{P1 norm_ritz} \rev{applied to} $\| \rho^{m+1}  \|^2+ a_0\|\nabla \rho^{m+1}  \|^2$ lead to the desired result. 
\end{proof}

\section{Numerical results}\label{sec:numer}
In this section, we investigate the application of the Kirchhoff–Love plate model in Subsection~\ref{subsec-numerics-Verif-kirchoff} to capture \ac{ted} in copper and \ac{tpe}  in flat Berea sandstone.
Subsections~\ref{subsec-num-smooth}-\ref{subsec-num-lshape} provide numerical results that validate theoretical estimates and illustrate the effective performance of the proposed scheme across different values of the parameter $\gamma$. The penalty parameter $\sigma_{\rm IP}$ is chosen according to  Ref. \refcite{MR4796047}. All simulations were conducted with the finite element library FEniCS\cite{Alns2015}, and executed on a desktop machine equipped with an Intel$^{\text{\textregistered{}}}$ Core\texttrademark{} i5-7500 CPU (Kaby Lake architecture), featuring 4 cores and 4 threads, operating at a base frequency of \SI{3.4}{\giga\hertz}.

\subsection{Example 1: Verification of Kirchhoff's model:  2D vs 3D   TED and TPE plate models}\label{subsec-numerics-Verif-kirchoff}
In this subsection, we illustrate Kirchhoff's hypothesis by comparing the solution of the  three-dimensional (3D) model in \eqref{diffusion-3dmodel} (resp. \eqref{poro-3dmodel}) for \ac{ted} (resp. \ac{tpe}) with two-dimensional (2D) model in \eqref{eq:coupled}, while systematically varying the plate thickness $d$. 
The plots in Figures~\ref{fig:TED}-\ref{fig:TPE} demonstrate that as the plate becomes thinner, the solution curves (plotted against time) from the 3D model approximates those for the 2D model. 

Building upon  classical theory as described in Ref. \refcite{Nowacki} (see also  Eq. (9) in Ref. \refcite{Aouadi}), given a space-time dependent loading $(\hat{f}_1(t),\hat{f}_2(t),\hat{f}_3(t))=\boldsymbol{f} (t): \widehat{\Omega}  \to \mathbb{R}^3$ ($\widehat{\Omega} \subset \mathbb{R}^3$), prescribed heat source  $\hat{\phi}(t): \widehat{\Omega}  \to \mathbb{R}$, and total amount of mass source/sink $\hat{g}(t):\widehat{\Omega}  \to \mathbb{R}$, the 3D \ac{ted}  model seeks the displacement vector $\boldsymbol {u}=(\hat{u}_1(t),\hat{u}_2(t),\hat{u}_3(t))$, the small temperature increment $\hat{\theta}=T_{\text{abs}}-T_0$ (with  $T_{\text{abs}},T_0$ as the absolute and reference temperature, respectively), and the chemical potential $\hat{p}$ such that
\begin{subequations}\label{diffusion-3dmodel}
\begin{align}
   \rho\boldsymbol{u}_{tt} -\boldsymbol{\nabla} \cdot \boldsymbol{\sigma}&=\rho\boldsymbol{f}  \qquad \text{in } \widehat{\Omega} \times [0,T],\label{3D1}\\ 
     \big(\frac{\rho c_E}{T_0}+\frac{\varpi^2}{\varrho}\big) \hat{\theta}_t+ \frac{\varpi}{\varrho}\hat{p}_t + \nabla \cdot \boldsymbol{q} +\gamma_1 \nabla \cdot\boldsymbol{u}_t&=\hat{\phi} \qquad\; \text{ in } \widehat{\Omega} \times [0,T],\label{3D2}\\
   \frac{1}{\varrho}\hat{p}_t+\frac{\varpi}{\varrho} \hat{\theta}_t+ \nabla \cdot \boldsymbol{p}+ \gamma_2 \nabla \cdot\boldsymbol{u}_t&=\hat{g} \qquad \;\; \text{ in } \widehat{\Omega} \times [0,T]. \label{3D3}
\end{align}
\end{subequations}
The $\boldsymbol{\sigma}=2\mu\nabla_{\mathrm{sym}}\boldsymbol{u} + [\lambda_0\nabla\cdot\boldsymbol{u}-\gamma_1\hat{\theta}-\gamma_2\hat{p}]\mathbf{I}$ above is the total Cauchy stress tensor, $\boldsymbol{q}=-{k_1}\nabla\hat{\theta}$ is the heat flux (Fourier’s law) and $\boldsymbol{p}=-k_2\nabla\hat{p}$ the diffusive flux (Fick’s law). The constants involved in the definition of $\boldsymbol{\sigma}$  are given by
$$\lambda_0= \lambda - \frac{(3\lambda+2\mu)^2\alpha^2_c}{\varrho},  \;\; \gamma_1 =(3\lambda+2\mu)\big(\alpha_t  +\frac{\varpi}{\varrho}\alpha_c\big),\;\; \gamma_2=\frac{(3\lambda+2\mu)\alpha_c}{\varrho},$$
with the basic parameters
from  Table~\ref{fig:subfig1}. The surfaces at $z=-d/2,d/2$
  are subject to traction-free and zero-flux boundary conditions, while the remaining boundaries are governed by homogeneous Dirichlet conditions. 
  
A dimensional reduction analysis in  eqns. (9)-(46) of Ref. \refcite{Aouadi} derives the  2D model  \eqref{eq:coupled} from the 3D model \eqref{diffusion-3dmodel} which seeks
   transverse displacement, first moments of temperature and chemical potential
   \begin{align}
   u=\frac{1}{ d }\int_{ - d /2}^{  d /2} \hat{u}_3 \dz,\; \theta= \int_{ - d /2}^{  d /2} z\hat{\theta} \dz,\; \text{and} \;\; p=\int_{-  d /2}^{  d /2} z\hat{p} \dz.\label{momonts-ted}
   \end{align}
  The transformation of the model coefficients in this process is given in Table~\ref{table-coeefients} and that of moments of the right-hand side functions (forces and sources)  by 
   \begin{align}
       &f=\frac{1}{  d }\int_{ - d /2}^{  d /2} \hat{f}_3\dz ,\;\; \phi= \frac{12}{\rho  d ^4}\int_{ - d /2}^{  d /2} z\hat{\phi}\dz,\; \text{and} \;\; g=\frac{12}{\rho  d ^4}\int_{-  d /2}^{  d /2} z\hat{g}\dz.      \label{load4}
   \end{align}
It is very important to note that the constant $\lambda_0$ is assumed to satisfy  $\lambda_0+\mu >0$\cite{paul2023novel} and this condition makes all the coefficients except $\gamma$ in the 2D model \eqref{eq:coupled} positive, see Table~\ref{table-coeefients}.
  
   Another  example of diffusion in porous media is the phenomenon of \ac{tpe}.  Consider now that  the domain $\widehat{\Omega} \subset \mathbb{R}^3$ is fully saturated with a viscous fluid.  The  flow occurs also in the $xy$ plane and the poroelastic material is subject to thermal energy effects. Given the mechanical load $(\hat{f}_1(t),\hat{f}_2(t),\hat{f}_3(t))=\boldsymbol{f} (t): \widehat{\Omega}  \to \mathbb{R}^3$ ($\widehat{\Omega} \subset \mathbb{R}^3$), prescribed heat source  $\hat{\phi}(t): \widehat{\Omega}  \to \mathbb{R}$, and  fluid mass source $\hat{g}^*(t):\widehat{\Omega}  \to \mathbb{R}$,
the three-dimensional \ac{tpe} equations\cite{MR4146798,son-young,Brun-Mats}  seeks displacement  $\boldsymbol{u}$, the small temperature increment $\hat{\theta}$, and  pore pressure $\hat{p}^*$ such that
\begin{subequations}\label{poro-3dmodel}
\begin{align}
   \rho\boldsymbol{u}_{tt} -\boldsymbol{\nabla} \cdot \boldsymbol{\sigma}&=\rho \boldsymbol{f}  \quad\;\; \text{ in } \widehat{\Omega} \times [0,T],\label{3d-poro1}\\ 
    \frac{\rho c_E}{T_0} \hat{\theta}_t- 3\gamma^*\hat{p}_t + \nabla \cdot \boldsymbol{q} +\gamma_1^* \nabla \cdot\boldsymbol{u}_t&=\hat{\phi} \qquad\; \text{ in } \widehat{\Omega} \times [0,T],\\
   \frac{1}{\varrho^*}\hat{p}_t^*-3\gamma^* \hat{\theta}_t+ \nabla \cdot \boldsymbol{p}^*+ \gamma_2^* \nabla \cdot\boldsymbol{u}_t&=\hat{g}^* \qquad  \text{ in } \widehat{\Omega} \times [0,T],\label{3d-poro3}
\end{align}
\end{subequations}
The surfaces at $z=-d/2,d/2$
  are subject to traction-free and zero-flux boundary conditions, while the remaining boundaries are governed by homogeneous Dirichlet conditions. Here, $\boldsymbol{\sigma}=2\mu\nabla_{\mathrm{sym}}\boldsymbol{u} + [\lambda\nabla\cdot\boldsymbol{u}-\gamma_1^*\hat{\theta}-\gamma_2^*\hat{p}^*]\mathbf{I}$, $\boldsymbol{q}=-{k_1}\nabla\hat{\theta}$ ( Fourier’s law) and $\boldsymbol{p}=-k_2^*\nabla\hat{p}$ ( Darcy’s law) represent the total stress tensor, heat and fluid flux respectively. Also $\gamma_1^*= \alpha_t(3\lambda+2\mu)$, $\gamma_2^*=\beta^*$, and the other constants are defined in Table~\ref{fig:subfig1}.
  
The three-dimensional \ac{tpe} model \eqref{poro-3dmodel} exhibits structural similarities to the three-dimensional \ac{ted} model \eqref{diffusion-3dmodel}. By employing a dimensional reduction approach analogous to that in Ref. \refcite{Aouadi}, we derive a two-dimensional \ac{tpe} model \eqref{eq:coupled} which seeks
   transverse displacement, first moments of temperature and pore pressure 
\begin{align}
   u=\frac{1}{ d }\int_{ - d /2}^{  d /2} \hat{u}_3 \dz,\; \theta= \int_{ - d /2}^{  d /2} z\hat{\theta} \dz,\; \text{and} \;\; p=\int_{-  d /2}^{  d /2} z\hat{p}^* \dz.\label{momonts-poro}
   \end{align}
The moments of the right-hand side functions (forces and sources) of this 2D plate model are given as 
   \begin{align}
       &f=\frac{1}{  d }\int_{ - d /2}^{  d /2} \hat{f}_3\dz ,\;\; \phi= \frac{12}{\rho  d ^4}\int_{ - d /2}^{  d /2} z\hat{\phi}\dz,\; \text{and} \;\; g=\frac{12}{\rho  d ^4}\int_{-  d /2}^{  d /2} z\hat{g}^*\dz,      \label{load-tpe}
   \end{align}
   and the parametrization of the coefficients  given in Table~\ref{table-coeefients}.

\begin{table}[t]
\setlength{\tabcolsep}{2.pt}
\renewcommand{\arraystretch}{1.7}
\centering
\begin{tabular}{|c|c|c|}
\hline
\textit{Coeff.} & \textit{2D-TED} & \textit{2D-TPE} \\
\hline
$a_0$& $\frac{  d ^2}{12}$ & $\frac{  d ^2}{12}$ \\
\hline
$d_0$ & $ \frac{4\mu  d ^2(\lambda_0+\mu)}{12\rho (\lambda_0+2\mu)}$& $\frac{4\mu  d ^2(\lambda+\mu)}{12\rho (\lambda+2\mu)}$ \\
\hline
$\alpha$ & $\frac{2\mu\gamma_1}{\rho   d (\lambda_0+2\mu)}$ &$\frac{2\mu\gamma_1}{\rho   d (\lambda+2\mu)}$ \\
\hline
$\beta$ & $\frac{2\mu\gamma_2}{\rho   d (\lambda_0+2\mu)}$ &$\frac{2\mu\gamma_2}{\rho   d (\lambda+2\mu)}$ \\
\hline
$a_1$ & $\frac{12}{ \rho d ^4} \bigg(\frac{\rho c_E}{T_0}+\frac{\varpi^2}{\varrho}+\frac{\gamma_1^2}{\lambda_0+2\mu}\bigg)$ &$\frac{12}{ \rho d ^4} \bigg(\frac{\rho c_E}{T_0}+\frac{\gamma_1^2}{\lambda+2\mu}\bigg)$ \\
\hline
$\gamma$&$-\frac{12}{\rho  d ^4}\big( \frac{\varpi}{\varrho}+\frac{\gamma_1\gamma_2}{\lambda_0+2\mu} \big)$&$\frac{12}{\rho  d ^4}\big(  3\gamma^* -\frac{\gamma_1\gamma_2}{\lambda+2\mu} \big)$ \\
\hline
$b_1$&$\frac{12k_1 }{\rho  d ^3}$ &$\frac{12k_1 }{\rho  d ^3}$ \\
\hline
$c_1$ & $\frac{12k_1}{\rho  d ^4}$&$\frac{12k_1}{\rho  d ^4}$ \\
\hline
$a_2$&$\frac{12}{   \rho d ^4}  \bigg(\frac{1}{\varrho}+\frac{\gamma_2^2}{\lambda_0+2\mu}\bigg)$&$\frac{12}{   \rho d ^4}  \bigg(\frac{1}{\varrho^{*}}+\frac{\gamma_2^2}{\lambda+2\mu}\bigg)$\\
\hline
$\kappa$&$\frac{12k_2}{\rho  d ^4}$& $ \frac{12k_2^{*}}{\rho  d ^4}$\\
\hline
\end{tabular}
\caption{Coefficients in the 2D model \eqref{p2;model11}-\eqref{p2;model33} for thermoelastic diffusion and thermo-poroelastic cases.}
\label{table-coeefients}
\end{table}

\begin{table}[h!]
\setlength{\tabcolsep}{3pt}
\centering
\begin{minipage}[t]{0.48\textwidth}
\caption{3D-TED model.}
\begin{tabular}{lll}
\toprule
\textbf{Coeff.} & \textbf{Value} & \textbf{SI Unit} \\
\midrule
$\quad \lambda$   & \num{7.76e10}   & \si{\kilogram\metre^{-1}\second^{-2}} \\
$\quad \mu$       & \num{3.36e10}  & \si{\kilogram\metre^{-1}\second^{-2}} \\
$\quad \varrho$   & \num{9.0e5}    & \si{\metre^5 \kilogram^{-1} \second^{-2}} \\
$\quad \alpha_t$  & \num{1.78e-5}  & \si{\per\kelvin} \\
$\quad \alpha_c$  & \num{1.98e-4}  & \si{\metre^4 \kilogram^{-1}} \\
$\quad \varpi$    & \num{1.2e4}    & \si{\metre^2\second^{-2}\kelvin^{-1}} \\
$\quad \rho$      & \num{8954}     & \si{\kilogram\metre^{-3}} \\
$\quad c_E$       & \num{383.1}    & \si{\joule\per\kilogram\per\kelvin} \\
$\quad T_0$       & \num{293}      & \si{\kelvin} \\
$\quad k_1$       & \num{386}      & \si{\watt\per\metre\per\kelvin} \\
$\quad k_2$       & \num{8.5e-9}   & \si{\kilogram\second\metre^{-3}} \\
\bottomrule
\bottomrule
\end{tabular}\label{Table-TED-constants}
\end{minipage}
\hfill
\begin{minipage}[t]{0.48\textwidth}
\centering
\caption{3D-TPE model.}
\begin{tabular}{lll}
\toprule
\textbf{Coeff.} & \textbf{Value} & \textbf{SI Unit} \\
\midrule
$\quad \lambda$   & \num{10.22e9} \cite{elastic-sandstone}  & \si{\kilogram\per\metre\per\second\squared} \\
$ \quad \mu$      & \num{4.09e9} \cite{elastic-sandstone}   & \si{\kilogram\per\metre\per\second\squared} \\
$\quad \alpha_t$  & \num{3e-5}  \cite{hotfluid}& \si{\per\kelvin} \\
$\quad \varrho^*$  & \num{12e9} \cite{biot-mod} & \si{\kilogram\per\metre\per\second\squared} \\
$\quad \beta^*$    & \num{0.79} \cite{biot-mod}   & --\\
$\quad \rho$      & \num{2280} \cite{elastic-sandstone}    & \si{\kilogram\metre^{-3}} \\
$\quad c_E$       & \num{800} \cite{hotfluid} & \si{\joule\per\kilogram\per\kelvin} \\
$\quad T_0$       & \num{293}  \cite{hotfluid}    & \si{\kelvin} \\
\quad $\gamma^*$  & \num{5e-5}    & \si{\per\kelvin} \\
$\quad k_1$       & \num{1e-6}  \cite{hotfluid}    & \si{\watt\per\metre\per\kelvin} \\
$\quad k_2$       & \num{1.9e-13} \cite{hotfluid}  & \si{\metre^2} \\
\bottomrule
\bottomrule
\end{tabular}\label{Table-TPE-constants}
\end{minipage}
\caption*{Example 1. 3D model coefficients for copper\cite{sherief2005half} (left) and   Berea sandstone\cite{elastic-sandstone,hotfluid,biot-mod} (right) plate.}
\end{table}

\medskip 
\noindent \textbf{Thermoelastic diffusion plate model verification:} 
Our objective is to illustrate that the 2D \ac{ted} model~\eqref{eq:coupled}, effectively approximates the 3D \ac{ted} model described by~\eqref{diffusion-3dmodel} in the sense that if $(U^n, \Theta^n , P^n)$ is the approximation of the solution $( {u} ,{\theta}, {p})$ of the 2D model \eqref{eq:coupled} at time $t=t_n$ computed with the discrete formulation \eqref{fully-discret-scheme} and $(\hat{\mathbf{U}}^n, \hat{\Theta}^n , \hat{P}^n)$ is the discrete solution of \eqref{diffusion-3dmodel} at $t=t_n$, then $(U^n,\Theta^n , P^n)$ approximates the triplet $(\int_{-d/2}^{d/2}\hat{U}_3^n \dz,\int_{-d/2}^{d/2}z\hat{\Theta}^n  \dz,\int_{-d/2}^{d/2}z\hat{P}^n  \dz)$ with $\hat{\mathbf{U}}^n=(\hat{U}_1^n,\hat{U}_2^n,\hat{U}_3^n)$, as motivated by \eqref{momonts-ted}. 

\begin{figure}[ht]
\begin{center}
\includegraphics[width=\linewidth, height=0.4\linewidth]{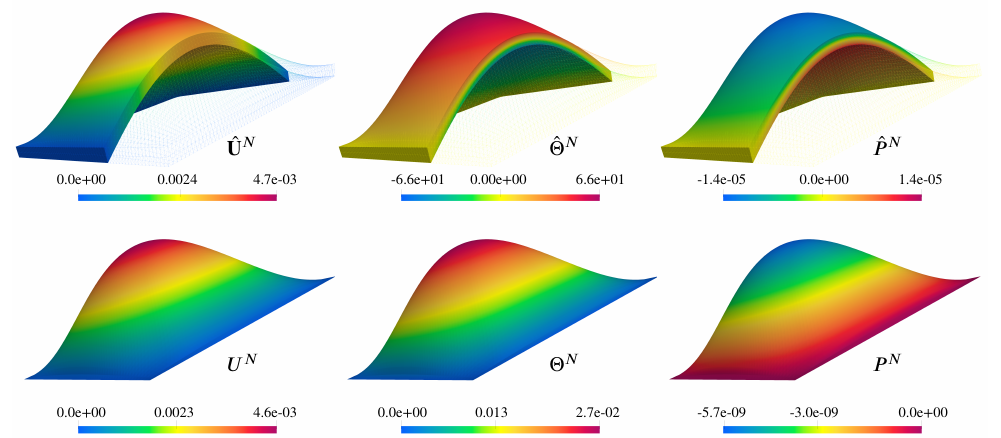}
\vspace{-0.5cm}
     \phantomcaption
  \caption{Example 1. 3D (upper) and 2D (lower) displacement , temperature, and chemical potential at final time $T$ for TED model.}  \label{fig:TEDsol}
 \end{center}
\end{figure}

\begin{figure}[ht]
\begin{center}
\includegraphics[width=\linewidth, height=0.6\linewidth]{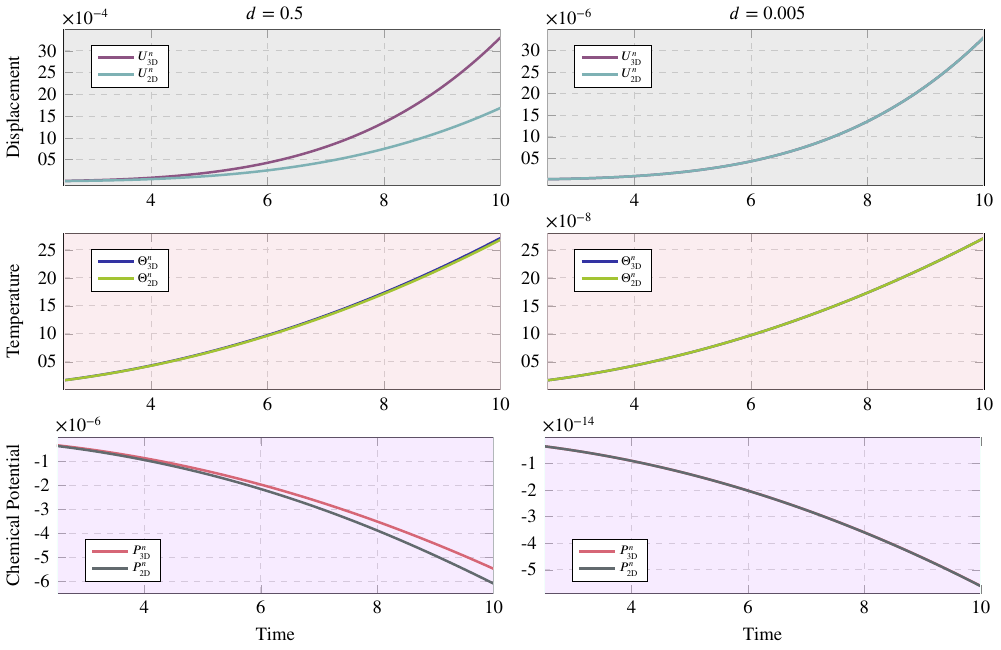}
\vspace{-0.5cm}
     \phantomcaption
   \label{fig:TED_p2}
  \caption{Example 1. 2D $(U^n_{2\rm D},\Theta^n_{2\rm D},P^n_{2\rm D})$ and 3D  $({U}^n_{3\rm D},{\Theta}^n_{3\rm D},{P}^n_{3\rm D})$ solution vs time $t_n$ with plate thickness $d=0.5$ (left) and $d=0.005$ (right) for the \ac{ted} model.}  \label{fig:TED}
 \end{center}
\end{figure}

\begin{figure}[ht]
\begin{center}
\includegraphics[width=\linewidth, height=0.6\linewidth]{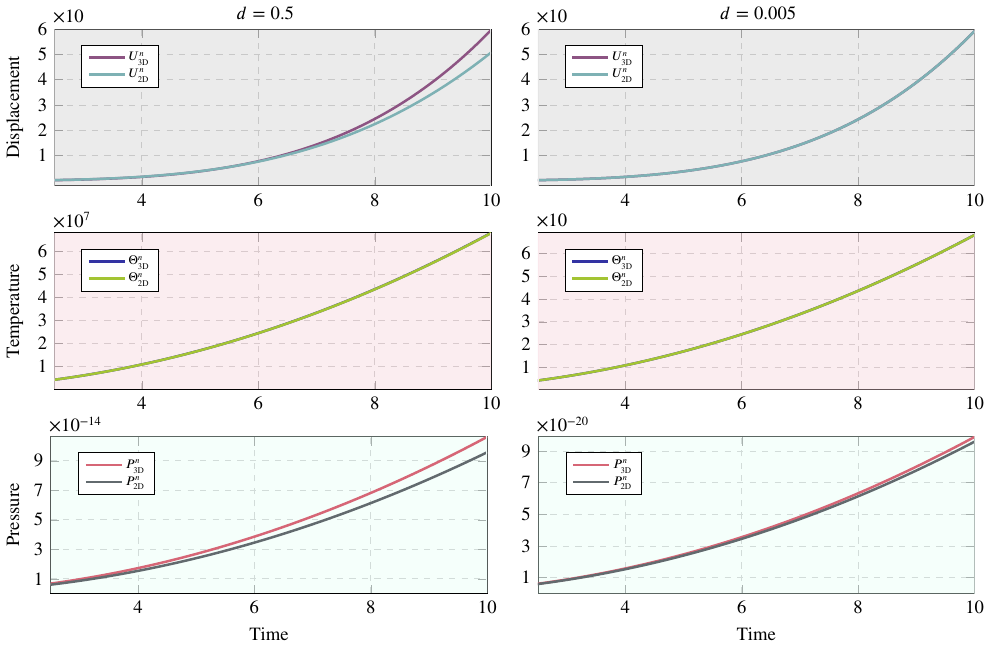}
\vspace{-0.5cm}
     \phantomcaption
  \caption{Example 1. 2D $(U^n_{2\rm D},\Theta^n_{2\rm D},P^n_{2\rm D})$ and 3D  $({U}^n_{3\rm D},{\Theta}^n_{3\rm D},{P}^n_{3\rm D})$ solution vs time $t_n$ with plate thickness $d=0.5$ (left) and $d=0.005$ (right panels) for the \ac{tpe} model.}  \label{fig:TPE}
 \end{center}
\end{figure}

To achieve this, we solve the 3D system \eqref{diffusion-3dmodel} using continuous \ac{fe} spaces: \( (\rev{\mathcal{P}_1}(\mathcal{T}))^3  \) for displacement \( \boldsymbol{u} \), and \( \rev{\mathcal{P}_1}(\mathcal{T}) \) for temperature \( \hat{\theta} \) and pressure \( \hat{p} \), with $\hat{\Omega}=[0,1] \times [0,1] \times [-d/2, d/2]$. The temporal discretization is handled by the Newmark scheme for \eqref{3D1} and by Crank--Nicolson scheme  for \eqref{3D2}-\eqref{3D3}. Homogeneous Dirichlet boundary conditions are set on all the sides except the surfaces $z=-d/2,d/2$ where the plate is assumed  traction free and  subject to zero heat/diffusion flux (in line with the theoretical discussion in Ref. \refcite{Aouadi}). In the 3D setting, the load, heat, and mass sources are defined as
\begin{align}
\mathbf{f} = \left(0, 0, t^2\sin(\pi x) \sin(\pi y)\right),\;\hat{\phi} = t xy(x - 1)(y - 1), \text{ and } 
\hat{g} = t \sin(\pi x) \sin(\pi y),\label{load-exp}
\end{align}
whereas for 2D we use equation \eqref{load4}. Initial conditions are set to zero in both 2D and 3D cases. 
The parameters used in the 3D model \eqref{diffusion-3dmodel}  assume typical values for copper plates\cite{sherief2005half}. See Table~\ref{Table-TED-constants}.

Let $T=10$, $\Delta t=1/8$, and consider the cells   $\hat{\Omega}_c=[5/64,6/64]\times [5/64,6/64]\times  [-d/2,d/2],$ and ${\Omega_c}=[5/64,6/64]\times [5/64,6/64]$.  At time $t=t_n$, we will use the following output quantities   
\begin{align*}
 U^n_{3\rm D}&:=\frac{1}{|\hat{\Omega}_c|} \int_{\hat{\Omega}_c}   \hat{U}_3  d\hat{\bx},\;\Theta^n_{3\rm D}:=\frac{1}{|\hat{\Omega}_c|} \int_{\Omega_c}z\hat{\Theta}^n\, d\hat{\bx},\text{ and } \\
 P^n_{3\rm D}&:= \frac{1}{|\hat{\Omega}_c|} \int_{\hat{\Omega}_c}z\hat{P}^n \, d\hat{\bx}\;\text{ for } \hat{\bx}=(x,y,z),\\
 U^n_{2\rm D}&:=\frac{1}{|{\Omega}_c|} \int_{{\Omega}_c}   {U}^n  d\bx,\; \Theta^n_{2\rm D}:=\frac{1}{|\Omega_c|} \int_{\Omega_c} \Theta^n \, d\bx,\text{ and } \\ P^n_{2\rm D}&:=\frac{1}{|{\Omega}_c|} \int_{{\Omega}_c}z{P}^n \, d\bx\; \text{ for } \bx=(x,y).
\end{align*}
The simulations in Figure~\ref{fig:TED} reveal that as the plate thickness \( d \) decreases (from upper to lower), the results of the 2D model approximate those of the 3D model and the solutions at the final time are plotted in Figure~\ref{fig:TEDsol}.
Furthermore, as expected, the computational efficiency is significantly improved: the 2D model requires approximately   138 (resp.   131) seconds, whereas the 3D model takes about  567 seconds (resp. 564) seconds for a plate width $d=0.5$ (resp $d=0.005)$.

\medskip 
\noindent \textbf{Thermo-poroelastic plate model  verification:}
 Motivated by Ref. \refcite{theodorakopoulos1994flexural}, in this experiment we choose a flat Berea sandstone   with  material parameters given in Table~\ref{Table-TPE-constants}, and  repeat the last experiment. 
 In 3D we consider   \eqref{poro-3dmodel} with the same load/source functions as in \eqref{load-exp}, and the transformation of source functions from 3D to 2D is given in \eqref{load-tpe}. The transformation of 3D to 2D parameters is given in Table~\ref{table-coeefients}. The quantities  
$(U^n_{3\rm D},\Theta^n_{3\rm D}, P^n_{3\rm D})$ and $(U^n_{2\rm D}, \Theta^n_{2\rm D}, P^n_{2\rm D})$ are defined similarly as in the last experiment. Moreover, we also consider $T=100$, $\Delta t=10/8$,  and the 
cells $\hat{\Omega}_c=[5/64,6/64]\times [5/64,6/64]\times  [-d/2,d/2],$ and ${\Omega_c}=[5/64,6/64]\times [5/64,6/64]$.
 The simulations in Figure~\ref{fig:TPE} reveal that as the plate thickness \( d \) decreases (from upper to bottom), the 2D model's results converge to those of the 3D model. Furthermore, the computational efficiency is significantly improved: the 2D model requires approximately  138 (resp.  85) seconds, whereas the 3D model takes about 566 (resp. 537) seconds for plate width $d=0.5$ (resp. $d=0.005$).

\subsection{Example 2: Convergence against smooth solutions}\label{subsec-num-smooth}
The theoretical results of  Section~\ref{sec:fully} are validated in this section by choosing a smooth manufactured  solution of \eqref{p2;model11}-\eqref{p2;model11}.  We consider the spatial domain $\Omega=(0,1)^2$ and time interval $[0,1]$. All model parameters are set to 1, except for $a_1 =35, a_2=40$, and $\gamma=1$ (and $\gamma=-1$) which are selected so that the condition \eqref{gamma_0} is satisfied and the results are robust with respect to $\gamma$. The   transverse load $f$, heat source $\phi$, and a total amount of mass source $g$ and the initial data $u^0$, $u^{*0}, \theta^0$ and $p^0$ are  chosen such that the exact solution of \eqref{eq:coupled} is given by
\begin{gather*}
    u(\bx,t)=\exp(5t)(x(x-1)y(y-1))^2, \\
\theta(\bx,t)=\exp(-t)\sin(\pi x)\sin(\pi y),\quad p(\bx,t)=\cos(t)\sin(\pi x)\sin(\pi y),\end{gather*}
and hence our theoretical regularity results with $\sigma=1$ (as well as the clamped boundary conditions) are satisfied. 

We construct a sequence of successively refined uniform triangular meshes $\mathcal{T}^i$ of $\Omega$ of size $h_i$ and split the time domain using the refined time step $\Delta t ={2}^{-3/2}h_i$. For each mesh refinement, we calculate errors as 
\begin{subequations}\label{num-norms}
\begin{align}
    &\norm{e_u}^{\ell^\infty}:=\underset{0 \le n \le N}{\max}{\norm{u^{n}-U^{n}}},\\
&\norm{{\nabla}e_u}^{{\ell}^\infty}\!\!:=\underset{0 \le n \le N}{\max}{\norm{\nabla(u^{n}-U^{n})}},\  \norm{\hat{e}_u}^{\ell^\infty}_h\!\!:=\underset{0 \le n \le N-1}{\max}{\norm{u^{n+1/2}-U^{n+1/2}}}_h,\label{norm1}\\
    &\norm{e_{\theta}}^{\ell^\infty}\!\!:=\underset{0 \le n \le N}{\max}{\norm{\theta^{n}-\Theta^{n}}},\  \norm{{\nabla}\hat{e}_{\theta}}^{\ell^2}\!\!:=\big(\Delta t\sum_{n=0}^{N-1}{\norm{\nabla(\theta^{n+1/2}-\Theta^{n+1/2})}}^2\big)^{1/2},
   \\
&\norm{e_{p}}^{\ell^\infty}\!\!:=\underset{0 \le n \le N}{\max}{\norm{p^{n}-P^{n}}},\ 
\norm{{\nabla}\hat{e}_{p}}^{\ell^2}\!\!:=\big(\Delta t\sum_{n=0}^{N-1}{\norm{\nabla(p^{n+1/2}-P^{n+1/2})}}^2\big)^{1/2}.\label{norm3}
\end{align}\end{subequations}
The experimental rates of convergence in space are computed as $\texttt{Rate} =$\, $\log(e_{i+1}/{e}_{i})[\log(h_{i+1}/{h}_i)]^{-1}$, where $e_{i}$ denotes a norm of the error on the mesh $\mathcal{T}^i$.  Then, by Theorem~\ref{error-estimates thm} and Corollary~\ref{corr}, the expected convergence rates are of order  $\mathcal{O}(h^\sigma)$ for $\norm{\hat{e}_u}^{\ell^\infty}_h,\norm{{\nabla}\hat{e}_{\theta}}^{\ell^2}, \norm{{\nabla}\hat{e}_{p}}^{\ell^2}$  and $\mathcal{O}(h^{2\sigma})$  for $\norm{e_u}^{\ell^\infty},\norm{{\nabla}e_u}^{\ell^\infty},\norm{{e}_{\theta}}^{\ell^\infty},\norm{{e}_{p}}^{\ell^\infty}$  norms defined in \eqref{num-norms}. Table~\ref{Table-smooth} shows the error history and convergence results for $u$, $\theta$ and $p$ and the numerical solution at the final time given in Figure~\ref{fig:TEDsmoothsol} for $\gamma=-1$. In all cases, the  numerical results are consistent with the expected theoretical results.

\begin{figure}[ht]
\begin{center}
\includegraphics[width=\linewidth, height=0.25\linewidth]{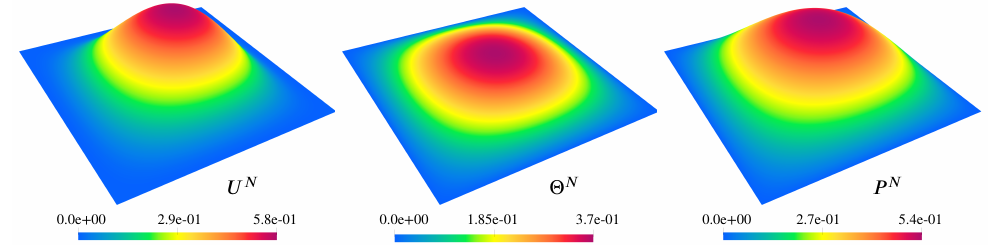}
\vspace{-0.5cm}
     \phantomcaption
  \caption{Example 2. Numerical solution shown  at final time $T$ for $\gamma=-1$.}  \label{fig:TEDsmoothsol}
\end{center}
\end{figure}
\begin{table}[hbt!]
\setlength{\tabcolsep}{4pt}
\centering
\small

\begin{tabular}{|c|A|B|B|B|B|B|B|}
\hline
\rule{0pt}{3ex}
 & \cellcolor{white}$h$ & \cellcolor{white}$\norm{e_u}^{\ell^\infty}$ & \cellcolor{white}\texttt{Rate} & \cellcolor{white}$\norm{{\nabla}e_u}^{\ell^\infty}$ & \cellcolor{white}\texttt{Rate} & \cellcolor{white}$\norm{\hat{e}_u}^{\ell^\infty}_h$ & \cellcolor{white}\texttt{Rate} \\
\hline
\multirow{14}{*}{\rotatebox{90}{\textbf{Displacement}}}
 & \multicolumn{7}{c|}{\rule{0pt}{2.5ex} $\gamma=-1$ (TED)} \\
\cline{2-8}
 & 0.3536 & 8.93e-02 & $\star$    & 4.39e-01 & $\star$  & 5.30e+00 & $\star$  \\
\cline{2-8}
 & 0.1768 & 2.99e-02 & 1.5801     & 1.48e-01 & 1.5628   & 3.36e+00 & 0.6564   \\
\cline{2-8}
 & 0.0884 & 8.14e-03 & 1.8757     & 4.13e-02 & 1.8461   & 1.84e+00 & 0.8684   \\
\cline{2-8}
 & 0.0442 & 2.07e-03 & 1.9753     & 1.06e-02 & 1.9588   & 9.54e-01 & 0.9496   \\
\cline{2-8}
 & 0.0221 & 5.12e-04 & 2.0159     & 2.64e-03 & 2.0060   & 4.84e-01 & 0.9802   \\
\cline{2-8}
 & 0.0110 & 1.07e-04 & 2.2608     & 5.77e-04 & 2.1952   & 2.43e-01 & 0.9913   \\
\cline{2-8}
 & \multicolumn{7}{c|}{\rule{0pt}{2.5ex} $\gamma=1$ (TPE)} \\
\cline{2-8}
 & 0.3536 & 8.92e-02 & $\star$    & 4.38e-01 & $\star$  & 5.30e+00 & $\star$  \\
\cline{2-8}
 & 0.1768 & 2.98e-02 & 1.5806     & 1.48e-01 & 1.5633   & 3.36e+00 & 0.6564   \\
\cline{2-8}
 & 0.0884 & 8.13e-03 & 1.8760     & 4.12e-02 & 1.8463   & 1.84e+00 & 0.8684   \\
\cline{2-8}
 & 0.0442 & 2.07e-03 & 1.9754     & 1.06e-02 & 1.9589   & 9.54e-01 & 0.9496   \\
\cline{2-8}
 & 0.0221 & 5.11e-04 & 2.0160     & 2.64e-03 & 2.0061   & 4.84e-01 & 0.9802   \\
\cline{2-8}
 & 0.0110 & 1.07e-04 & 2.2609     & 5.77e-04 & 2.1952   & 2.43e-01 & 0.9913   \\
\hline
\end{tabular}

\begin{tabular}{|c|A|C|C|C|C|D|D|D|D|}
\hline
\rule{0pt}{3ex}
 & \cellcolor{white}$h$ & \cellcolor{white}$\norm{{e}_{\theta}}^{\ell^\infty}$ & \cellcolor{white}\texttt{Rate} & \cellcolor{white}$\norm{{\nabla}\hat{e}_{\theta}}^{\ell^2}$ & \cellcolor{white}\texttt{Rate} & \cellcolor{white}$\norm{e_{p}}^{\ell^\infty}$ & \cellcolor{white}\texttt{Rate} & \cellcolor{white}$\norm{{\nabla}\hat{e}_{p}}^{\ell^2}$ & \cellcolor{white}\texttt{Rate} \\
\hline
\multirow{14}{*}{\rotatebox{90}{\textbf{Temp. \& Potential/Pressure}}}
 & \multicolumn{9}{c|}{\rule{0pt}{2.5ex} $\gamma=-1$ (TED)} \\
\cline{2-10}
 & 0.3536 & 7.91e-02 & $\star$  & 5.57e-01 & $\star$  & 7.91e-02 & $\star$  & 7.15e-01 & $\star$  \\
\cline{2-10}
 & 0.1768 & 2.11e-02 & 1.9038   & 2.86e-01 & 0.9630   & 2.11e-02 & 1.9038   & 3.69e-01 & 0.9549   \\
\cline{2-10}
 & 0.0884 & 5.38e-03 & 1.9745   & 1.43e-01 & 0.9941   & 5.38e-03 & 1.9745   & 1.86e-01 & 0.9903   \\
\cline{2-10}
 & 0.0442 & 1.35e-03 & 1.9935   & 7.17e-02 & 0.9995   & 1.35e-03 & 1.9935   & 9.30e-02 & 0.9981   \\
\cline{2-10}
 & 0.0221 & 3.38e-04 & 1.9984   & 3.58e-02 & 1.0000   & 3.38e-04 & 1.9984   & 4.65e-02 & 0.9996   \\
\cline{2-10}
 & 0.0110 & 8.45e-05 & 1.9996   & 1.79e-02 & 1.0000   & 8.45e-05 & 1.9996   & 2.32e-02 & 0.9999   \\
\cline{2-10}
 & \multicolumn{9}{c|}{\rule{0pt}{2.5ex} $\gamma=1$ (TPE)} \\
\cline{2-10}
 & 0.3536 & 7.91e-02 & $\star$  & 5.57e-01 & $\star$  & \cellcolor{aquamarine!10}7.91e-02 & \cellcolor{aquamarine!10}$\star$  & \cellcolor{aquamarine!10}5.57e-01 & \cellcolor{aquamarine!10}$\star$  \\
\cline{2-10}
 & 0.1768 & 2.11e-02 & 1.9038   & 2.86e-01 & 0.9624   & \cellcolor{aquamarine!10}2.11e-02 & \cellcolor{aquamarine!10}1.9038 & \cellcolor{aquamarine!10}2.85e-01 & \cellcolor{aquamarine!10}0.9649  \\
\cline{2-10}
 & 0.0884 & 5.38e-03 & 1.9745   & 1.43e-01 & 0.9943   & \cellcolor{aquamarine!10}5.38e-03 & \cellcolor{aquamarine!10}1.9745 & \cellcolor{aquamarine!10}1.43e-01 & \cellcolor{aquamarine!10}0.9940  \\
\cline{2-10}
 & 0.0442 & 1.35e-03 & 1.9935   & 7.17e-02 & 0.9996   & \cellcolor{aquamarine!10}1.35e-03 & \cellcolor{aquamarine!10}1.9935 & \cellcolor{aquamarine!10}7.17e-02 & \cellcolor{aquamarine!10}0.9993  \\
\cline{2-10}
 & 0.0221 & 3.38e-04 & 1.9984   & 3.59e-02 & 1.0000   & \cellcolor{aquamarine!10}3.38e-04 & \cellcolor{aquamarine!10}1.9984 & \cellcolor{aquamarine!10}3.58e-02 & \cellcolor{aquamarine!10}0.9999  \\
\cline{2-10}
 & 0.0110 & 8.45e-05 & 1.9996   & 1.79e-02 & 1.0000   & \cellcolor{aquamarine!10}8.45e-05 & \cellcolor{aquamarine!10}1.9996 & \cellcolor{aquamarine!10}1.79e-02 & \cellcolor{aquamarine!10}1.0000 \\
\hline
\end{tabular}

\caption{Example 2. Error decay with respect to mesh refinement, and convergence rates in the norms \eqref{norm1}-\eqref{norm3} with smooth exact solution. Errors and rates for displacement, temperature, and chemical potential (resp.\ pore pressure) are represented by black, red, and violet (resp.\ aquamarine) colors in the background.}
\label{Table-smooth}
\end{table}
\subsection{Example 3: Convergence for a non-convex domain}\label{subsec-num-lshape}
This example illustrates the  convergence of the proposed method even when the domain $\Omega$ is  non-convex, constituting a case where  $\sigma <1$.
Consider $\Omega=[-1,1]^2\setminus [-1,0]^2$, $T=1$, and  choose the load and source functions such that the triplet $(u,\theta,p)$ in polar coordinates is given by 
\begin{align*}
&u(r,\varphi,t)=t^2(r^2\sin^2(\varphi)-1)^2(r^2\cos^2(\varphi)^2-1)r^{1+\upsilon}G(r, \varphi+\pi/2),\\
& \theta(r,\varphi,t)= p(r,\varphi,t)=2t(r^2\sin^2(\varphi)-1)(r^2\cos^2(\varphi)-1)r^{2/3}\sin\big(2/3 (\varphi+\pi/2)
    \big), 
\end{align*}
where
\begin{align*}
G(r, \varphi) &= \Big(\frac{1}{\upsilon-1}\sin\big((\upsilon-1)\frac{3\pi}{2}\big)-\frac{1}{\upsilon+1}\sin\big((\upsilon+1)\frac{3\pi}{2}\big)\Big)\\
&\qquad \times \Big(\cos\big((\upsilon-1)\varphi\big)-\cos((\upsilon+1)\varphi)\Big)\\
&\quad-
\Big(\frac{1}{\upsilon-1}\sin\big((\upsilon-1)\varphi\big)-\frac{1}{\upsilon+1}\sin\big((\upsilon+1)\varphi\big)\Big)\\
&\quad \qquad \times \Big(\cos\big((\upsilon-1)3\pi/2\big)-\cos\big((\upsilon+1)3\pi/2\big)\Big).
\end{align*}
It is easy to check that $u \in C^\infty([0,T];H^{2+\sigma}(\Omega)\cap H^2_0(\Omega))$, $\theta, p \in C^\infty([0,T];H^{1+\sigma}(\Omega)\cap H^1_0(\Omega))$  with $\sigma =\upsilon =0.5444837$\cite{MR4796047,MR2784879}. Then, by Theorem~\ref{error-estimates thm} and Corollary~\ref{corr}, the expected convergence rates are of order  $\mathcal{O}(h^\sigma)$ for $\norm{\hat{e}_u}^{\ell^\infty}_h,\norm{{\nabla}\hat{e}_{\theta}}^{\ell^2}, \norm{{\nabla}\hat{e}_{p}}^{\ell^2}$  and $\mathcal{O}(h^{2\sigma})$  for $\norm{e_u}^{\ell^\infty},\norm{{\nabla}e_u}^{\ell^\infty},\norm{{e}_{\theta}}^{\ell^\infty},\norm{{e}_{p}}^{\ell^\infty}$  norms defined in \eqref{num-norms}. In this case, the model coefficients as well as the used norms are as in Example 2. Furthermore, we take $\Delta t=1/4$ and the experimental convergence rates are reported in Table~\ref{Table-lshape}, exhibiting the anticipated behavior. 
\begin{table}[hbt!]
\setlength{\tabcolsep}{4pt}
\centering
\small

\begin{tabular}{|c|A|B|B|B|B|B|B|}
\hline
\rule{0pt}{3ex}
 & \cellcolor{white}$h$ & \cellcolor{white}$\norm{e_u}^{\ell^\infty}$ & \cellcolor{white}\texttt{Rate} & \cellcolor{white}$\norm{{\nabla}e_u}^{\ell^\infty}$ & \cellcolor{white}\texttt{Rate} & \cellcolor{white}$\norm{\hat{e}_u}^{\ell^\infty}_h$ & \cellcolor{white}\texttt{Rate} \\
\hline
\multirow{16}{*}{\rotatebox{90}{\textbf{Displacement}}}
 & \multicolumn{7}{c|}{\rule{0pt}{2.5ex} $\gamma=-1$ (TED)} \\
\cline{2-8}
 & 0.7071 & 3.38e-01 & $\star$    & 1.25e+00 & $\star$  & 9.63e+00 & 0.0729   \\
\cline{2-8}
 & 0.3536 & 1.61e-01 & 1.0685     & 5.92e-01 & 1.0788   & 6.38e+00 & 0.5932   \\
\cline{2-8}
 & 0.1768 & 4.56e-02 & 1.8225     & 1.68e-01 & 1.8217   & 3.27e+00 & 0.9633   \\
\cline{2-8}
 & 0.0884 & 1.35e-02 & 1.7593     & 4.93e-02 & 1.7664   & 1.72e+00 & 0.9302   \\
\cline{2-8}
 & 0.0442 & 4.40e-03 & 1.6138     & 1.65e-02 & 1.5811   & 9.45e-01 & 0.8629   \\
\cline{2-8}
 & 0.0221 & 1.61e-03 & 1.4491     & 6.35e-03 & 1.3734   & 5.49e-01 & 0.7835   \\
\cline{2-8}
 & 0.0110 & 6.08e-04 & 1.4062     & 2.61e-03 & 1.2853   & 3.36e-01 & 0.7078   \\
\cline{2-8}
 & \multicolumn{7}{c|}{\rule{0pt}{2.5ex} $\gamma=1$ (TPE)} \\
\cline{2-8}
 & 0.7071 & 3.38e-01 & $\star$    & 1.25e+00 & $\star$  & 9.63e+00 & $\star$  \\
\cline{2-8}
 & 0.3536 & 1.61e-01 & 1.0688     & 5.92e-01 & 1.0791   & 6.38e+00 & 0.5932   \\
\cline{2-8}
 & 0.1768 & 4.56e-02 & 1.8228     & 1.67e-01 & 1.8219   & 3.27e+00 & 0.9634   \\
\cline{2-8}
 & 0.0884 & 1.35e-02 & 1.7594     & 4.92e-02 & 1.7665   & 1.72e+00 & 0.9302   \\
\cline{2-8}
 & 0.0442 & 4.40e-03 & 1.6138     & 1.65e-02 & 1.5811   & 9.45e-01 & 0.8629   \\
\cline{2-8}
 & 0.0221 & 1.61e-03 & 1.4491     & 6.35e-03 & 1.3734   & 5.49e-01 & 0.7835   \\
\cline{2-8}
 & 0.0110 & 6.08e-04 & 1.4065     & 2.61e-03 & 1.2854   & 3.36e-01 & 0.7078   \\
\hline
\end{tabular}

\begin{tabular}{|c|c|C|C|C|C|D|D|D|D|}
\hline
\rule{0pt}{3ex}
&$h$ & \cellcolor{white}$\norm{{e}_{\theta}}^{\ell^\infty}$ & \cellcolor{white}\texttt{Rate} & \cellcolor{white}$\norm{{\nabla}\hat{e}_{\theta}}^{\ell^2}$ & \cellcolor{white}\texttt{Rate} & \cellcolor{white}$\norm{e_{p}}^{\ell^\infty}$ & \cellcolor{white}\texttt{Rate} & \cellcolor{white}$\norm{{\nabla}\hat{e}_{p}}^{\ell^2}$ & \cellcolor{white}\texttt{Rate} \\
\hline
\multirow{16}{*}{\rotatebox{90}{\textbf{Temp. \& Potential/Pressure}}}
 & \multicolumn{9}{c|}{\rule{0pt}{2.5ex} $\gamma=-1$ (TED)} \\
\cline{2-10}
 & 0.7071  & 1.76e-01 & $\star$  & 1.04e+00 & $\star$  & 1.74e-01 & $\star$  & 1.04e+00 & $\star$  \\
\cline{2-10}
& 0.3536 & 6.05e-02 & 1.5427   & 5.04e-01 & 1.0466   & 5.78e-02 & 1.5927   & 5.02e-01 & 1.0509   \\
\cline{2-10}
 & 0.1768 & 1.74e-02 & 1.7988   & 2.56e-01 & 0.9799   & 1.65e-02 & 1.8069   & 2.55e-01 & 0.9751   \\
\cline{2-10}
 & 0.0884  & 5.27e-03 & 1.7226   & 1.38e-01 & 0.8914   & 5.02e-03 & 1.7184   & 1.38e-01 & 0.8902   \\
\cline{2-10}
 & 0.0442  & 1.76e-03 & 1.5774   & 7.72e-02 & 0.8357   & 1.69e-03 & 1.5736   & 7.72e-02 & 0.8354   \\
\cline{2-10}
 & 0.0221  & 6.65e-04 & 1.4069   & 4.45e-02 & 0.7933   & 6.34e-04 & 1.4118   & 4.45e-02 & 0.7933   \\
\cline{2-10}
 & 0.0110  & 3.03e-04 & 1.1340   & 2.63e-02 & 0.7581   & 2.83e-04 & 1.1615   & 2.63e-02 & 0.7581   \\
\cline{2-10}
 & \multicolumn{9}{c|}{\rule{0pt}{2.5ex} $\gamma=1$ (TPE)} \\
\cline{2-10}
  & 0.7071& 1.77e-01 & $\star$  & 1.04e+00 & $\star$  & \cellcolor{aquamarine!10}1.75e-01 & \cellcolor{aquamarine!10}$\star$  & \cellcolor{aquamarine!10}1.04e+00 & \cellcolor{aquamarine!10}$\star$  \\
\cline{2-10}
& 0.3536 & 6.16e-02 & 1.5218   & 5.05e-01 & 1.0447   & \cellcolor{aquamarine!10}5.89e-02 & \cellcolor{aquamarine!10}1.5715 & \cellcolor{aquamarine!10}5.03e-01 & \cellcolor{aquamarine!10}1.0491  \\
\cline{2-10}
 & 0.1768 & 1.78e-02 & 1.7958   & 2.56e-01 & 0.9820   & \cellcolor{aquamarine!10}1.69e-02 & \cellcolor{aquamarine!10}1.8036 & \cellcolor{aquamarine!10}2.55e-01 & \cellcolor{aquamarine!10}0.9771  \\
\cline{2-10}
& 0.0884  & 5.37e-03 & 1.7243   & 1.38e-01 & 0.8920   & \cellcolor{aquamarine!10}5.12e-03 & \cellcolor{aquamarine!10}1.7202 & \cellcolor{aquamarine!10}1.38e-01 & \cellcolor{aquamarine!10}0.8907  \\
\cline{2-10}
 & 0.0442  & 1.80e-03 & 1.5790   & 7.72e-02 & 0.8358   & \cellcolor{aquamarine!10}1.72e-03 & \cellcolor{aquamarine!10}1.5752 & \cellcolor{aquamarine!10}7.72e-02 & \cellcolor{aquamarine!10}0.8355  \\
\cline{2-10}
  & 0.0221   & 6.79e-04 & 1.4049   & 4.45e-02 & 0.7933   & \cellcolor{aquamarine!10}6.47e-04 & \cellcolor{aquamarine!10}1.4097 & \cellcolor{aquamarine!10}4.45e-02 & \cellcolor{aquamarine!10}0.7933  \\
\cline{2-10}
 &0.0110& 3.12e-04 & 1.1234   & 2.63e-02 & 0.7581   & \cellcolor{aquamarine!10}2.92e-04 & \cellcolor{aquamarine!10}1.1496 & \cellcolor{aquamarine!10}2.63e-02 & \cellcolor{aquamarine!10}0.7581 \\
\hline
\end{tabular}

\caption{Example 3. Error history in the norms from \eqref{norm1}--\eqref{norm3} for an $L$-shaped domain. Errors and rates for displacement, temperature, and chemical potential (resp.\ pore pressure) are represented by black, red, and violet (resp.\ aquamarine) colors in the background.}
\label{Table-lshape}
\end{table}

\medskip 
\subsection*{Acknowledgments} This work has been partially supported by  the J.C. Bose grant ANRF/JBG/2025/000209/HAA; by the Australian Research Council through the Future Fellowship   FT220100496 and  Discovery Project   DP22010316; by the Center of Advanced Study (CAS) at the Norwegian Academy of Science and Letters under the program \textit{Mathematical Challenges in Brain Mechanics}; and by the IITB-Monash Research Academy.

{\small 
\bibliographystyle{ws-m3as}
\bibliography{nry_references}
}


\appendix 
\section{Detailed proofs}\label{appendix}
\begin{proof}[\textbf{Proof of Theorem~\ref{existence;thm}}]
{The proof of existence is presented in Steps 1-4 and that of uniqueness using energy arguments in Step~5.} 

\medskip
\noindent \textit{{Step 1} (Construction of a sequence of approximate solutions).}
{It is well known that} there exists an orthogonal basis $\{w^1,w^2, \cdots \}$ (resp. $\{y^1,y^2, \cdots \}$) of $H_0^2(\Omega)$ (resp.  $H_0^1(\Omega)$)  and {this also forms an orthonormal basis of $L^2(\Omega)$\cite{MR1115235,MR990018}}. For a fixed integer $m$, we proceed to write 
\begin{equation}
 u^m(t):=\sum_{k=1}^{m}{\rm d}_m^k(t)w^k,\; \theta^m(t):=\sum_{k=1}^{m}\eta_m^k(t)y^k \text{ and } p^m(t):=\sum_{k=1}^{m}{\rm l}_m^k(t)y^k, \label{rep}
 \end{equation}
 where the coefficients ${\rm d}_m^k(t)$, $\eta_m^k(t)$ and ${\rm l}_m^k(t)$ are selected such that
\begin{subequations}\label{weak-mo}
\begin{align}
    &{\rm d}_m^k(0)=(u^0,w^k), \  {{\rm d}_m^k}'(0)+a_0\norm{\nabla w^k}^2 {{\rm d}_m^k}'(0)=(u^{*0},w^k)+a_0(\nabla u^{*0},\nabla w^k), \label{d0}\\
    &\eta_m^k(0)=(\theta^0,y^k), \ {\rm l}_m^{k}(0)=(p^0,y^k)  \label{l0}
\end{align} 
\end{subequations}
and
\begin{subequations}
\label{weak-m}
\begin{align}
& (u^m_{tt},w^k)+a_0 (\nabla u^m_{tt},\nabla w^k)+d_0(\nabla ^2 u^m, \nabla^2 w^k) -\alpha (\nabla \theta^m, \nabla w^k)-\beta (\nabla p^m, \nabla w^k)\nonumber\\
& =(f,w^k), \label{um} \\
&     a_1(\theta^m_{t}, y^k)-\gamma (p^m_{t}, y^k)+b_1(\theta^m,y^k) +c_1 (\nabla \theta^m,\nabla y^k) +\alpha (\nabla u^m_{t}, \nabla y^k)\nonumber \\
&     =(\phi,y^k),  \label{thetam} \\
&    a_2 (p^m_{t},y^k)-\gamma (\theta^m_{t},y^k)+\kappa (\nabla p^m,\nabla y^k)+\beta (\nabla u^m_{t},\nabla y^k)\nonumber \\&=(g,y^k),\label{pm}
     \end{align} 
     \end{subequations} 
     hold {for all $0 < t \le T$ and $k=1,2 \cdots ,m$}.
(Since \eqref{weak-m} forms a linear ODE system with initial conditions \eqref{weak-mo}, standard ODE theory\cite{evans} guarantees the existence of \rev{the} unique $C^2$ (resp. $C^1 $) functions 
  $({\rm d}_m^1(t),  {\rm d}_m^2(t), \cdots, {\rm d}_m^m(t))$ (resp. $(\eta_m^1(t), \eta_m^2(t), \cdots, \eta_m^m(t))$ and $({\rm l}_m^1(t),{\rm l}_m^2(t), \cdots, {\rm l}_m^m(t))$), that satisfy \eqref{weak-mo}-\eqref{weak-m} for $0\le t \le T. $)
  
\medskip
\noindent
\textit{ Step 2 (Derivation of a
priori bounds for approximate solutions)}. 
{We aim to take the limit 
$m \rightarrow \infty $ and hence shall} derive estimates that are uniform with respect to $m$.  
Multiply the equations \eqref{um}, \eqref{thetam}, and \eqref{pm} by ${{\rm d}_m^k}'(t),\eta_m^k(t)$, and ${\rm l}_m^k(t)$, respectively, {and sum} up the result    for $k=1,2,\ldots, m.$ {Then, the definitions in \eqref{rep} lead to}
\begin{align*}
 &   \frac{1}{2}\frac{\mathrm{d}}{\dt}\big(\norm{u^m_t}^2+a_0\norm{\nabla u_{t}^m}^2+d_0 \norm{\nabla^2 u^m}^2+a_1\norm{\theta^m}^2+a_2\norm{p^m}^2\big)\\
    &+b_1\norm{\theta^m}^2+c_1\norm{\nabla \theta^m}^2+\kappa\norm{\nabla p^m}^2-\gamma \frac{\mathrm{d}}{\dt}(\theta ^m,p^m)=(f,u^m_t)+(\phi, \theta^m)+(g,p^m).
\end{align*}
{An integration from $0$ to $t$ and} simple manipulations show
\begin{align}
&   \frac{1}{2}\big(\norm{u^m_t}^2+a_0\norm{\nabla u_{t}^m}^2+d_0 \norm{\nabla^2 u^m}^2+a_1\norm{\theta^m}^2+a_2\norm{p^m}^2\big) \nonumber \\
&\qquad +\int_0^t \big( b_1\norm{ \theta^m}^2+c_1\norm{\nabla \theta^m}^2+\kappa\norm{\nabla p^m}^2 \big) \ds \nonumber \\
   &
   \quad =\frac{1}{2}\big(\norm{u^m_t(0)}^2+a_0\norm{\nabla u_{t}^m(0)}^2+d_0 \norm{\nabla^2 u^m(0)}^2+a_1\norm{\theta^m(0)}^2+a_2\norm{p^m(0)}^2\big)\nonumber\\
   &\qquad
   +\gamma (\theta ^m,p^m) -\gamma (\theta^m(0)  ,p^m(0)) + \int_0^t \big((f,u^m_t)+(\phi, \theta^m)+(g,p^m)\big)\ds.\label{int}
\end{align}
\rev{Applying} Cauchy--Schwarz and Young's inequalities to the {last three terms in \eqref{int} yields}
\begin{align}
&    \gamma (\theta ^m,p^m) -\gamma (\theta^m(0)  ,p^m(0)) 
   + \int_0^t \big((f,u^m_t)+(\phi, \theta^m)+(g,p^m)\big)\ds \nonumber\\
   &\qquad 
   \le \frac{|\gamma|}{2\gamma_0}\big(\norm{\theta^m}^2+\norm{\theta^m(0)}^2\big)+\frac{|\gamma|\gamma_0}{2}\big(\norm{p^m}^2 +\norm{p^m(0)}^2\big)\nonumber\\
   & \quad\qquad 
   + \frac{1}{2} \int_0^t \big(\norm{f}^2+\norm{\phi}^2+\norm{g}^2+\norm{u^m_t}^2+\norm{\theta^m}^2+\norm{p^m}^2 \big)\ds,\label{gamma}
\end{align}
where $\gamma_0$ is defined in \eqref{gamma_0}.
Next, substitute \eqref{gamma} in \eqref{int}, to obtain
\begin{align}
   &\norm{u^m_t}^2+a_0\norm{\nabla u_{t}^m}^2+d_0 \norm{\nabla^2 u^m}^2+(a_1-|\gamma|/\gamma_0)\norm{\theta^m}^2\nonumber \\
   & \quad+(a_2-|\gamma|\gamma_0)\norm{p^m}^2
   +2\int_0^t \big( b_1\norm{ \theta^m}^2+c_1\norm{\nabla \theta^m}^2+\kappa\norm{\nabla p^m}^2 \big) \ds \nonumber \\
   & 
   \le\norm{u^m_t(0)}^2+a_0\norm{\nabla u_{t}^m(0)}^2+d_0 \norm{\nabla^2 u^m(0)}^2+(a_1+|\gamma|/\gamma_0)\norm{\theta^m(0)}^2\nonumber \\
   & \quad +(a_2+|\gamma|\gamma_0)\norm{p^m(0)}^2
+\norm{f}_{L^2(0,T;L^2(\Omega))}^2+\norm{\phi}^2_{L^2(0,T;L^2(\Omega))}+\norm{g}^2_{L^2(0,T;L^2(\Omega))}\nonumber \\
   &  \quad+ \int_0^t \big(\norm{u^m_t}^2+\norm{\theta^m}^2+\norm{p^m}^2 \big)\ds.\label{int1}
\end{align}

We next utilize \eqref{d0}-\eqref{l0} to show
\begin{align}
&\norm{u^m_t(0)}^2+a_0\norm{\nabla u_{t}^m(0)}^2+d_0 \norm{\nabla^2 u^m(0)}^2\nonumber\\
&\quad+(a_1+|\gamma|/\gamma_0)\norm{\theta^m(0)}^2+(a_2+|\gamma|\gamma_0)\norm{p^m(0)}^2\nonumber\\
&
\lesssim 
\norm{u^{*0}}^2+a_0\norm{u^{*0}}^2_{H^1(\Omega)}+d_0\norm{u^0}^2_{H^2(\Omega)}\nonumber\\
&\quad+(a_1+|\gamma|/\gamma_0)\norm{\theta^0}^2+(a_2+|\gamma|\gamma_0)\norm{p^0}^2.\label{intt2}
\end{align}
A combination of \eqref{int1}-\eqref{intt2} and an  application of Lemma~\ref{P1 gronwall} lead to the bound
\begin{align} \label{eq:aprioribounds}
&2E(u^m, \theta^m, p^m; t)=\norm{u^m_t}^2+a_0\norm{\nabla u_{t}^m}^2+d_0 \norm{\nabla^2 u^m}^2+(a_1-|\gamma|/\gamma_0)\norm{\theta^m}^2\nonumber \\
   &\quad \qquad \quad \qquad
  +(a_2-|\gamma|\gamma_0)\norm{p^m}^2 +2\int_0^t \big( b_1\norm{ \theta^m}^2+c_1\norm{\nabla \theta^m}^2+\kappa\norm{\nabla p^m}^2 \big) \ds \nonumber \\
   & \lesssim 
\norm{u^{*0}}^2+a_0\norm{u^{*0}}^2_{H^1(\Omega)}+d_0\norm{u^0}^2_{H^2(\Omega)}+(a_1+|\gamma|/\gamma_0)\norm{\theta^0}^2+(a_2+|\gamma|\gamma_0)\norm{p^0}^2\nonumber\\
&\qquad\qquad\qquad+\norm{f}_{L^2(0,T;L^2(\Omega))}^2+\norm{\phi}^2_{L^2(0,T;L^2(\Omega))}+\norm{g}^2_{L^2(0,T;L^2(\Omega))}.
\end{align}

Now, fix any $v \in H^2_0(\Omega)$ and $\psi,q \in H^1_0(\Omega)$ with $ \norm{v}_{H^2_0(\Omega)} \le 1,\; \norm{\psi}_{H^1_0(\Omega)} \le 1,$ and $ \norm{q}_{H^1_0(\Omega)} \le 1$. Write $v=v_1+v_2,$ $\psi=\psi_1+\psi_2$ and $q=q_1+q_2$,  where $v_1 \in$ span $\{w^k\}_{k=1}^m$, and both   $\psi_1,q_1 \in$ span $\{y^k\}_{k=1}^m$ with $(v_2, w^k)=(\nabla v_2, \nabla w^k)=(\psi_2,y^k)=(q_2,y^k)=0$ $(k =1,2,\cdots,m).$ Note $\norm{v_1}_{H^2_0(\Omega)} \le 1,\; \norm{\psi_1}_{H^1_0(\Omega)}\le 1$, and $\norm{q_1}_{H^1_0(\Omega)}\le 1.$ Then \eqref{rep} and \eqref{um} imply that 
\begin{align}
    \langle u^m_{tt}, v \rangle&+a_0\langle \nabla u^m_{tt},\nabla v\rangle=(u^m_{tt}, v)+a_0(\nabla u^m_{tt},\nabla v)=(u^m_{tt}, v_1)+a_0(\nabla u^m_{tt},\nabla v_1)\nonumber\\
    &\qquad=(f,v_1)-d_0(\nabla^2 u^m, \nabla^2 v_1)+\alpha (\nabla \theta^m, \nabla v_1)+\beta (\nabla p^m, \nabla v_1).\label{duality}
\end{align}
This and a Cauchy--Schwarz inequality reveal
\begin{align*}
\norm{u_{tt}^m}_{H^{-1}(\Omega)}&\lesssim| \langle u^m_{tt}, v\rangle+ 
 \langle \nabla u^m_{tt},\nabla v\rangle| \lesssim  | \langle u^m_{tt}, v\rangle+ 
 a_0\langle \nabla u^m_{tt},\nabla v\rangle| \\
 &\lesssim \norm{f}+d_0\norm{\nabla^2 u^m}+\alpha \norm{\nabla \theta^m}+ \beta \norm{\nabla p^m}.
\end{align*}
An integration from $0$ to $T$ and the bounds from \eqref{eq:aprioribounds} allow us to assert that 
\begin{align}
&\int_0^T\norm{u_{tt}^m}_{H^{-1}(\Omega)} \: \dt \nonumber\\
&\lesssim \norm{u^{*0}}^2+a_0\norm{u^{*0}}^2_{H^1(\Omega)}+d_0\norm{u^0}^2_{H^2(\Omega)} +(a_1+|\gamma|/\gamma_0)\norm{\theta^0}^2+(a_2+|\gamma|\gamma_0)\norm{p^0}^2\nonumber\\
& \quad  +\norm{f}_{L^2(0,T;L^2(\Omega))}^2+\norm{\phi}^2_{L^2(0,T;L^2(\Omega))}+\norm{g}^2_{L^2(0,T;L^2(\Omega))}.\label{uttm-h-2}
\end{align}
On the other hand, the combination of \eqref{rep} and \eqref{thetam}–\eqref{pm} with similar arguments as in  \eqref{duality}, yields 
\begin{align*}
    a_1 \langle \theta^m_{t}, \psi\rangle-\gamma \langle p^m_{t}, \psi \rangle&=a_1 ( \theta^m_{t}, \psi_1)-\gamma ( p^m_{t}, \psi_1 )\\
    & =(\phi,\psi_1)-b_1(\theta^m,\psi_1) -c_1 (\nabla \theta^m,\nabla \psi_1 ) -\alpha (\nabla u^m_{t}, \nabla y^k),   \\
    a_2 \langle p^m_{t},q \rangle -\gamma \langle \theta^m_t,q\rangle & =a_2 (p^m_{t},q_1 ) -\gamma ( \theta^m_t,q_1)\\
    & =(g,q_1)-\kappa (\nabla p^m,\nabla q_1)-\beta (\nabla u^m_{t},\nabla q_1).
\end{align*}
Applying Cauchy--Schwarz inequality again (shifting the second terms from left to right-hand side in both inequalities and using $|\gamma \langle p^m_{t}, \psi \rangle| \le |\gamma| \norm{p^m_t}_{H^{-1}(\Omega)}$ and $|\gamma \langle \theta^m_t,q\rangle| \le |\gamma| \norm{\theta^m_t}_{H^{-1}(\Omega)} $\rev{)}, we readily get 
\begin{align*}
a_1\norm{\theta^m_t}_{H^{-1}(\Omega)}&\le  |\gamma| \norm{p^m_t}_{H^{-1}(\Omega)}+C \Big( \norm{\phi}+b_1\norm{\theta^m}+c_1\norm{\nabla\theta^m} +\alpha \norm{\nabla u^m_t}\Big),\\
a_2\norm{p^m_t}_{H^{-1}(\Omega)} &\le  |\gamma| \norm{\theta^m_t}_{H^{-1}(\Omega)}+C \Big( \norm{g}+\kappa\norm{\nabla p^m} +\beta \norm{\nabla u^m_t}\Big).
\end{align*}
Next, we multiply the first equation above by ${\gamma_0}^{{1}/{2}}$, the second by ${\gamma_0}^{-{1}/{2}}$, and add the two inequalities. Again, applying integration from $0$ to $T$ and the bounds from \eqref{eq:aprioribounds},  leads to
\begin{align}
&  {\gamma_0}^{{1}/{2}}(a_1-|\gamma|/\gamma_0)\int_0^T\norm{\theta^m_t}_{H^{-1}(\Omega)} \: \dt +\gamma_0^{-{1}/{2}}(a_2-|\gamma|\gamma_0)\int_0^T\norm{p^m_t}_{H^{-1}(\Omega)  } \dt \nonumber\\
&\; \lesssim  \norm{u^{*0}}^2+a_0\norm{u^{*0}}^2_{H^1(\Omega)}+d_0\norm{u^0}^2_{H^2(\Omega)}+(a_1+\frac{|\gamma|}{\gamma_0})\norm{\theta^0}^2+(a_2+|\gamma|\gamma_0)\norm{p^0}^2\nonumber\\
&\qquad\qquad\qquad\qquad+\norm{f}_{L^2(0,T;L^2(\Omega))}^2+\norm{\phi}^2_{L^2(0,T;L^2(\Omega))}+\norm{g}^2_{L^2(0,T;L^2(\Omega))}.\label{t-p-h-1}
\end{align}

\medskip
\noindent
\textit{{Step 3} (Existence of a limit for a subsequences).}
The estimates in \eqref{eq:aprioribounds} indicate that $ \{u^{m}\}_{m=1}^\infty $ and $ \{u_t^{m}\}_{m=1}^\infty $ are bounded in the spaces $ L^\infty(0,T;H^2_0(\Omega)) $ and $ L^\infty([0,T]
;H^1_0(\Omega)) $, respectively, and both $ \{\theta^{m}\}_{m=1}^\infty $ and $ \{p^{m}\}_{m=1}^\infty $ are bounded in $ L^\infty(0,T;L^2(\Omega)) $ as well as in $ L^2(0,T;H^1_0(\Omega)) $. Moreover, the estimates in \eqref{uttm-h-2} and \eqref{t-p-h-1} reveal that  $ \{u_{tt}^{m}\}_{m=1}^\infty $, $ \{\theta_t^{m}\}_{m=1}^\infty $ and $ \{p_t^{m}\}_{m=1}^\infty $ are bounded in  $ L^2(0,T;H^{-1}(\Omega)) $.  Consequently, there exist subsequences $ \{u^{m}\}_{m=1}^\infty $, $ \{\theta^{m}\}_{m=1}^\infty $, and $ \{p^{m}\}_{m=1}^\infty $ (where relabeling is used), and some $ u \in L^\infty(0,T;H^2_0(\Omega)) $ with $ u_t \in L^\infty(0,T;H^1_0(\Omega)) $, and $u_{tt} \in L^2(0,T;H^{-1}(\Omega)) $, $ \theta, p \in L^\infty(0,T;L^2(\Omega)) \cap L^2(0,T;H^1_0(\Omega)) $, and $ \theta_t, p_t \in L^2(0,T;H^{-1}(\Omega))$ such that
\begin{subequations}\label{m-inf}
\begin{align}
    ( u^{m},u_t^{m},\theta^{m},p^{m})  &\xrightarrow{\text{weak*}}(u,u_t,\,\theta,p) \quad  \text{in}\ L^\infty\big(0,T;H_0^2(\Omega) 
 \times H^1_0(\Omega) \times (L^2(\Omega) )^2\big), \label{c1}\\
   (\theta^{m}, p^{m})  &\xrightarrow{\text{weak}} (\theta,p) \quad \qquad \; \;\;  \text{in} \ L^2 \big(0,T; (H_0^1(\Omega))^2\big),\label{c2}\\
{(u_{tt}^{m},\theta^m_t ,p^{m}_t) }&\xrightarrow{\text{weak}} (u_{tt},\theta_t,p_t)  \quad \;\; \;  \text{in} {\ L^2 \big(0,T; (H^{-1}(\Omega))^3\big)}.
\end{align}\end{subequations}
\textit{{Step 4} (Limit is a weak solution).} Now we show that $(u,\theta,p)$ satisfies \eqref{u}-\eqref{p}.
For this, we introduce 
$\widehat{\rm d}_{j_0}^k(t)\in C^2[0,T]$, $\widehat{\eta}_{j_{0}}^k(t)$ and $\widehat{\rm l}_{j_{0}}^k(t) \in C^1[0,T]$  such that $\widehat{\rm d}_{j_{0}}^k(T)=\widehat{\rm d}_{j_{0}t}^k(T)=\widehat{\eta}_{j_{0}}^k(T)=\widehat{\rm l}_{j_{0}}^k(T)=0$,  and  define 
\begin{equation}
\widehat{u}^{j_{0}}:=\sum_k^{j_{0}} \widehat{\rm d}_{j_{0}}^k(t) w^k, \; \widehat{\theta}^{j_{0}} :=\sum_k^{j_{0}} \widehat{\eta}_{j_{0}}^k(t) y^k, \text{ and }\widehat{p}^{j_{0}}:=\sum_k^{j_{0}}\widehat{\rm l}_{j_{0}}^k(t)y^k. \label{dense}
\end{equation}
Multiply \eqref{um} by $\widehat{\rm d}_{j_0}^k(t)$, \eqref{thetam} by $\widehat{\eta}_{j_{0}}^k(t)$, and \eqref{pm} by $\widehat{\rm l}_{j_{0}}^k(t)$, add the resulting    equations for $k=1,2,\dots, j_0$, and integrate by parts in $t$ from $0$ to $T$,   to obtain
   \begin{align*}
 &-\!\int_0^T\!\!(u^m_{t},\widehat{u}_t^{j_{0}})\dt-a_0 \!\int_0^T\!\!(\nabla u^m_{t},\nabla \widehat{u}_t^{j_{0}})\dt+d_0\! \int_0^T\!\! (\nabla ^2 u^m, \nabla^2 \widehat{u}^{j_{0}})\dt \\
 &\quad  - \alpha\! \int_0^T \!\!(\nabla \theta^m, \nabla \widehat{u}^{j_{0}})\dt-\beta \!\int_0^T \!\!(\nabla p^m, \nabla \widehat{u}^{j_{0}})\dt\\
     & \qquad
     =\int_0^T(f,\widehat{u}^{j_{0}}) \dt+(u^m_t({0}),\widehat{u}^{j_{0}}(0))+a_0(\nabla u^m_t({0}),\nabla \widehat{u}^{j_{0}}(0)),\\
     &-a_1\int_0^T\!\!(\theta^m, \widehat{\theta}_t^{j_{0}})\dt+ \gamma \int_0^T\!\!(p^m, \widehat{\theta}^{j_{0}}_t)\dt+ b_1\int_0^T\!\!(\theta^m,\widehat{\theta}^{j_{0}})\dt +c_1 \int_0^T\!\!(\nabla \theta^m,\nabla \widehat{\theta}^{j_{0}})\dt\\
     &\quad +\alpha \int_0^T(\nabla u^m_t, \nabla \widehat{\theta}^{j_{0}})\dt \\
     & \qquad
=\int_0^T(\phi,\widehat{\theta}^{j_{0}})\dt+a_1(\theta^m(0), \widehat{\theta}_t^{j_{0}}(0))-\gamma(p^m(0), \widehat{\theta}^{j_{0}}(0)),  \\
    &-a_2 \int_0^T(p^m,\widehat{p}_t^{j_{0}})\dt+\gamma \int_0^T(\theta^m,\widehat{p}_t^{j_{0}})\dt+\kappa \int_0^T(\nabla p^m,\nabla \widehat{p}^{j_{0}})\dt\\
    & \quad +\beta \int_0^T (\nabla u^m_t,\nabla \widehat{p}^{j_{0}})\dt\\
    &\qquad
    =\int_0^T(g,\widehat{p}^{j_{0}})\dt+a_2(p^m(0), \widehat{p}_t^{j_{0}}(0))-\gamma(\theta^m(0), \widehat{p}^{j_{0}}(0)),
    \end{align*}
where we have utilized  that $\widehat{\rm d}_{j_0}^k(t)$, $\widehat{\eta}_{j_{0}}^k(t)$ and $\widehat{\rm l}_{j_{0}}^k(t)$  are such that $\widehat{u}^{j_{0}}(T)=0$, $\widehat{\theta}^{j_{0}}(T)=0$ and $\widehat{p}^{j_{0}}(T)=0$. 

Then we invoke \eqref{c1} and \eqref{c2} to pass to the limit as $m \to \infty$ in the final system of equations. Also, since the functions  in \eqref{dense} are dense in $C^2([0,T];H^2_0(\Omega))$, $C^1([0,T];H^1_0(\Omega))$, and $C^1([0,T];H^1_0(\Omega))$, respectively; we can observe that $u, \theta, p$ satisfy \eqref{u}-\eqref{p}. Moreover,  the regularity  stated in \eqref{new-reg-weak} is guaranteed by \eqref{m-inf} and   Theorem 3, p. 287 in Ref. \refcite{evans}.
   To ensure that  \eqref{init-conds} holds we can follow verbatim   p. 384 of Ref. \refcite{evans}, and omit further details. 
   We therefore establish the existence of weak solution to   \eqref{p2;model11}-\eqref{p2;model33}, and the bounds \eqref{reg-ut} are a consequence of passing to the limit as $m$ tends to infinity in \eqref{eq:aprioribounds}, and utilizing \eqref{m-inf}.
   
\medskip
\noindent \textit{Step 5 (Uniqueness)}.
The uniqueness of solution to the coupled system \eqref{u}-\eqref{p} (under the data regularity provided in the first part of the proof above) was still an open problem as in Refs. \refcite{vishnevskii2020evolutionary,MR3471659},  and it is not trivial.  However, the uniqueness of solution to the uncoupled system -- under the same data assumptions -- can be proved using   Section 7.1.2-Theorem 4 and Section 7.2.1-Theorem 4 in Ref. \refcite{evans}. 
 To this end, we follow the approach in~Ref. \refcite{Rivera} to construct mollified test functions that possess sufficient regularity and are compactly supported in the time interval \([0, T]\).

Let us define $\rho_\varepsilon (s)=\varepsilon^{-1} \rho(\varepsilon^{-1}s)$ for $\varepsilon >0$, where $\rho(t)$ is a function in $C_0^{\infty}(\mathbb{R})$ satisfying
    \[\rho \ge 0, \; \text{supp} \rho \subset [-2,-1], \text{ and } \int_{-\infty}^\infty \rho(s)\ds=1.\]
Let us take $v \in H^{2}_0(\Omega)$, $\psi \in H^1_0(\Omega)$, and $q \in H^1_0(\Omega)$ and further denote 
    \[\tilde{v}^t(s,\bx)= \rho_\varepsilon(t-s)v(\bx),\; \tilde{\psi}^t(s,\bx)=\rho_\varepsilon(t-s) \psi(\bx), \; \tilde{q}^t(s,\bx)=\rho_\varepsilon(t-s) q(\bx),\] for $t \in [0,T]$. 
    Clearly, $\tilde{v}^t \in C^\infty([0,T];H^2_0(\Omega)), \, \tilde{\psi}^t\in C^\infty([0,T];H^1_0(\Omega)), \, \tilde{q}^t\in C^\infty([0,T];H^1_0(\Omega))$ for   $0\le t \le T.$ Substituting  $\tilde{v}^t,\;\tilde{\psi}^t$,\ and $ \tilde{q}^t$ in \eqref{u}-\eqref{p}, and noting that $\rho_\varepsilon(t)=(d/dt)\rho_\varepsilon(t)=0$ for $0 \le t \le T$, yields  
    \begin{subequations}
    \begin{align}
        (u_{\varepsilon tt}(t,\bullet), v)+a_0(\nabla u_{\varepsilon tt}(t,\bullet), \nabla v)+d_0 (\nabla^2 u_{\varepsilon }(t,\bullet), \nabla^2 v)&\nonumber \\
        -\alpha (\nabla \theta_\varepsilon(t,\bullet), \nabla v)-\beta (\nabla p_\varepsilon(t,\bullet), \nabla v)&=0,\label{uep}\\
        a_1(\theta_{\varepsilon t }(t,\bullet),\psi)-\gamma(p_{\varepsilon t}(t,\bullet),\psi)+b_1(\theta_{\varepsilon}(t,\bullet),\psi)&\nonumber \\
        +c_1(\nabla \theta_\varepsilon(t,\bullet),\nabla \psi)+\alpha (\nabla u_{\varepsilon t}(t,\bullet),\nabla\psi)&=0,\label{thetaep}\\
        a_2(p_{\varepsilon t }(t,\bullet),q)-\gamma(\theta_{\varepsilon t}(t,\bullet),q)+\kappa(\nabla p_\varepsilon(t,\bullet),\nabla q)&\nonumber \\
        +\beta (\nabla u_{\varepsilon t}(t,\bullet),\nabla q)&=0, \label{pep}
        \end{align}\end{subequations}
        for any $ 0 \le t \le T$, $v \in H^2_0(\Omega)$, $\psi \in H^1_0(\Omega)$, and $q \in H^1_0(\Omega)$, where 
        \begin{align*}u_{\varepsilon }(t,\bx) &= \int_{-\infty}^\infty \rho_\varepsilon(t-s)u(s,\bx)\ds,\; \theta_\varepsilon(t,\bx)  = \int_{-\infty}^\infty \rho_\varepsilon(t-s)\theta(s,\bx)\ds, \text{ and }\\
        p_\varepsilon (t,\bx)&=\int_{-\infty}^\infty \rho_\varepsilon(t-s)p(s,\bx)\ds.\end{align*}
      From \eqref{c1}-\eqref{c2}, it follows that $$u_{\varepsilon t } \in C^\infty ([0,T]; H^2_0(\Omega)), \;\theta_{\varepsilon } \in C^\infty ([0,T]; H^1_0(\Omega)),  \text{ and }p_{\varepsilon } \in C^\infty ([0,T]; H^1_0(\Omega)).$$  
      For all $t \in [0,T]$, we choose the test functions in \eqref{uep}-\eqref{pep}  as $v =u_{\varepsilon t } (t,\bullet), 
    \; \psi=\theta_{\varepsilon } (t,\bullet)$, and $p=p_{\varepsilon }(t,\bullet)$, and integrate the resulting equations with respect to $t$, and follow  arguments analogous to Step 2 to show 
\[ 
0 \le E(u_{\varepsilon},\theta_\varepsilon,p_\varepsilon;t) \le 0.\]
(The system \eqref{uep}-\eqref{pep} is similar to \eqref{um}-\eqref{pm} with $f=\phi= g=0$ and zero initial conditions).

\noindent Now take the limit as $ \varepsilon \rightarrow 0$ to obtain  
\[ 
    E(u,\theta,p;t) = 0,
\]
which shows, directly from  \eqref{energy}, that $u=0, \; \theta=0$, and $p=0$. This completes the proof. 
\end{proof}

\noindent{\bf Proof of Theorem~\ref{thm;regularity}}.

 \noindent Let $(\tilde{u},\tilde{\theta},\tilde{p})$ be the solution to \eqref{p2;model11}-\eqref{p2;model33} (its existence is guaranteed by Theorem~\ref{weak-sol-thm}) satisfying \eqref{u-weak}-\eqref{p-weak}, and consider the initial conditions 
\begin{align*}
&\tilde{u}(0,\bx)=u^{*0}(\bx)  \in H^3(\Omega) \cap H^2_0(\Omega) ,\; \tilde{u}_t(0,\bx)= u_{tt}(0) \in H^2(\Omega) \cap H^1_0(\Omega),\\
&\tilde{\theta}(0,\bx)=\theta_t(0)\in  H^2(\Omega) \cap H^1_0(\Omega),\;  \tilde{p}(0,\bx) = p_t(0) \in H^2(\Omega) \cap H^1_0(\Omega).
\end{align*}
We then write 
\begin{align*}
    u(t,\bx)&=u^0+\int_0^t\tilde{u}(s,\bx)\ds , \;   \theta(t,\bx)=\theta^0+\int_0^t\tilde{\theta}(s,\bx)\ds, \text{ and }  \\p(t,\bx)&=p^0+\int_0^t\tilde{p}(s,\bx)\ds ,
\end{align*}
and can readily employ similar arguments as in  Theorem~\ref{weak-sol-thm}  and \eqref{ellip-reg}-\eqref{ell-reg-sol} to obtain the bounds \eqref{reg-uttt}. Part (b) follows by analogous arguments. 
\end{document}